\documentclass[]{amsart}

\usepackage{ae,lmodern} 
\usepackage[utf8]{inputenc} 
\usepackage[T1]{fontenc} 
\usepackage[T1]{fontenc}
\usepackage[english]{babel} 

\usepackage{amsthm}
\usepackage{amsmath}
\usepackage{amssymb}
\usepackage{geometry}
\usepackage{dsfont}
\usepackage{MnSymbol}
\usepackage{comment}
\usepackage[hidelinks]{hyperref}
\usepackage{stmaryrd}
\usepackage{multirow}
\usepackage{enumitem}
\usepackage{pifont}
\usepackage{etoolbox}

\usepackage{pinlabel}
\usepackage{float}
\usepackage{caption}
\usepackage{subcaption}
\usepackage{graphicx}
\usepackage{xcolor}
\usepackage{psfrag}
\usepackage{wrapfig}

\usepackage[
backend=bibtex,        
sorting=nyt,          
citestyle=alphabetic,
bibstyle=alphabetic,  
isbn=false, 
url=false, 
doi=false, 
eprint=false,
]{biblatex}
\AtEveryBibitem{%
	\clearfield{note}%
	}

\usepackage{indentfirst}

\let\oldtocsection=\tocsection
\let\oldtocsubsection=\tocsubsection
\let\oldtocsubsubsection=\tocsubsubsection

\renewcommand{\tocsection}[2]{\hspace{0em}\oldtocsection{#1}{#2}}
\renewcommand{\tocsubsection}[2]{\hspace{1em}\oldtocsubsection{#1}{#2}}
\renewcommand{\tocsubsubsection}[2]{\hspace{2em}\oldtocsubsubsection{#1}{#2}}

\newtheorem{Theorem}{Theorem}[section]
\newtheorem{Corollary}[Theorem]{Corollary}
\newtheorem{Lemma}[Theorem]{Lemma}
\newtheorem{Proposition}[Theorem]{Proposition}
\newtheorem*{Question*}{Question}

\theoremstyle{remark}
\newtheorem{Remark}[Theorem]{Remark}

\theoremstyle{definition}
\newtheorem{Definition}{Definition}[section]

\newcommand{\p}[2][]{%
	\mathsf{p}_{#1}({#2})}

\newcommand{\s}[2][]{%
	\mathsf{s}_{#1}({#2})}

\newcommand{\N}{\mathbb{N}}
\newcommand{\Z}{\mathbb{Z}}
\newcommand{\Q}{\mathbb{Q}}
\newcommand{\F}{\mathbb{F}}
\newcommand{\R}{\mathbb{R}}
\newcommand{\C}{\mathbb{C}}

\newcommand{\HH}{\mathbb{H}}
\newcommand{\eps}{\varepsilon}

\newcommand{\cF}{\mathcal{F}}
\newcommand{\cI}{\mathcal{I}}
\newcommand{\cJ}{\mathcal{J}}
\newcommand{\cK}{\mathcal{K}}

\newcommand{\lf}{\lfloor}
\newcommand{\rf}{\rfloor}
\newcommand{\lc}{\lceil}
\newcommand{\rc}{\rceil}
\newcommand{\la}{\langle}
\newcommand{\ra}{\rangle}
\newcommand{\inv}[1]{\overline{#1}}
\renewcommand{\t}[1]{\widetilde{#1}}
\newcommand{\vphi}{\varphi}

\newcommand{\1}{\ding{172}}
\newcommand{\2}{\ding{173}}
\newcommand{\3}{\ding{174}}
\newcommand{\4}{\ding{175}}
\newcommand{\5}{\ding{176}}

\newcommand{\sph}{\mathrm{S}_{0,4}}
\newcommand{\simple}{\mathcal{S}(\sph)}
\newcommand{\groupe}{\pi_1(\sph)}
\newcommand{\slope}{\mathrm{Slope}}
\newcommand{\len}{\mathrm{Length}}
\newcommand{\mcg}{\mathrm{MCG}(\sph)}
\newcommand{\aut}{\mathrm{Aut}(\F_3)}
\newcommand{\out}{\mathrm{Out}(\F_3)}
\newcommand{\axis}{\mathrm{Axis}}
\newcommand{\isom}{\mathrm{Isom}(X)}

\definecolor{red}{RGB}{255,0,0}
\definecolor{blue}{RGB}{0,50,162}
\definecolor{green}{RGB}{0,132,0}
\definecolor{purple}{RGB}{80,20,122}
\definecolor{pink}{RGB}{187,0,165}
\definecolor{black}{RGB}{0,0,0}

\begin{document}

	\title[Simple-stable and Bowditch representations of the four-punctured sphere group]{ Simple-stable and Bowditch representations of the four-punctured sphere group in Gromov-hyperbolic spaces.}
	\author{Suzanne \textsc{Schlich}}
	
	\renewcommand{\labelitemi}{$\bullet$}

	\maketitle
	\begin{abstract} In this paper, we study representations from the four-punctured sphere group into isometry groups of Gromov-hyperbolic spaces. We prove that the set of simple-stable representations (in analogy with Minsky's notion of primitive-stability) and the set of Bowditch representations (inspired by Bowditch's work on the free group of rank two, generalized by Tan, Wong and Zhang) are equal. Along the way, we study the combinatorics of simple closed curves on the four-punctured sphere and prove a result which quantifies the redundancy of subwords of certain given lengths within simple words.
	\end{abstract}
	
	\tableofcontents 
	
	\newpage
	\section{Introduction}
	
	When $S$ is an orientable surface of finite type (possibly with boundary) with negative Euler characteristic, the Teichmüller space $\mathcal{T}(S)$ of $S$, which is the deformation space of hyperbolic structures on $S$ has been an object of huge interest in the last decades. One of its important features is its interaction with the mapping class group $\mathrm{MCG}(S)$ of the surface $S$, that is the group of orientation preserving homeomorphisms of the surface up to homotopy : Fricke's Theorem says that the mapping class group acts properly discontinuously on the Teichmüller space (see for example Theorem 12.2 in \cite{farb_primer_2012}). One can adopt a more algebraic point of view on this framework by  noticing that the Teichmüller space can be thought of as a space of representations from the fundamental group $\pi_1(S)$ into $\mathrm{PSL}(2,\R)$, the isometry group of $\HH^2$, by considering the holonomy of a given hyperbolic structure. More precisely, $\mathcal{T}(S)$ is a subset of the character variety $\chi(\pi_1(S),\mathrm{PSL}(2,\R))$ which is the quotient of the space of representation $\mathrm{Hom}(\pi_1(S),\mathrm{PSL}(2,\R))$ by the action of the group $\mathrm{PSL}(2,\R)$ by conjugacy (in fact, when considering character varieties, we usually work with the hausdorffization of this space, which is the biggest Hausdorff quotient). When $S$ is a closed surface, Goldman proved (\cite{goldman_topological_1988}) that $\mathcal{T}(S)$ correspond exactly to two components of the character variety (corresponding to the two possible orientations on the surface), consisting entirely of discrete and faithful representations. 
	The mapping class group $\mathrm{MCG}(S)$ is a subset of the outer automorphism group $\mathrm{Out}(\pi_1(S))=\mathrm{Aut}(\pi_1(S))/\mathrm{Inn}(\pi_1(S))$, which in the case of closed surfaces, is exactly an index two subgroup of the outer automorphism group. This result is known as the Dehn-Nielsen-Baer Theorem (see for example Chapter Eight in \cite{farb_primer_2012}). Thus, Fricke's Theorem tells us about  the action of the mapping class group (or of the outer automorphism group) on the $\mathrm{PSL}(2,\R)$ character variety. In general, the action of the mapping class group on the complement of the Teichmüller space is unknown. For closed surfaces, Goldman conjectured (\cite{goldman_topological_1988}) that this action should be ergodic. 
	
	When replacing $\mathrm{PSL}(2,\R)$ by other isometry groups  $\mathrm{Isom}(X)$ of Gromov-hyperbolic spaces $X$, there is a natural analogue of the Teichmüller space by considering convex-cocompact representations. Here we will say that a representation $\rho : \pi_1(S) \to \mathrm{Isom}(X)$ is convex-cocompact if its associated orbit map is a quasi-isometric embedding (see Definition \ref{def:quasi-iso-embbed}). When $X$ a rank-one symmetric space, this is equivalent to saying that there exists a non-empty convex subset of $X$, invariant by the action of $\rho(\pi_1(S))$, on which $\rho(\pi_1(S))$ acts properly discontinuous and cocompactly. It is a classical result that, as for the Teichmüller space, the set of convex-cocompact representations is open and that the mapping class group acts properly discontinuously on it. Determining if this domain is maximal or if there can exist other domains of discontinuity in the character variety is a central issue in the study of the action of the mapping class group. Bowditch on the one hand (\cite{bowditch_markoff_1998}), and Minsky on the other hand (\cite{minsky_dynamics_2013}), have proposed two different but somehow related approaches to construct open domain of discontinuity in the $\mathrm{PSL}(2,\C)$ character variety. While their work was focusing on representations of free groups (of rank two in Bowditch's work) in $\mathrm{PSL}(2,\C)$, they have natural analogues when considering representations of fundamental group of surfaces in isometry groups of metric spaces. We will discuss in detail Bowditch's and Minsky's approaches in section \ref{subsection:simple-stab} and \ref{subsec:Bowditch rep}.
	\\
	
	In this article, we will be interested in the specific case where $S$ is the four-punctured sphere, which we denote by $S_{0,4}$. We will consider representations in $\mathrm{Isom}(X)$, where $X$ is a $\delta$-hyperbolic, geodesic and visibility space. By \emph{$\delta$-hyperbolic} we refer to hyperbolicity in the sense of Gromov, with hyperbolicity constant $\delta$. Hence $X$ is in particular a metric space for which we denote $d$ the associated metric. The \emph{geodesicity} hypothesis means that there always exist a geodesic between two distinct points in $X$ and the \emph{visibility} hypothesis means that there always exist a bi-infinite geodesic between two distinct points in the Gromov-boundary of the hyperbolic space $X$. 
	
	The aim of this article is to compare the notions of simple-stability (Definition \ref{def:simple-stable} in section \ref{subsection:simple-stab}), denoted $\mathcal{SS}(\groupe,\isom)$, and of Bowditch representations (Definition \ref{def:bowditch rep} in section \ref{subsec:Bowditch rep}), denote $\mathcal{BQ}(\groupe,\isom)$, in the case of the four-punctured sphere group $S_{0,4}$. 
	
	\begin{Theorem}
		\label{sph-BQ=PS} Let $X$ be a $\delta$-hyperbolic, geodesic and visibility space. \\ 
		The sets of Bowditch representations and of simple-stable representations of $\groupe$ in $\mathrm{Isom}(X)$ are equal. 
		\begin{equation*}
			\mathcal{BQ}(\groupe,\isom)=\mathcal{SS}(\groupe,\isom).
		\end{equation*}
	\end{Theorem} 
	
	A similar result was proved by the author in \cite{schlich_equivalence_2024} in the case of the free group of rank two. For the free group of rank two and $X=\HH^3$, this was first done by Lee and Xu (\cite{lee_bowditchs_2019}) and independently Series (\cite{series_primitive_2019},\cite{series_primitive_2020}). The proofs and the techniques developed here are independent of the work of Lee-Xu, and Series in $\HH^3$. \\
	Examples of spaces $X$ which satisfies the hypothesis of our Theorem include all rank one symmetric spaces (real hyperbolic spaces $\HH^n$, complex hyperbolic spaces $\HH^n_\C$, quaternionic hyperbolic spaces $\HH_H^n$ and the hyperbolic plane over the Cayley number $\HH_\mathbb{O}$, see \cite{knapp_lie_1996} for the classification) or even non proper spaces such as $\HH^{\infty}$ for example. \\
	
	We say that a closed curve on $S_{0,4}$ is \emph{simple} if there exists a representative in its free homotopy class which has no self-intersection, and which does not bound a disk or a once-punctured disk. We denote by $\simple$ the set of free homotopy classes of (unoriented) simple closed curves on $S_{0,4}$. Although this set is technically in bijective correspondence with a subset of $\groupe/\sim$, where $\sim$ identifies conjugate and inverse elements in $\groupe$, we will sometimes, with a slight abuse of notation, simply write $\gamma \in \simple$ to refer to a (cyclically reduced) word $\gamma$ in the fundamental group $\pi_1(\sph)$ representing a simple closed curve in $\simple$, thus identifying curves on the surface $\sph$ and representatives in the fundamental group. 
	
	\subsection{Simple-stability} \label{subsection:simple-stab}~ \\
	Let us denote by $\mathcal{C}(\groupe)$ the Cayley graph of the group $\groupe$ (recall that $\groupe \simeq \F_3$). This graph is endowed with a distance $d$ which comes from the word length. More precisely, when $\gamma \in \groupe$, we will denote $|\gamma|$ the reduced word length of $\gamma$ (in the set of generators used to define the Cayley graph) and for all $u,v \in \groupe$, $d(u,v)=|u^{-1}v|$. Recall that we also denote $d$ the distance in $X$, but in context there should be no ambiguity to which distance we refer. Note that the Cayley graph depends on the choice of a generating set for $\groupe$, but any two such graphs are quasi-isometric. Since in the following, we will only be interested in properties that are invariant under quasi-isometries, the notions we will define will not depend on this choice of generators and so, for simplicity we do not refer about the generating set. When $\gamma \in \groupe$, $\gamma$ has an axis in the Cayley graph $\mathcal{C}(\groupe)$, that is a $\gamma$-invariant geodesic in $\mathcal{C}(\groupe)$. We will use the notation $L_\gamma$ to refer to this geodesic. \\
	
	Fix a basepoint $o \in X$.
	To any representation $\rho : \groupe \to \mathrm{Isom}(X)$, we can associate an orbit map, which we denote by $\tau_\rho$, and which is a map from $\mathcal{C}(\groupe)$ to $X$ such that $\tau_\rho(1)=o$, $\tau_\rho$ is $\rho$-equivariant, and  $\tau_\rho$ sends the edges of the graph to hyperbolic segment in $X$. Note that under these assumptions we have $\tau_\rho(\gamma)=\rho(\gamma)o$, for all element $\gamma \in \groupe$. Moreover, since $\groupe$ is finitely generated, we automatically have, using the triangle inequality, that $\tau_\rho$ is Lipschitz-continuous. \\
	
	Let us recall the definition of a quasi-isometric embedding, which will be used in following : 
	\begin{Definition}\label{def:quasi-iso-embbed}
		Let $C>0$ and $D\geq 0$ be two constants. A map $f : (X,d_X) \to (Y,d_Y)$ between two metric spaces is a \emph{$(C,D)$-quasi-isometric embedding} if, for all $x,x' \in X$, we have :
		\begin{equation*}
			\frac{1}{C}d_X(x,x')-D \leq d_Y(f(x),f(x')) \leq Cd(x,x')+D.
		\end{equation*}
	\end{Definition}

	In \cite{minsky_dynamics_2013}, Minsky defined the notion of \emph{primitive-stability} for representations of free groups $\F_n$ into $\mathrm{PSL}(2,\C)$ by considering the set of primitive elements in $\F_n$ and requiring that the orbit map restricted to  geodesic corresponding to primitive elements are mapped, via the orbit map, to uniform quasi-geodesics. His definition can easily be adapted to other contexts, in particular to the case where the group is the fundamental group of a surface (closed or not), by replacing in the definition the set of primitive elements by the set of simple elements in the group (the elements in the group corresponding to simple closed curves on the surface). Hereafter we write the definition of simple-stable representations of the four-punctured sphere group, because this paper is dedicated to this particular surface, but the definition of course also makes sense for any other surfaces. 

	\begin{Definition}[Simple-stability] \label{def:simple-stable}Let $\rho : \groupe \to \isom$ be a representation. We say that $\rho$ is \emph{simple-stable} if there exist two constants $C>0$ and $D\geq 0$ such that for any simple closed curve $\gamma \in \simple$, the orbit map $\tau_\rho$ restricted to the geodesic $L_\gamma$ is a $(C,D)$-quasi-isometric embedding. \\
	We denote $\mathcal{SS}(\groupe),\isom)$ the set of simple-stable representations from $\groupe$ to $\isom$.
	\end{Definition}
	
	In particular, we immediately see from the definition that the set of simple-stable representations contains the set of convex-cocompact representations, as defined above in the introduction.
	
	Minsky proved (in \cite{minsky_dynamics_2013}), that in the character variety, the set of primitive-stable representations in $\mathrm{PSL}(2,\C)$ is open, $\mathrm{Out}(\F_n)$-invariant and that $\mathrm{Out}(\F_n)$ acts on it properly discontinuously. Moreover, he also proved that the set of primitive-stable representations strictly contains the set of convex-cocompact representations. In particular, this shows the existence of a non discrete primitive-stable representation in this setting. The openness still holds when considering simple-stable representations in the more general setting of $\delta$-hyperbolic spaces. For details, a proof is written by the author in section 4.4 of \cite{schlich_equivalence_2024}. Note that although the article \cite{schlich_equivalence_2024} focuses on the group $\F_2$, the proof of openness works exactly the same for other groups. Since the set of simple closed curves on the surface is invariant under the action of the mapping class group, we easily deduce that $\mathcal{SS}(\groupe,\isom)$ is $\mathrm{MCG}(S_{0,4})$-invariant. The fact that the action is properly discontinuous follows the same lines than in the case of primitive-stable representations.  
	Thus we have that $\mathcal{SS}(\groupe,\isom)$ is an open domain of discontinuity for the action of the mapping class group. \\
	
	Let us make a brief remark on the case of the once-punctured torus. Its fundamental group is the free group of rank two and it happens that there is a bijective correspondence between the set of simple closed curves on the once-punctured torus and the primitive elements in $\F_2$. Moreover, the mapping class group of the once-punctured torus is the outer automorphism group of $\F_2$. Thus, here the notions of primitive-stability in $\F_2$ and simple-stability of the once-punctured torus coincide (but this is the only punctured surface for which this happens).
	
	\subsection{Bowditch representations}    ~\\ 
	\label{subsec:Bowditch rep}
	Recall that we introduced above the word length $\vert \gamma \vert$ of an element $\gamma \in \groupe$. We will also consider its cyclically reduced word length, which we denote $\Vert \gamma \Vert$. 
	For $A$ an isometry of the space $X$, we define its displacement length as the quantity $l(A):=\underset{x\in X}\inf d(Ax,x)$. Note that both the cyclically reduced word length and the displacement length are conjugacy invariants. \\ 
	
	We now give the definition of a Bowditch representation, and we will explain the denomination  after. 
	
	\begin{Definition}[Bowditch representations]
		\label{def:bowditch rep}
		Let $\rho : \groupe \to \isom$ be a representation. We say that $\rho$ is a \emph{Bowditch representation} if there exist two constants $C>0$ and $D\geq 0$ such that for all simple closed curve $\gamma \in \simple$, we have $\displaystyle 	\frac{1}{C}\Vert \gamma \Vert -D \leq l(\rho(\gamma))$.\\
		We denote $\mathcal{BQ}(\groupe,\isom)$ the set of Bowditch representations from $\groupe$ to $\isom$. 
	\end{Definition}
	
	We will see in section \ref{sph-prop-Bowditch} that the additive constant actually plays no role, and that we get an equivalent definition by removing it. Also note that the other inequality ($\displaystyle l(\rho(\gamma)) \leq C' \Vert \gamma \Vert $) is always true for some constant $C'$, because the group $\groupe$ is finitely generated. \\
	
	It is quite straightforward from the definition that simple-stable representations are in particular Bowditch representations. A proof of this inclusion can be found in section 4.2 in \cite{schlich_equivalence_2024} (here the group which is studied is $\F_2$, but nothing relies on the specific structure of $\F_2$, so the proof also works for $\groupe$, or other fundamental groups of surfaces). Therefore the whole difficulty of Theorem \ref{sph-BQ=PS} is to prove that Bowditch representations are simple-stable. \\
	
	In \cite{bowditch_markoff_1998}, Bowditch studied representations of the free group of rank two $\F_2$ in $\mathrm{PSL}(2,\C)$. Let $\F_2=<a,b>$ be a free generating set for $\F_2$ and $\mathcal{P}(\F_2)$ denote the set of conjugacy class of primitive elements in $\F_2$. He introduced the three following conditions on a representation $\rho : \F_2 \to \mathrm{PSL}(2,\C)$ :
	\begin{enumerate}
		\item \label{cond:BQ2} $\forall \gamma \in \mathcal{P}(\F_2), \mathrm{Tr}(\rho(\gamma)) \in \C \backslash [-2,2]$
		\item \label{cond:BQ3} The set $\{ \gamma \in \mathcal{P}(\F_2) : \vert \mathrm{Tr}(\rho(\gamma)) \vert \leq 2 \} $ is finite. 
		\item \label{cond:BQ1} $\mathrm{Tr}(\rho([a,b]))=-2 $
	\end{enumerate} 
	Bowditch proved that the representations satisfying those conditions form an open domain of discontinuity in the character variety. His approach makes extensive use of the Farey triangulation of the hyperbolic plane and more specifically its dual tree, and uses the language of Markoff maps. His work was later generalized by Tan, Wong and Zhang in \cite{tan_generalized_2008}, who made an analogue study in the case where the value of the commutator is no more necessarily $-2$, hence removing the third condition above. Both Bowditch and Tan-Wong-Zhang proved a "Fibonacci-growth" property on these representations, which we can, combining their works, reformulate as follows~:
\begin{equation} \label{fibonacci}
	\rho \text{ satisfies the conditions  \eqref{cond:BQ2} and \eqref{cond:BQ3}} \iff \exists C>0, \forall \gamma \in \mathcal{P}(\F_2), \quad \frac{1}{C} \Vert \gamma \Vert \leq l(\rho(\gamma)).
\end{equation}
	This equivalence inspires the more general definition of a Bowditch representation (Definition \ref{def:bowditch rep}) in the more general setting of $\delta$-hyperbolic spaces and for fundamental groups of surface (possibly with boundary). Note that it is not clear from Definition \ref{def:bowditch rep} that this defines an open subset in the character variety, nor that the mapping class group acts properly discontinuously on it. However as mentioned in section \ref{subsection:simple-stab}, it is not hard to prove that the set of simple-stable representations is an open domain of discontinuity, we therefore obtain the following corollary :
	\begin{Corollary}
		The set $\mathcal{BQ}(\groupe,X)$ is an open domain of discontinuity for the action of $\mathrm{MCG}(\groupe)$ on $\chi(\groupe,\mathrm{Isom}(X))$. \\
	\end{Corollary}
	
	Maloni, Palesi and Tan studied in \cite{maloni_character_2015} the $\mathrm{SL}(2,\C)$-character variety of the four-punctured sphere together with the action of the mapping class group on it. They defined, in the same spirit as Bowditch and Tan-Wong-Zhang, a set of representation defined by conditions on the traces and proved that it forms an open domain of discontinuity. More precisely, their conditions on $S_{0,4}$ are as follows :
	\begin{enumerate} 
		\item[(BQ1)] 
		 $\forall \gamma \in \simple, \mathrm{Tr}(\rho(\gamma)) \in \C \backslash [-2,2]$
		\item[(BQ2)] The set $\{\gamma \in \simple : \vert \mathrm{Tr}(\rho(\gamma)) \vert \leq K(\mu) \} $ is finite. 
	\end{enumerate}
	Here, $K(\mu)$ is a constant which depends on $\mu=(\mathrm{Tr}(\rho(a)),\mathrm{Tr}(\rho(b)),\mathrm{Tr}(\rho(c)),\mathrm{Tr}(\rho(d)))\in \C^4$, where $a,b,c,d$ are the four boundary components of $S_{0,4}$.  In addition, Maloni, Palesi and Tan proved that the representations satisfying the above two conditions also have "Fibonacci growth", in other words that the equivalence \eqref{fibonacci} is again true in this context.  We can thus combine their work with our Theorem \ref{sph-BQ=PS} to obtain :
	 
	 \begin{Corollary} \label{cor:PSL(2,C)}
	 	Let $\rho : \groupe \to \mathrm{SL}(2,\C)$. The following conditions are equivalent :
	 	\begin{enumerate}
	 		\item $\rho$ satisfies conditions (BQ1) and (BQ2) of Maloni-Palesi-Tan 
	 		\item $\rho$ is simple-stable.
	 	\end{enumerate}
	 \end{Corollary}
	 
	In an other direction, Maloni and Palesi also analysed the $\mathrm{SL}(2,\C)$ character variety of the three-holed projective plane in \cite{maloni_character_2020} and together with Yang they studied its $\mathrm{PGL}(2,\R)$ character variety in \cite{maloni_type-preserving_2021}.	

	\subsection{Redundancy of subwords of simple words in $S_{0,4}$ and strategy of the proof} ~\\
	We now explain what is the general strategy of the proof of Theorem \ref{sph-BQ=PS}. 

	\subsubsection{Redundancy of subwords of simple words}
 	First, we will develop a study of the combinatorics of simple closed curves on the four-punctured sphere. This is done in sections \ref{sec:combinatorics} and \ref{sec:redundancy}. In particular, we prove Theorem \ref{magic-len-S_{0,4}}, which quantifies, after defining some specific lengths, how much an arbitrary subword of one of these lengths repeats itself inside simple words of $\groupe$ (we will explain in more details what we mean by this).  \\
 	
 	We will identify the four-punctured sphere with a quotient of the plane $\R^2$ minus the lattice $\Z^2$ in order to see the simple closed curves as quotient of straighlines with rational slope in $\R^2$. In particular, this will highlight the (bijective) correspondence between rational numbers and equivalence classes of simple closed curves on the four-punctured sphere (Proposition \ref{essential->Q}). We will call these rationals the "slopes" associated to simple closed curves and their continued fractions expansions $[n_1,\dots,n_r]$ will be of particular interest. In particular, considering the truncated fractions expansion $[n_1,\dots,n_i]$ for all $1 \leq i \leq r$, we will define some "special lengths", denoted $l_i(\gamma)$ (see Definition \ref{def:magic-length}) and which will play a crucial role in our study and in Theorem \ref{magic-len-S_{0,4}}. In section \ref{construc-simple}, we will study the combinatorics of a simple closed curve $\gamma$ at different scales, each of them corresponding to a level $i$ in the continued fraction expansion of the slope of $\gamma$ (that is each integer $0 \leq i \leq r)$. In particular, Lemma \ref{general-form-gamma} will decompose a simple closed curve $\gamma$ seen as a word in $\groupe$, for each level $i$, as a concatenation of words which will only depend on the integers $n_1,\dots,n_i$. \\
 	
	Then, section \ref{sec:redundancy} will be dedicated to the statement and the proof of the "Special-lengths" Theorem \ref{magic-len-S_{0,4}}, which will be a major tool for our proof of Theorem \ref{sph-BQ=PS}, and might be of independent interest. Since Theorem \ref{magic-len-S_{0,4}} has a few technical assumptions, for the sake of clarity we give here a slightly weaker version (Theorem \ref{thm:magic-len-simplifié}) but which gets rid of some of the technicalities. In this theorem, we first fix a simple closed curve $\gamma$. Now recall that we have defined some "special lengths" associated to $\gamma$ : $l_1(\gamma),\dots,l_r(\gamma)$. We choose one of these lengths and an arbitrary subword of $\gamma$ of this specific length (in fact of this specific length minus five). Then we prove that this subword and its inverse occur a certain number of times in $\gamma$ which we can quantify uniformly in $\gamma$: the occurrences of this subword and its inverse in $\gamma$ occupy at least one thirtieth of the word $\gamma$.
	
	\begin{Theorem}[Special-lengths]\label{thm:magic-len-simplifié}
		Let $\gamma \in \simple$. Fix $10 \leq i \leq r(\gamma)$ such that $|\gamma| \geq 10 l_i(\gamma)$.  \\
		Let $w \in \F_3$ be any subword of $\gamma$ of length $l_i(\gamma)-5$. \\
		Then we can decompose the word $\gamma$ as a concatenation $\gamma=u_1 \hdots u_q$ such that there exists a subset $\cI \subset \{1,\hdots,q\}$ satisfying the following :
		\begin{enumerate}
			\item \label{change-letter-simplifié} For all $k \in \cI$, $u_k\in \{w,w^{-1}\}$.
			\item \label{proportion-simplifié} $\displaystyle \sum_{k\in \cI} |u_k| \geq \frac{1}{30} |\gamma|$.
		\end{enumerate}
	\end{Theorem}

Our approach to proving this theorem will be geometric, and so in section \ref{lattices} we will first need to introduce some generalities about tilings of $\R^2$ adapted to a lattice of $\R^2$, set up our notations and state a few facts needed in the following. Then section \ref{farey} will provide a short reminder about Farey neighbors. 

The proof will start by associating to every simple closed curve $\gamma$ and every level $1 \leq i \leq r$ of the continued fraction expansion of the slope of $\gamma$ a lattice $\Lambda$ of $\R^2$ (section \ref{mapping}). This lattice will be particularly useful to understand the word $\gamma$ at the level $i$ : within this geometric framework one can "read" the word $\gamma$ in $\R^2$ by following an horizontal line and the subwords of lengths approximately $l_i(\gamma)$ will be easier to visualize and to understand (section \ref{reading-gamma} and \ref{subwords-square}). Finally, section \ref{proof-magic-len} will combine the results on lattices in $\R^2$ of section \ref{lattices} with our study of subwords of $\gamma$ of length approximately $l_i(\gamma)$  in section \ref{reading-gamma} and \ref{subwords-square} to prove Theorem \ref{magic-len-S_{0,4}}.

\subsubsection{Strategy of the proof of Theorem \ref{sph-BQ=PS}} 
Having established the combinatorial results in the group and in particular Theorem \ref{thm:magic-len-simplifié}, the proof will basically follow the same strategy than the proof of the corresponding theorem for the free group of rank two in \cite{schlich_equivalence_2024}. However, since the combinatorics of simple closed curves on the four-punctured sphere is different and more complicated than the combinatorics of primitive elements in $\F_2$, the statements and their proofs have to be adapted. For completeness, we present now the main steps and ideas of the rest of the proof.\\

In section \ref{sph-UQG}, we will first study a uniform quasi-geodesicity setting in section \ref{sph-first-ex-UQG} which will later, when combined with Lemma \ref{D(G)neqG}, appears as a local uniform quasi-geodesicity property of Bowditch representations. In Lemma \ref{simple->hyp} of section \ref{sph-prop-Bowditch}, we will show that the images of simple elements under a Bowditch representation are always hyperbolic. In particular, denoting $L_\gamma$ the axis of $\gamma$ in the Cayley graph of $\groupe$ and $\tau_\rho$ the orbit map of a Bowditch representation $\rho$, we deduce that $\tau_\rho(L_\gamma)$ is always a quasi-geodesic. 

In section \ref{sph-first-step}, we will prove Proposition \ref{BQ=>CTU-Sph}, which will be the main step for proving Theorem \ref{sph-BQ=PS}. It states that the simple leaves $L_\gamma$ are sent by the orbit map to uniform neighborhoods of the axes of the hyperbolic isometries $\rho(\gamma)$. To do so, we proceed by contradiction and to this purpose we introduce a sequence $(\gamma_n)_{n\in \N}$ of simple elements in $\groupe$ such that the images of the orbit map on the simple leaf $L_{\gamma_n}$ becomes further and further away from the axis of $\rho(\gamma_n)$ as $n$ increases. We will then study the continued fraction expansion of the slope of $\gamma_n$ which we denote by $[N^n_1,\dots,N^n_{r(n)}]$. The uniform quasi-geodesicity setting studied in section \ref{sph-first-ex-UQG}, together with Lemma \ref{D(G)neqG} will enable us to show that the sequences $(N^n_i)_{n\in\N}$ must be bounded for all $i$, and, as a consequence, the depth $r(n)$ of the continued fraction expansion will tend to infinity (Lemma \ref{forme-frac-continue-sph}). We will next introduce the notion of a $K$-excursion of the orbit map (section \ref{excursion-orbit-map-sph}) and notice that we can extract from the sequence $(\gamma_n)_{n\in \N}$ a sequence of $K_n$-excursions as large as we want ($K_n \to \infty$) in Lemma \ref{excursions-grandes-sph}. We will also introduce the notion of an $\varepsilon$-quasi-loop, and since large excursions will correspond to quasi-loop (Lemma \ref{excursion->QB-sph}), we will find a quasi-loop in each $\gamma_n$ (for large $n$). But we will be even more precise and, using the Special-lengths Theorem (\ref{magic-len-S_{0,4}}) together with the fact that once found a quasi-loop, we can find many others of smaller lengths "inside", we will find a uniformly bounded from below proportion of the word $\gamma_n$ consisting of disjoint quasi-loops (Lemma \ref{decoupe-restes-c-sph}). We will repeat our argument for the remainders in $\gamma_n$ that do not yet consist of disjoint quasi-loops in order to find an arbitrarily large proportion of the word $\gamma_n$ consisting of disjoint quasi-loops (Lemma \ref{trouve-gamma-sph}). To formalize this idea we will use a recursive argument in the proof. This mean that we will find an arbitrarily large proportion of the word $\gamma_n$ which does not displace the basepoint much, and this will be in contradiction with the Bowditch’s hypothesis.

At last, section \ref{second-step-proof-sph} will conclude the proof by showing that a Bowditch representation of $\groupe$ satisfying the conclusion of Proposition \ref{BQ=>CTU-Sph} is simple-stable by using similar techniques.
	
\subsection*{Acknowledgments} Part of this work was carried out during my PhD at the Université de Strasbourg. I would like to thank my thesis advisor François Guéritaud for his continuous support and guidance throughout this work. I would also like to
thank Louis Funar, Vincent Guirardel and Andrea Seppi for useful comments and remarks.
The author is funded by the European Union (ERC, GENERATE, 101124349). Views and opinions expressed are however those of the author only and do not necessarily reflect those of the European Union or the European Research Council Executive Agency. Neither the European Union nor the granting authority can be held responsible for them. \\

\section{Combinatorics of simple-closed curves on the four-punctured sphere} \label{sec:combinatorics}
	\subsection{The four-punctured sphere}
	Let $\sph$ be a (topological) four-punctured sphere, that is, a sphere with four distinct points removed. \\
	
	\begin{figure}[h]
		\centering

		\includegraphics[width=4cm]{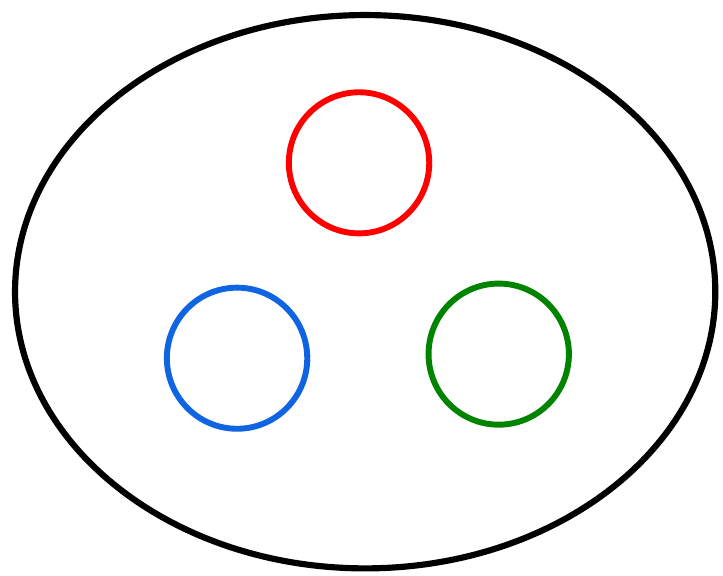} \hspace{2cm}
		\includegraphics[width=3cm]{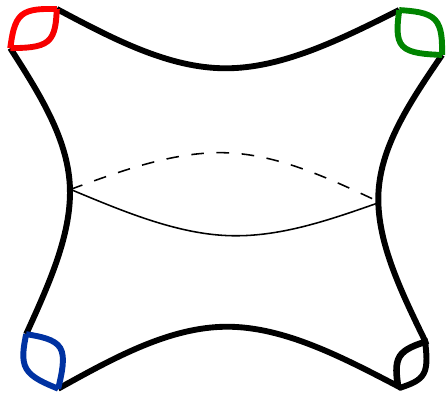}
		\caption{The four punctured-sphere $S_{0,4}$}
		\label{fig:sphere}
	\end{figure}
	
	We will use the following model for $\sph$, and see it as the quotient of the plane minus a lattice in the following way : 
	\begin{equation*}
		\sph \simeq (\R^2 \setminus \Z^2) / \la \{s_\lambda : \lambda \in \Z^2 \} \ra 
	\end{equation*}
	where $s_\lambda$ is the reflection of the plane across the point $\lambda$ (hence we have the formula $s_\lambda(u)=2\lambda-u$).
	Moreover, it is not hard to check that this model is equivalent to the following one : 
	\begin{equation*}
		\sph \simeq (\R^2 \setminus \Z^2) / \la 2\Z^2,\pm \ra 
	\end{equation*}
	The four punctures on the sphere are given by the four classes of points in $\Z^2$ given by the action of $2\Z^2$ on it. They are drawn in black, blue, green and red on figure \ref{fig:plane}. A fundamental domain is given in figure \ref{fig:triangle} by identifying the edges of the triangle as illustrated on the figure. We will refer to this triangle in the following as the \emph{fundamental triangle}. \\

	\begin{figure}[h!]
		\centering
		\includegraphics[width=5cm]{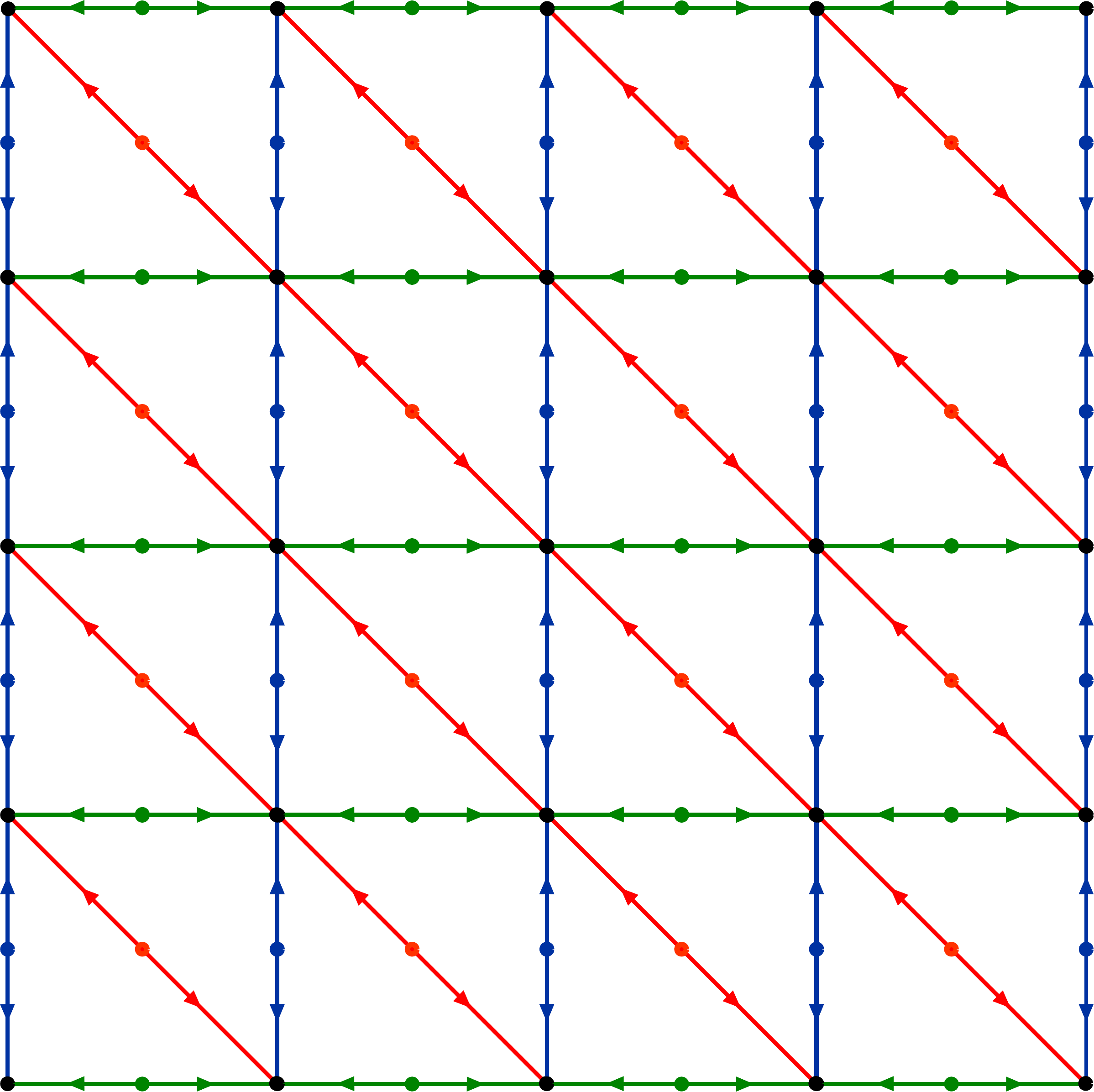}
		\caption{The sphere as the quotient \\ of the plane : $\sph \simeq (\R^2 \setminus \Z^2) / (2\Z^2,\pm)$}
		\label{fig:plane}
	\end{figure}

		\begin{figure}[h!]
		\centering
		\begin{subfigure}{0.45 \textwidth}
			\centering
			\includegraphics[width=2.5cm]{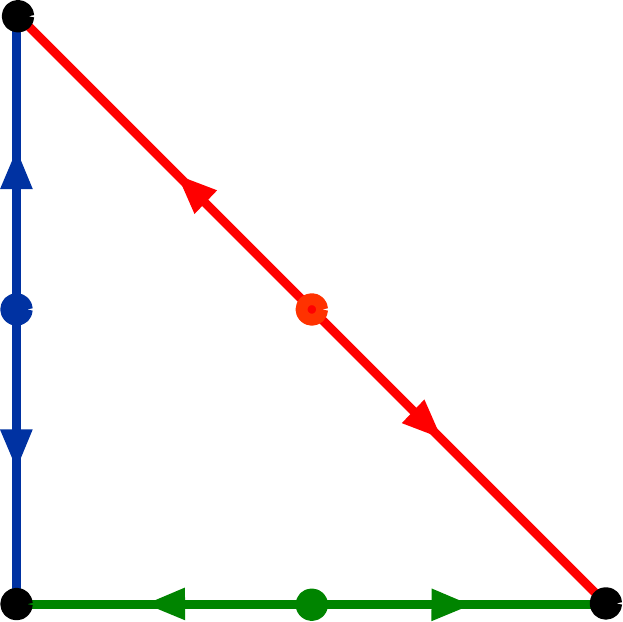}
			\caption{A fundamental domain}
			\label{fig:triangle}
		\end{subfigure}
		\begin{subfigure}{0.45 \textwidth}
			\centering
			\includegraphics[width=2.5cm]{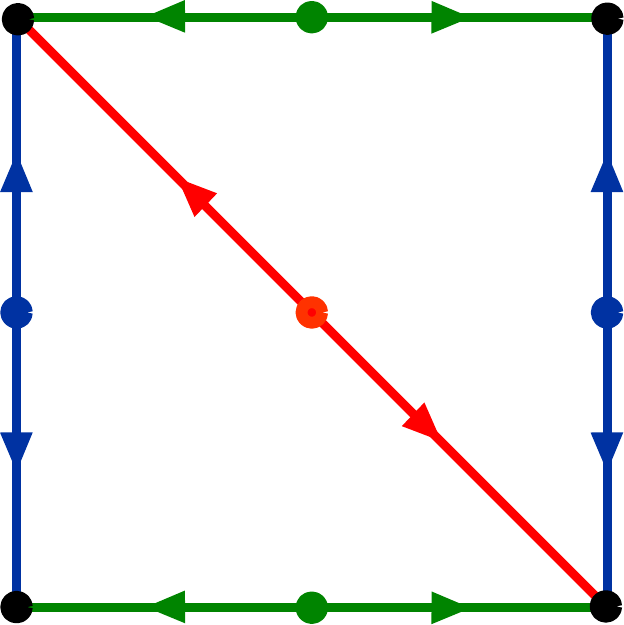}
			\caption{The sphere as the quotient of the square}
			\label{fig:square}
		\end{subfigure}
		\caption{Fundamental triangle and square for $\sph$}
		\label{fig:sphere-fun}
	\end{figure}
	
	Finally, we can also see the sphere $\sph$ as a quotient of the square $[0,2]^2 \setminus (\Z^2 \cap [0,2]^2)$ in $\R^2$ by identifying the opposite sides of the square and applying the reflection across the center point of the square~:~$(1,1)$ (see figure \ref{fig:square}). In this last description, we recover the classical fact that the four-punctured sphere is the quotient of the torus minus 4 points by the hyper-elliptic involution. 
	
	\subsection{Simple closed curves on the four-punctured sphere} \label{subsec:simple closed curves}
	The fundamental group of the four-punctured sphere is the free group of rank three $\F_3$. We fix once and for all the following free generating set : $\groupe = \F_3 = \la a,b,c \ra$, with $a$,$b$ and $c$ the curves corresponding to three of the four boundary components of $\sph$ as described in the picture \ref{fig:gen}.  Note that with this convention of orientation, the element $abc$ corresponds to the fourth boundary component (in black in the pictures). \\
	
	\begin{figure}[h]
		\centering
		\labellist
		\small\hair 2pt
		\pinlabel {$\color{green} a$} [c] at 300 150 
		\pinlabel {$\color{blue} b$} [c] at 40 130
		\pinlabel {$\color{red} c$} [c] at 95 215
		\pinlabel {$\color{green} a$} [c] at 665 50 
		\pinlabel {$\color{blue} b$} [c] at 525 80
		\pinlabel {$\color{red} c$} [c] at 565 150
		\pinlabel {$\color{green} a$} [c] at 1050 125 
		\pinlabel {$\color{blue} b$} [c] at 925 265
		\pinlabel {$\color{red} c$} [c] at 1030 300
		\endlabellist
		\includegraphics[width=3.5cm]{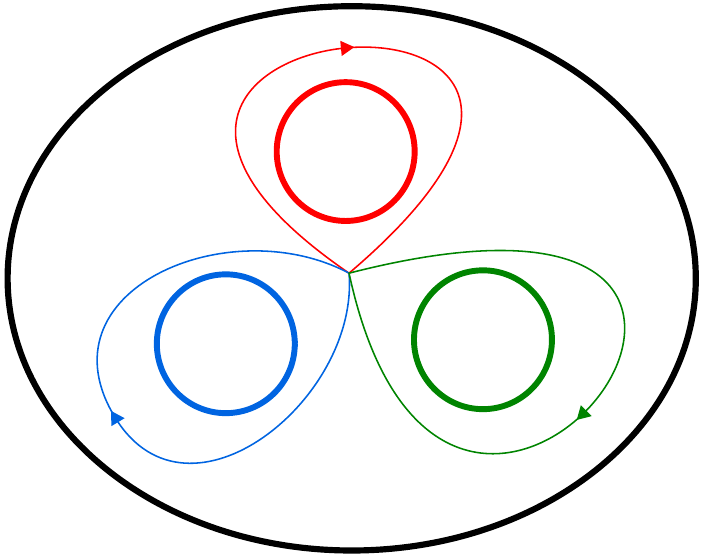} \hspace{1.05cm}
		\includegraphics[width=3cm]{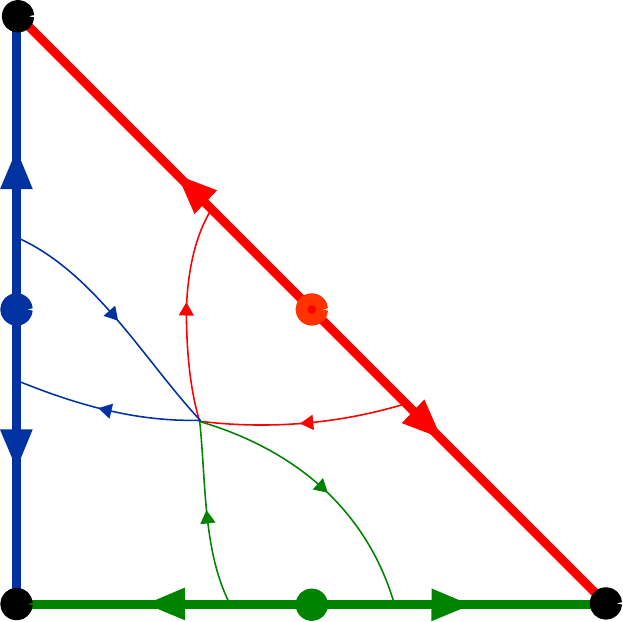}
		\includegraphics[width=3.5cm]{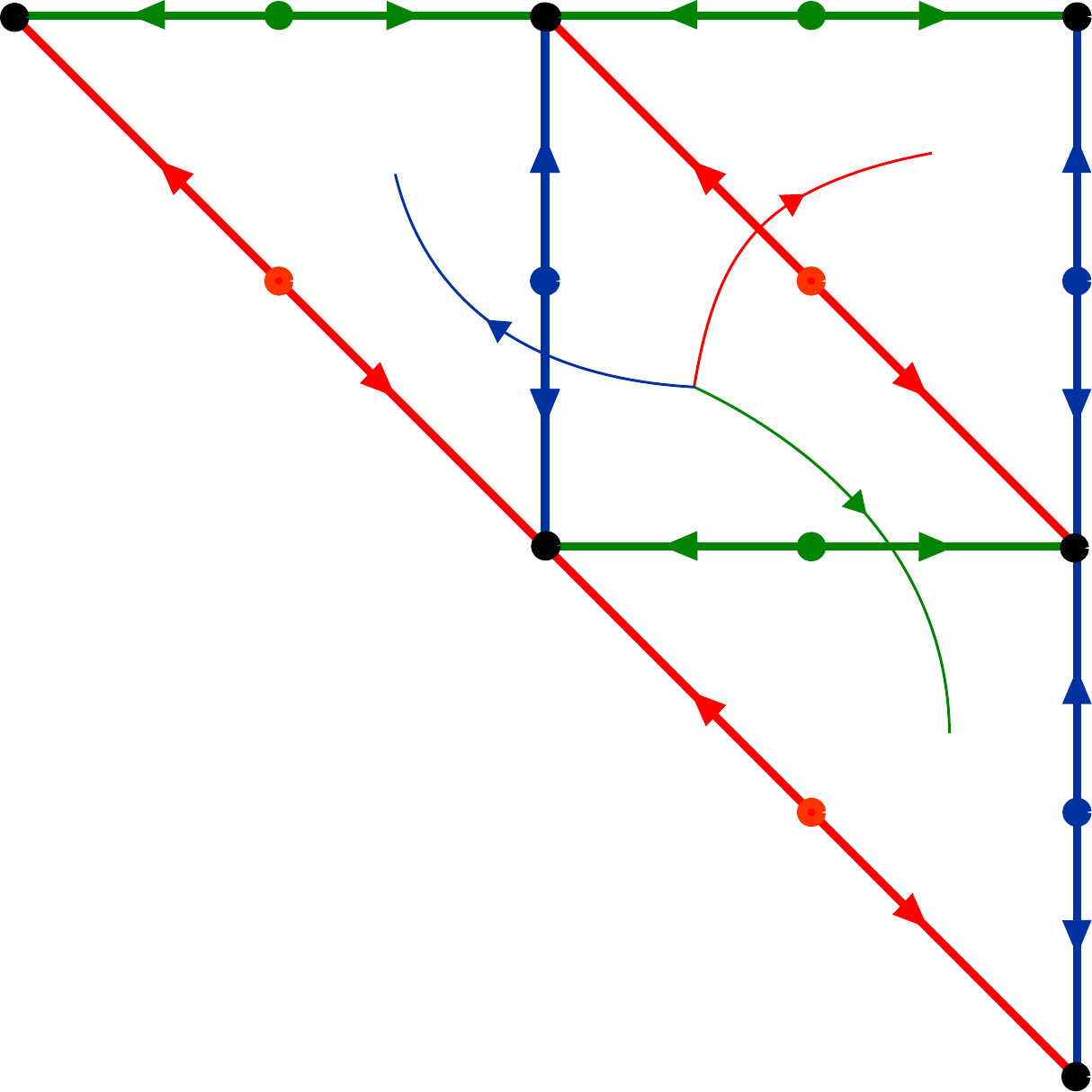}
		\caption{The three generators of $\groupe = \F_3 = \la a,b,c \ra $.}
		\label{fig:gen}
	\end{figure}
	
	Denote by $A,B$ and $C$ the three edges of the fundamental triangle respectively containing the punctures represented by the homotopy class $a,b$ and $c$. The puncture at the middle of each edge defines two half-edges. There is a total of six half-edges, on which we fix alternating transverse orientations (see figure \ref{fig:curve}). These transverse orientations are respected by the gluings. For every oriented curve $\gamma$ on $\sph$, we can write the corresponding word in the fundamental group (up to cyclic permutation) by following the curve $\gamma$ and writing the letter $a,b$ or $c$ respectively each time $\gamma$ crosses the edge $A,B$ or $C$, with power $\pm 1$ depending on whether or not the orientation of $\gamma$ at the intersection point agrees with the transverse orientation on the corresponding half-edge. \\
	
	\begin{figure}[h]
		\centering
		\labellist
		\small\hair 2pt
		\pinlabel {$\color{green} A$} [c] at 165 -10
		\pinlabel {$\color{blue} B$} [c] at -20 170
		\pinlabel {$\color{red} C$} [c] at 200 190
		\pinlabel {$\gamma$} [c] at 70 80
		\pinlabel {$c$} [c] at 90 265
		\pinlabel {$a$} [c] at 200 0
		\pinlabel {$c^{-1}$} [c] at 210 160
		\pinlabel {$b$} [c] at 5 140
		\pinlabel {$a$} [c] at 285 0
		\endlabellist
		\includegraphics[width=3cm]{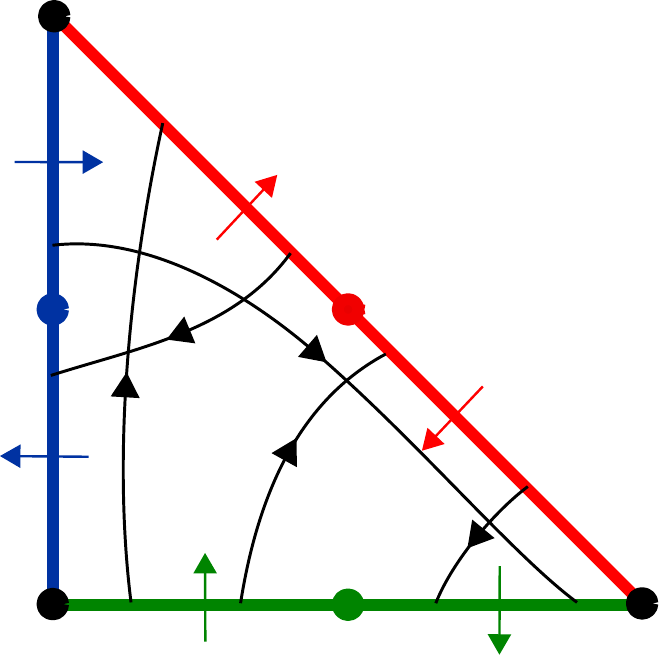}
		\caption{Reading the word $\gamma = cac^{-1}ba$ in the fundamental triangle}
		\label{fig:curve}
	\end{figure}
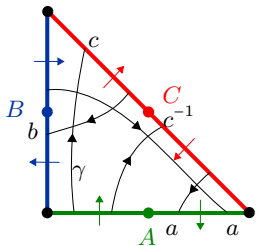
	
	Let us go back to the plane. We say that a point in $\Z^2$ is of \emph{type} $a,b,c$ or $abc$, if its class mod $2\Z^2$ corresponds to the puncture represented respectively by the homotopy class $a,b,c$ or $abc$ (hence a point is of type $a,b,c$ or $abc$, if its class mod $\Z^2$ is respectively $(1,0),(0,1),(1,1)$ or $(0,0)$). The points of type $a,b,c$ or $abc$ are exactly the lifts of the puncture represented respectively by $a,b,c$ or $abc$. In our pictures, the points of type $a$ are in green, the points of type $b$ in blue, the points of type $c$ in red and the points of type $abc$ in black. We will keep this convention in all the following. Now let us denote by $L_A$ the set of horizontal lines in the plane passing through a point of type $a$, $L_B$ the set of vertical lines in the plane passing through a point of type $b$, and $L_C$ the set of lines of slope $-1$ in the plane passing through a point of type $c$ (hence we have $L_A = \{ \{y = 2 \lambda \} \mid \lambda \in \Z \}$, $L_B = \{ \{x = 2 \lambda \} \mid \lambda \in \Z \}$ and $L_C = \{ \{y=-x+2\lambda \} \mid \lambda \in \Z \}$). Then through a point of type $a,b$ or $c$ passes exactly one line in $L_A\cup L_B \cup L_C$ whereas through a point of type $abc$ pass exactly three lines in $L_A\cup L_B \cup L_C$ : one in $L_A$, one in $L_B$ and one in $L_C$. We say that a line $l$ is of type $a,b$ or $c$ respectively if $l \in L_A, l \in L_B$ or $l \in L_C$. Two different lines of the same type never intersect whereas two lines of two different types always intersect in a point of $\Z^2$ of type $abc$. Moreover, each line in $L_A$, $L_B$ and $L_C$ is a union of segments (of the same lengths) with endpoints in $\Z^2$ and with no point of $\Z^2$ in the interior of the segments. The two endpoints of each segment are of two different types (one endpoint is of type $t \in \{a,b,c\}$, where $t$ is the type of the line, and the other endpoint is of type $abc$). Those segments correspond exactly to a half-edge in the fundamental triangle. Therefore, on each segment, we can put a transverse orientation which is just the lift of the orientation on the corresponding half-edge in the fundamental triangle. \\
	Now we can play in the plane the same game as in the fundamental triangle and write, for a curve $\gamma$ in $\sph$, the corresponding word in the fundamental group by following a lift $\t{\gamma}$ of $\gamma$ in $\R^2 \setminus \Z^2$ and record $a,b$ or $c$ each time $\t{\gamma}$ crosses a line in $L_A,L_B$ or $L_C$ and put a sign $\pm 1$ depending on whether or not the orientation of $\t{\gamma}$ at the intersection point coincides with the transverse orientation on the segment.   
	
	\begin{figure}[h!]
		\centering
		\labellist
		\small\hair 2pt
		\pinlabel {$\t{\gamma}$} [c] at 230 530
		\pinlabel {$c$} [c] at 160 200
		\pinlabel {$a$} [c] at 155 330
		\pinlabel {$c^{-1}$} [c] at 220 440
		\pinlabel {$b$} [c] at 335 500
		\pinlabel {$a$} [c] at 500 325
		\pinlabel {$\color{green} L_A$} [c] at 100 930
		\pinlabel {$\color{blue} L_B$} [c] at -35 800
		\pinlabel {$\color{red} L_C$} [c] at 210 770 
		\endlabellist
		\includegraphics[width=7cm]{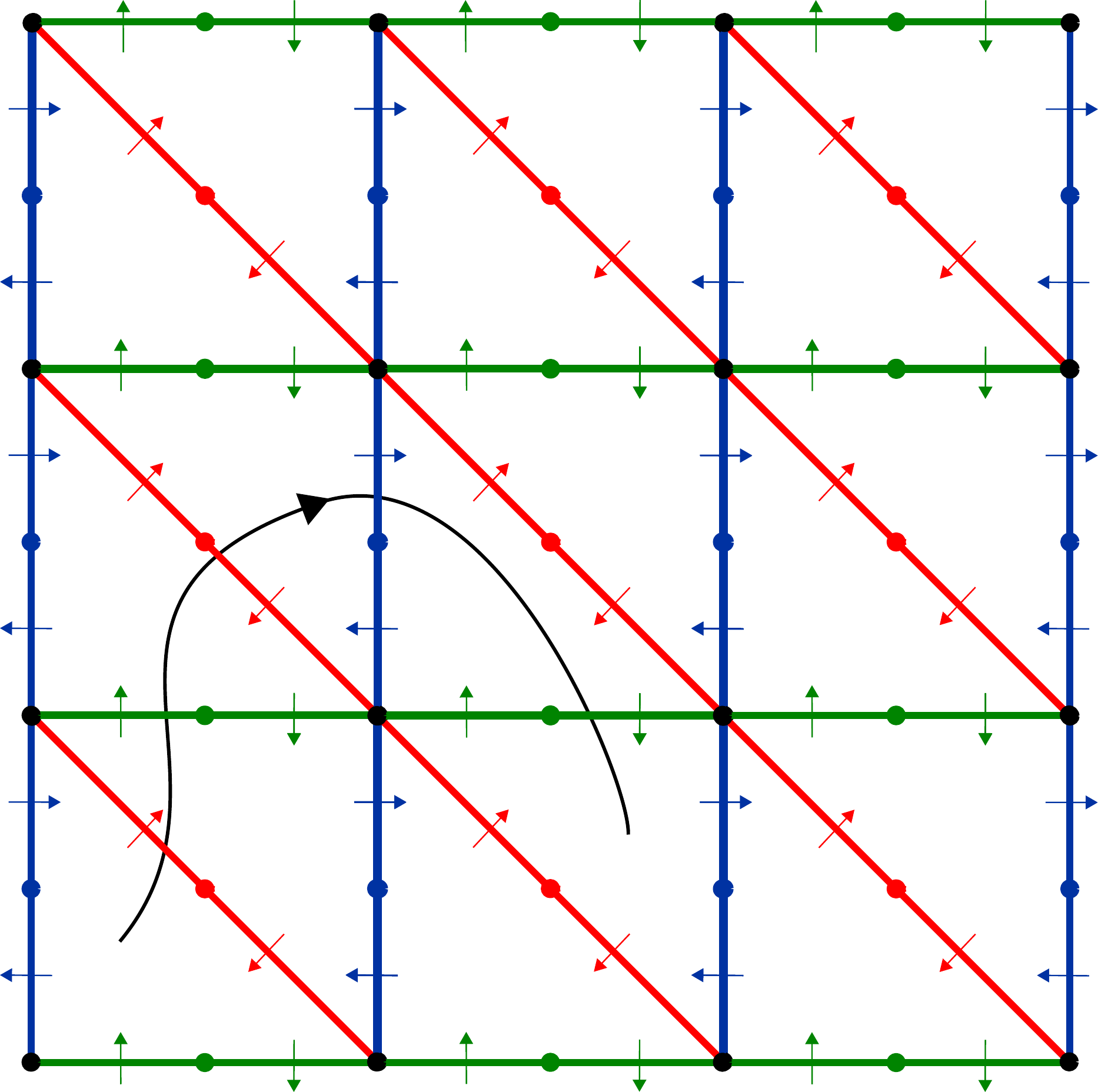}
		\caption{Reading the curve $\gamma=cac^{-1}ba$ in the plane using the three sets of parallel lines $L_A,L_B$ and $L_C$}
		\label{fig:curve_plane}
	\end{figure}
	
	In the fundamental triangle, curves on $\sph$ look like arcs joining the three edges $A,B$ and $C$ of the triangle. Notice that since the two half-edges of an edge are identified by a reflection, if an arc has an endpoint in one half-edge, this forces another arc to have one of its endpoints in the other half-edge of the same edge, and those two endpoints must be reflected from each other across the midpoint of the edge. An arc can possibly join the same edge, but in this case it joins two different halves of the same edge, otherwise this arc would be trivial and could be removed by a homotopy of the curve (corresponding to a cancellation such as $aa^{-1}$ in $\gamma$). \\
	
	Recall that, as defined in the introduction, we say that a closed curve $\gamma$ on $\sph$ is \emph{simple} if there exists a representative of $\gamma$ in its homotopy class which has no self-intersection and which does not bound a disk or a once-punctured disk, and we denote by $\simple$ the set of free homotopy classes of (unoriented) simple closed curves on $\sph$. First notice that every line with rational slope in the plane (which avoids $\Z^2$) gives a simple closed curve in the quotient. The converse is also true and is the main purpose of the following classical result :
	
	\begin{Proposition} \label{essential->Q}
		Let $\gamma$ be a simple closed curve on $\sph$. Then, after homotopy, $\gamma$ can be lifted to a line of rational slope in the plane. 
		Therefore, there exists a well defined map $\slope : \simple \to \Q \cup \infty$, which is a bijection.
	\end{Proposition}
	
	\begin{proof}
		Fix $\gamma$ a simple closed curve on $\sph$. Let us look at $\gamma$ as a collection of arcs in the fundamental triangle. This collection of arcs does not intersect, because the curve is assumed to be simple. By a slight abuse of notation, also denote by $\gamma$ the corresponding cyclically reduced word in the fundamental group. We are going to prove the following facts : 
		
		\begin{itemize}
			\item \textbf{Fact 1} : Every letter in $\gamma$ is isolated. \\
			It suffices to show that if a word contains the pattern $\ldots s^2 \dots$, with $s \in \{a,a^{-1},b,b^{-1},c,c^{-1} \}$, then the corresponding curve contains self-intersection. It is easy to see that the presence of $s^2$ in the word forces the existence of an arc $\beta$ from an edge $e$ of the fundamental triangle to itself, starting and ending in two different half-edges of $e$. Suppose $\beta$ is an innermost such arc. Then another arc (the next one or the previous one when following the curve) has an endpoint on $e$, in between the two endpoints of $\beta$. But since this new arc is not innermost, it has to "escape" the region bordered by $\beta \cup e$. This would create self-intersection. 
			\begin{figure}[h]
				\centering
				\labellist
				\small\hair 2pt
				\pinlabel{$s$} [c] at 220 0
				\pinlabel {$s$} [c] at 270 0
				\pinlabel{$\beta$} at 70 80
				\pinlabel{$\color{green} e$} at 170 -15 
				\endlabellist
				\includegraphics[width=3cm]{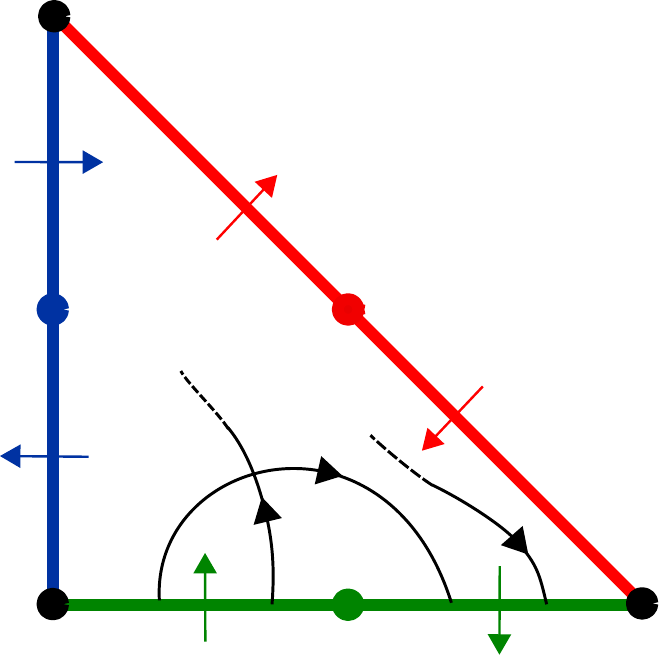}
				\caption{The pattern $\hdots s^2 \hdots$ forces intersection. (Here $e=A, s=a$).}
				\label{fig:simple1}
			\end{figure}
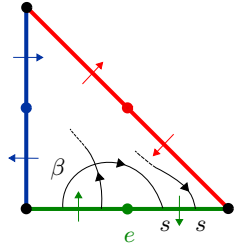
			
			\item \textbf{Fact 2} : There is an edge of the fundamental triangle such that every arc has an endpoint on it. \\
			By fact 1, there is no arc joining an edge to itself. Note that the only case which is not covered by the previous fact is when $\gamma$ has a single arc from an edge to itself, but in this case, this would mean that $\gamma$ is in fact $a,b$ or $c$, hence a boundary curve, which is not possible because $\gamma$ is a simple curve (and then it is supposed to be non-peripheral). In other words, all the arcs join different edges. Suppose that any two distinct edges are joined by at least an arc. Under this assumption, consider for every pair of distinct edges of the fundamental triangle the innermost arc joining those two edges, that is, the arc whose two endpoints are the closest to the intersection point between the two edges (which is the puncture $abc$). Then these three arcs combine together to form the boundary curve $abc$. We now deduce of this observation that these three arcs form a connected component of $\gamma$, which is taken simple (hence connected), so there is no other arc in $\gamma$. Therefore $\gamma=abc$ (up to cyclic permutation and inversion) which is not possible because $\gamma$ is supposed to be non-peripheral and $abc$ is a boundary curve. \\
			\begin{figure}[h]
				\centering
				\labellist
				\small\hair 2pt
				\pinlabel{$a$} [c] at 250 -10
				\pinlabel {$b$} [c] at -10 55
				\pinlabel {$c$} [c] at 70 255
				\endlabellist
				\includegraphics[width=3cm]{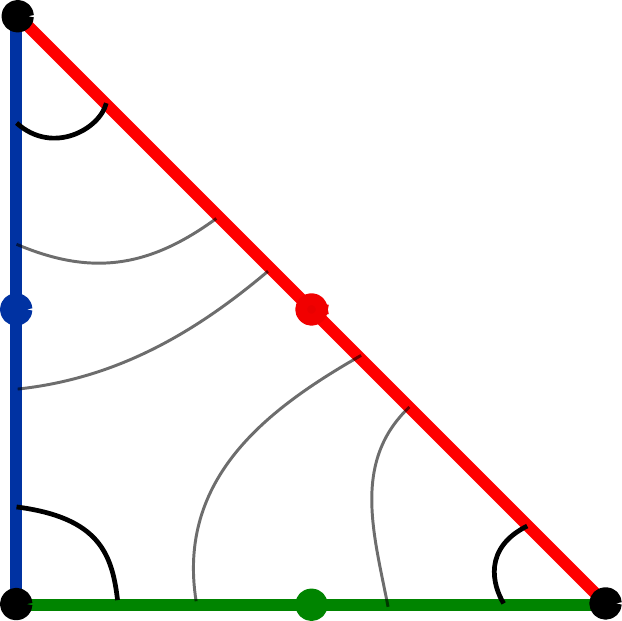}
				\caption{The three arcs in black together form the boundary component $abc$.}
				\label{fig:simple2}
			\end{figure}
			
			\item \textbf{Fact 3} : The curve $\gamma$ is uniquely determined by its intersection numbers with the three edges $A,B$ and $C$. \\
			Denote by $n_A,n_B$ and $n_C$ the numbers of intersection between the curve $\gamma$ and the three edges $A,B$ and $C$. First notice that since the two half-edges of an edge are identified, the three integers $n_A,n_B$ and $n_C$ must be even. Then, remark that these numbers are also the number of arcs having an endpoint on $A,B$ and $C$. By fact 2, we know that there exists an edge such that every arc has an endpoint on it. Hence we deduce that either $n_C=n_A+n_B$, or $n_B=n_A+n_C$ or $n_A=n_B+n_C$. Without loss of generality, let us now suppose that $n_C=n_A+n_B$. In this case, every intersection point on $C$ must be linked by an arc to an intersection point on $A\cup B$. There is only one way of pairing intersection points on $C$ and on $A\cup B$ without creating self-intersection, and this determines the curve $\gamma$. 
			\begin{figure}[h]
				\centering
				\labellist
				\small\hair 2pt
				\pinlabel {\textbf{Case $n_C=n_A+n_B$ with :}} [c] at 370 275
				\pinlabel{$n_A=4$} [c] at 350 225
				\pinlabel{$n_B=2$} [c] at 350 185
				\pinlabel{$n_C=6$} [c] at 350 145
				\pinlabel {$\color{green} A$} [c] at 140 -25
				\pinlabel {$\color{blue} B$} [c] at -30 150
				\pinlabel {$\color{red} C$} [c] at 180 180
				\endlabellist
				\includegraphics[width=3cm]{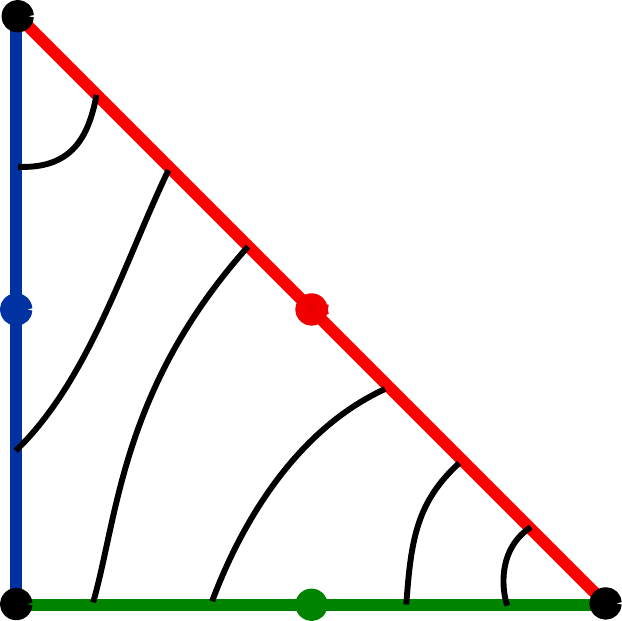}
				\caption{Pairing the endpoints on the edge $C$ with the endpoints on the edge $A\cup B$.}
				\label{fig:simple3}
			\end{figure}
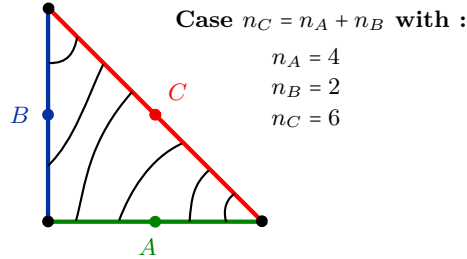
			
			\item \textbf{Fact 4} : After homotopy, the curve $\gamma$ can be lifted to a simple closed curve on the torus $\R^2 / 2\Z^2$. \\
			Suppose that we are in the case where $n_C=n_A+n_B$ (exchange the role of $A,B$ and $C$ for the other cases). Now consider the triangle obtained as a reflection of the fundamental triangle across the center point of the edge $C$. The union of the two triangles forms a square with the opposite sides identified, thus a torus. Let us justify that any lift of the curve $\gamma$ in the torus is a simple closed curve on the torus. Choose a basepoint $x$ on the curve $\gamma$, then it has two lifts on the torus, one in each triangle, and let us choose one, $\t{x}$. Notice that by following the curve $\gamma$ from $x$, its lift $\t{\gamma}$ from $\t{x}$ changes triangle each time $\gamma$ meets an edge of the fundamental triangle. This happens exactly $\frac{n_A}{2}+\frac{n_B}{2}+\frac{n_C}{2}$ times, and since $\frac{n_A}{2}+\frac{n_B}{2}+\frac{n_C}{2}=n_C$ is even, this means that the endpoint of $\t{\gamma}$ lies in the same triangle as $\t{x}$, therefore is $\t{x}$. Thus the lift $\t{\gamma}$ is a closed curve, and it is simple because $\gamma$ is. Finally, the intersection numbers between the lift $\t{\gamma}$ of the curve $\gamma$ and the sides of the square are given by the numbers $\frac{n_A}{2},\frac{n_B}{2}$. Therefore, $\frac{n_A}{2}$ and $\frac{n_B}{2}$ are relatively prime and $\gamma$ can be lifted to a line of slope $\frac{n_A}{n_B}$ in the plane. This allows us to define the map $\slope$ from $\simple$ to $\Q \cup \infty$, such that $\slope(\gamma)=\frac{n_A}{n_B}$. See figure \ref{fig:simple4} on page \pageref{fig:simple4} for a picture of the two lifts in the torus. \\
		
			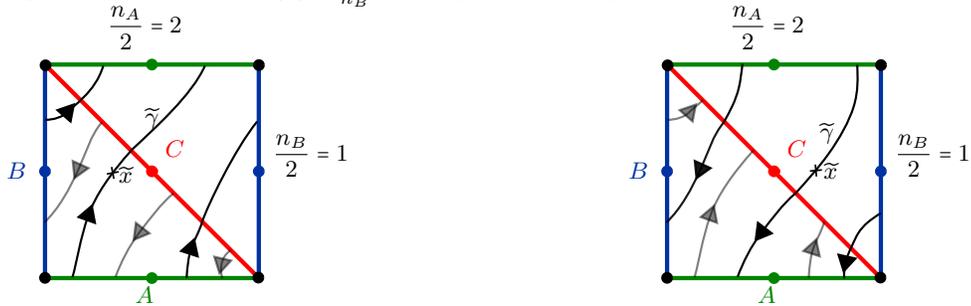
\begin{figure}[h!]
				\centering
				\begin{subfigure}{0.4 \textwidth}
					\centering
					\labellist
					\small\hair 2pt
					\pinlabel{$\displaystyle \frac{n_A}{2}=2$} [c] at 140 340
					\pinlabel{$\displaystyle \frac{n_B}{2}=1$} [c] at 360 170
					\pinlabel {$\color{green} A$} [c] at 140 -15
					\pinlabel {$\color{blue} B$} [c] at -30 150
					\pinlabel {$\color{red} C$} [c] at 180 180
					\pinlabel{$\t{x}$} at 115 145
					\pinlabel{$\t{\gamma}$} at 150 220
					\endlabellist
					\includegraphics[width=3cm]{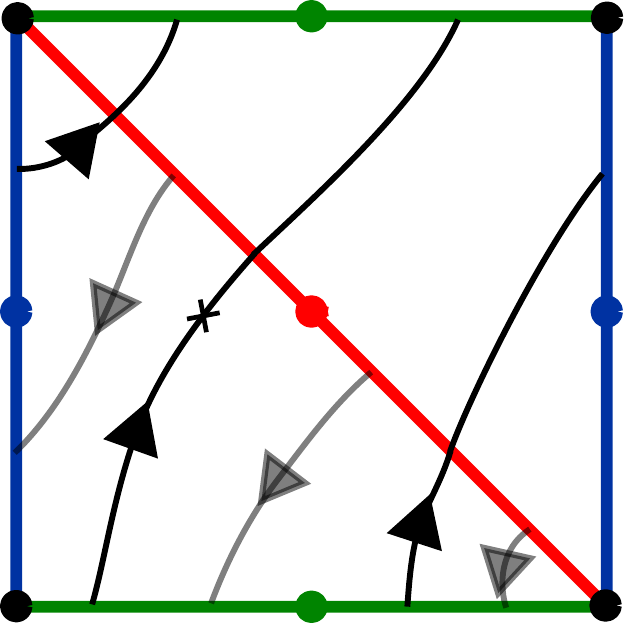}
				\end{subfigure} \hspace{1cm}
				\begin{subfigure}{0.4 \textwidth}
					\centering
					\labellist
					\small\hair 2pt
					\pinlabel{$\displaystyle \frac{n_A}{2}=2$} [c] at 140 340
					\pinlabel{$\displaystyle \frac{n_B}{2}=1$} [c] at 360 170
					\pinlabel {$\color{green} A$} [c] at 140 -15
					\pinlabel {$\color{blue} B$} [c] at -30 150
					\pinlabel {$\color{red} C$} [c] at 180 180
					\pinlabel{$\t{x}$} at 225 150
					\pinlabel{$\t{\gamma}$} at 220 200
					\endlabellist
					\includegraphics[width=3cm]{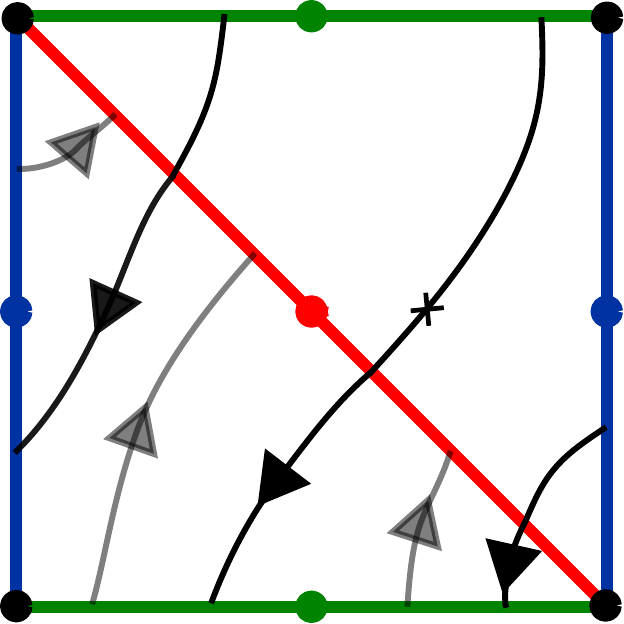}
				\end{subfigure}
				\caption{The two lifts (in black) of the curve $\gamma$ (in black and gray) in the torus.}
				\label{fig:simple4}
			\end{figure}
			
		\end{itemize}
	\end{proof}
	
	\begin{Remark} \label{link with F2} The proof of the previous proposition allows us to have a better understanding of the combinatorics of simple closed curves in $\sph$. In particular we deduce the following : \\
		Let $\gamma \in \simple$. Then, if $\gamma$ is cyclically reduced, one out of two letters in $\gamma$ is, up to inversion, always the same. If this letter is $c$ (resp. $a$, resp. $b$), then $\slope(\gamma) \in [0,\infty]$ (resp.  $\slope(\gamma) \in [-\infty,-1]$, resp. $\slope(\gamma) \in [-1,0]$). Moreover, consider the map $\vphi$ from $\F_3$ to $\F_2$ which delete the appearances of $c^\pm$ (resp. $a^\pm$, $b^\pm$) and forget the sign of the two other letters. More precisely, if $u = s_1 \hdots s_p$, with $s_i \in \{a,a^{-1},b,b^{-1},c,c^{-1} \}$, then $\vphi$ is such that $\vphi(u)=\vphi(s_1)\hdots \vphi(s_p)$, $\vphi(c^\pm)=1$ (resp. $\vphi(a^\pm)=1$, resp. $\vphi(b^\pm)=1$), $\vphi(a^\pm)=a$ and $\vphi(b^\pm)=b$ (resp. $\vphi(b^\pm)=b$ and $\vphi(c^\pm)=c$, resp. $\vphi(a^\pm)=a$ and $\vphi(c^\pm)=c)$. Then $\vphi(\gamma)$ is a primitive element of $\F_2$, and its slope as an element in $\F_2$ is the same as the slope of~$\gamma$~: $\slope(\gamma)=\slope(\vphi(\gamma))$.  Moreover, word length of $\vphi(\gamma)$ is half the word length of $\gamma$.
	\end{Remark}

	\begin{Definition}[Special-lengths $l_i(\gamma)$] \label{def:magic-length}
		Let $\gamma \in \simple$ be a cyclically reduced simple closed curve on $S_{0,4}$ and consider $\slope(\gamma)= [n_1,\dots, n_r]$ the continued fraction expansion of the slope of $\gamma$. For every $1 \leq i \leq r$, consider $\gamma_i  \in \simple$, respectively $\gamma'_i \in \simple$, a cyclically reduced simple closed curve of slope $[n_1,\dots,n_i]$, respectively of slope $[n_1,\dots,n_i+1]$. For $i=0$, consider $\gamma_0 \in \simple$ of slope 0 and $\gamma'_0$ of slope 1.\\
		Then, for $0\leq i \leq r$, we define $l_i(\gamma):=\frac{1}{2}\Vert \gamma_i \Vert$ and $l'_i(\gamma):=\frac{1}{2}\Vert \gamma'_i \Vert$. 
	\end{Definition}
	
		By the previous remark (Remark \ref{link with F2}), $l_i(\gamma)$ and $l'_i(\gamma)$ are exactly the lengths of  primitive words in $\F_2$ of slope $[n_1,\dots,n_i]$ and $[n_1,\dots,n_i+1]$. Moreover, we can provide the following formulas on the lengths $l'_i(\gamma),l_i(\gamma)$ (here after we simply denote $l_i,l'_i$ to reduce notations). For details, we refer the reader to section 2.1 of \cite{schlich_equivalence_2024} and in particular Remark 2.4.
		\begin{align}
		& \forall 1 \leq i \leq r, & &  l_i=(n_i-1)l_{i-1}+l'_{i-1}, & \label{rec-li}\\ 
		 &\forall 1 \leq i \leq r, & & l'_i=l_i + l_{i-1},& \label{rec-l'i}\\ 
		 &\forall 0 \leq i \leq r, & &  l_i \leq l'_i,& \label{li<l'i}\\
		 &\forall 1 \leq i \leq r,  & &  l_{i-1} \leq l_i, & \label{l{i-1<li}}\\
		 &\forall 1 \leq i \leq r,  & &  l'_i \leq 2l_i, & \label{l'i<2li}\\ 
		 & \forall 2 \leq i \leq r \text{ and for } i=1 \text{ with } n_1 \geq 1, & &  l'_i+1 \leq 2l_i,& \label{l'i+1<2li} \\
		 & \forall 0 \leq i \leq r, & & i \leq l_i. \label{i<li}
		\end{align}

	\begin{Remark} \label{change-basis}
		Notice that the map $\slope$ depends on the choice of a basis $\{a,b,c\}$ for $\groupe= \nolinebreak \F_3$. The choice of another basis leads to another map $\slope$. However, we can link the slope map of two different bases. Denote $\slope_{a,b,c}$ the slope map in the basis $\{a,b,c\}$. \\
		Let $\gamma \in \simple$ be such that $\slope_{a,b,c}(\gamma)=[n_1,\hdots,n_r]$. Fix $1 \leq i < r$ and let $a',b',c'$ be another basis of $\F_3$ such that $\slope_{a,b,c}(c'b')=[n_1,\hdots,n_i]$ and $\slope_{a,b,c}(c'a')=[n_1,\hdots,n_i+1]$. Then $\slope_{a',b',c'}(\gamma)=[n_{i+1},\hdots,n_r]$. 
	\end{Remark}
	
	The $\slope$ map only depends on the class of an element up to conjugacy and inversion. Therefore, for an element $\vphi$ in the mapping class group $\mcg$, even if the element $\vphi(\gamma)$ is defined only up to conjugacy, the rational $\slope(\vphi(\gamma))$ is well-defined. 
	
	\begin{Lemma} \label{exists-phi}
		Let $\F_3=\groupe=\la a,b,c \ra$. Let $\frac{p}{q}$ and $ \frac{p'}{q'}$ be two rational numbers such that $|pq'-p'q|= 1$. Then there exists a mapping class $\vphi \in \mcg$ such that $\slope(\vphi(cb))=\frac{p}{q}$ and $\slope(\vphi(ca))=\frac{p'}{q'}$. 
	\end{Lemma}
	
	\begin{proof}
		First notice that the hypothesis $|pq'-p'q|= 1$ means that the two vectors $(q,p)$ and $(q',p')$ of $\Z^2$ form a basis of $\Z^2$. This implies that the matrix $M=\begin{pmatrix} q & q' \\ p & p' \end{pmatrix}$ belongs to $\mathrm{SL}^{\pm}(2,\Z)$, so $M$ preserve the lattice $\Z^2$ inside $\R^2$. The action of $(2\Z^2,\pm)$ on $\R^2 \setminus \Z^2$ is $M$ equivariant, so $M$ induces a homeomorphism of the quotient $(\R^2 \setminus \Z^2) / (2\Z^2,\pm)$, which is the sphere $\sph$. The line $l_x$ from $x$ to $x+(2,0)$ in $\R^2 \setminus \Z^2$ (here $x$ is an arbitrary point in $\R^2$ such that $l_x$ avoids the lattice $\Z^2$) induces a simple closed curve $\gamma_{\frac{0}{1}}$ in $\sph$ represented by $cb$ in $\groupe$. This line is sent by $M$ to the line from $M(x)$ to $M(x)+2(q,p)$ which induces a simple closed curve $\gamma_{\frac{p}{q}}$ in $\sph$ of slope~$\frac{p}{q}$. Similarly, the line from 
		$x$ to $x+(0,2)$ in $\R^2 \setminus \Z^2$ induces a simple closed curve $\gamma_{\frac{1}{0}}$ in $\sph$ represented by $ca$ in $\groupe$. This line is sent by $M$ to the line from $M(x)$ to $M(x)+2(q',p')$ which induces a simple closed curve $\gamma_{\frac{p'}{q'}}$ in $\sph$ of slope $\frac{p'}{q'}$. Hence $\vphi=M_*$ is in the mapping class group of $\sph$ and satisfies the conditions required by the lemma.
	\end{proof}

	\subsection{Constructing simple closed curves} 

	\label{construc-simple}
	In this section, we try to have a better understanding of the structure of the elements of $\simple$.
	\begin{Lemma} \label{slope-n1}
		Let $\gamma \in \simple$. Suppose that $\slope(\gamma)\geq0$. Denote $n=\lf \slope(\gamma) \rf \in \N$. Then, 
		\begin{itemize}
			\item If $n$ is even, $\gamma$ can be written (up to conjugation) as a concatenation of subwords of the form~:
			\begin{align*}
				w(\t{c},\t{b}) & =(ca)^{\frac{n}{2}}\t{c} (ca)^{-\frac{n}{2}}\t{b} \\
				w'(\t{c},\t{b},0) & =(ca)^{\frac{n}{2}}\t{c} (ca)^{-\frac{n}{2}-1}\t{b} \\
				w'(\t{c},\t{b},1) & =(ca)^{\frac{n}{2}+1}\t{c} (ca)^{-\frac{n}{2}}\t{b}
			\end{align*}
			where $\t{c} \in \{c,c^{-1} \}, \t{b} \in \{b,b^{-1}\}$.
			\item If $n$ is odd, $\gamma$ can be written (up to conjugation) as a concatenation of subwords of the form~:
			\begin{align*}
				w(\t{c},\t{b},0) & =(ca)^{\frac{n-1}{2}}\t{c}(ca)^{-\frac{n+1}{2}}\t{b} \\
				w(\t{c},\t{b},1) & =(ca)^{\frac{n+1}{2}}\t{c}(ca)^{-\frac{n-1}{2}}\t{b} \\
				w'(\t{c},\t{b}) & =(ca)^{\frac{n+1}{2}}\t{c}(ca)^{-\frac{n+1}{2}}\t{b}
			\end{align*}
			where $\t{c} \in \{c,c^{-1} \}, \t{b} \in \{b,b^{-1}\}$.    
		\end{itemize}
	\end{Lemma}

	Before starting the proof, recall that we can see the sphere $\sph$ as the quotient of the square $[0,2]^2 \setminus (\Z^2 \cap [0,2]^2)$ by the reflection across the center point of the square $(1,1)$. Thus we can write the word in the fundamental group corresponding to a curve by following the curve and record an $a$ (resp. a $b$, resp. a $c$) each time it crosses a horizontal side (resp. a vertical side, resp. the diagonal from $(0,2)$ to $(2,0)$) of the square, with sign $\pm 1$ depending on whether or not the orientation of the line coincides with the transverse orientation on the corresponding half-side (resp. half-diagonal). If the curve is simple of slope $\frac{p}{q}$, Proposition \ref{essential->Q} (and its proof) shows us that it suffices to follow a straight line of slope $\frac{p}{q}$ in the quotient of the plane $\R^2$ by $2\Z^2$. 
	
	\begin{proof}
		First notice that is $\slope(\gamma)=0$, then up to inversion and conjugation, $\gamma=cb$, which is precisely the word $w(\t{c},\t{b})$ defined in the Lemma (here $n=0$).
		
		Now consider a line of slope $\frac{p}{q}=\slope(\gamma) > 0$ in the plane. The proof is illustrated on figure \ref{fig:line_square} on page \pageref{fig:line_square}. We can assume that this line avoids the points of the lattice $\Z^2$. In the quotient by $2\Z^2$, this line is a collection of parallel segments of slope $\frac{p}{q}$ joining two sides of the square. The identification between two opposite sides of the square gives an ordering on the segments.  Now we want to understand what can be read between two $b^{\pm 1}$. Thus look at a collection of successive segments such that the first segment starts on the left vertical side of the square, the last segment ends on the right vertical side of the square, and no other segment starts or ends on a vertical side. A picture is drawn on figure \ref{fig:line_square} on page \pageref{fig:line_square}. Since the slope of the line is $\frac{p}{q}$, between two $b$ we need to read $n$ or $n+1$ times the letter $a$ or $a^{-1}$, with $n=\lf \frac{p}{q} \rf$. Moreover, since the slope is positive, it imposes that  before the letter $a$ stands necessarily the letter $c$  and after the letter $a^{-1}$ stands necessarily the letter $c^{-1}$. Because of the ordering of the successive segments, we easily deduce that a subword of $\gamma$ read between two letters $b^\pm$ is necessarily of the form $(ca)^{n_1}\t{c}(ca)^{-n_2}\t{b}$, with $n_1+n_2 \in \{n,n+1\}$.  Finally, the intersection points between the curve and the horizontal side of the square are evenly spaced then we must also have $|n_1-n_2| \leq 1$. Thus :
		
		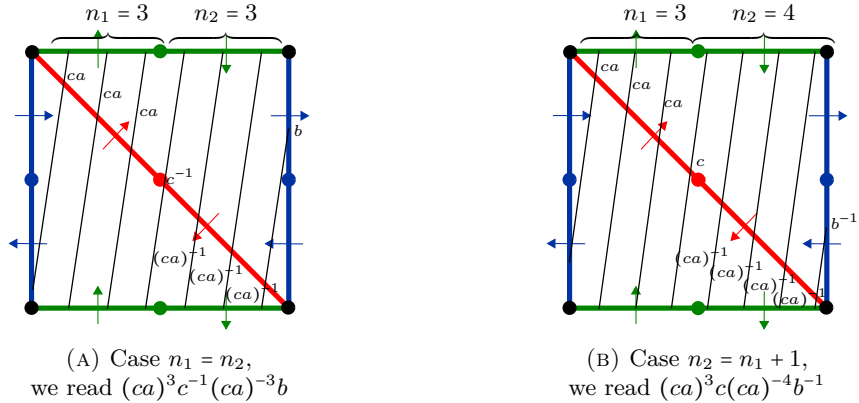
\begin{figure}[h!]
			\centering
			\begin{subfigure}{0.4 \textwidth}
				\centering
				\labellist
				\small\hair 2pt
				\pinlabel{$n_1=3$} [c] at 120 350
				\pinlabel{$\overbrace{\hspace{1.5cm}}$} [c] at 110 320
				\pinlabel{$n_2=3$} [c] at 240 350
				\pinlabel{$\overbrace{\hspace{1.5cm}}$} [c] at 240 320
				\pinlabel{$\scriptstyle ca$} at 75 280
				\pinlabel{$\scriptstyle ca$} [c] at 115 260
				\pinlabel{$\scriptstyle ca$} at 155 240
				\pinlabel{$\scriptstyle c^{-1}$} at 190 170
				\pinlabel{$\scriptstyle (ca)^{-1}$} at 190 80
				\pinlabel{$\scriptstyle (ca)^{-1}$} at 230 60
				\pinlabel{$\scriptstyle (ca)^{-1}$} at 270 40
				\pinlabel{$\scriptstyle b$} at 320 220
				\endlabellist
				\includegraphics[width=4cm]{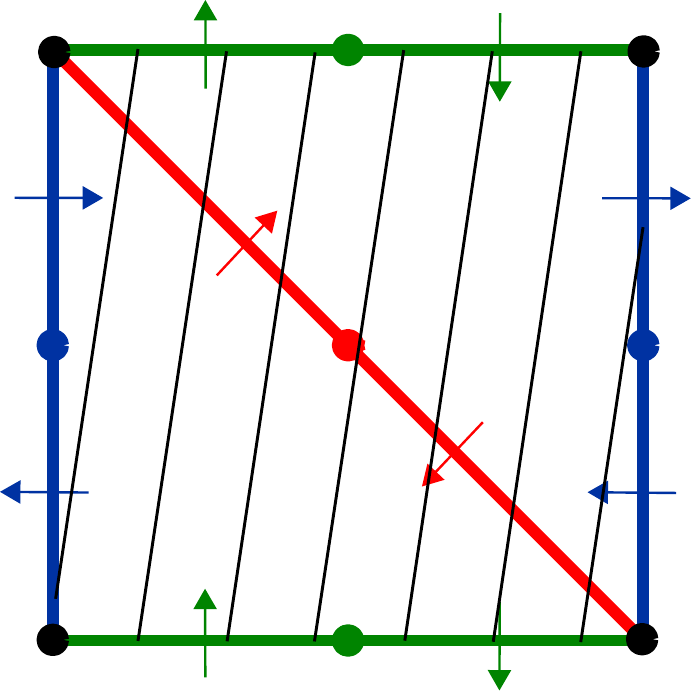}
				\caption{Case $n_1=n_2$,\\ we read $(ca)^3c^{-1}(ca)^{-3}b$}
				\label{fig:n1=n2}
			\end{subfigure}
			\begin{subfigure}{0.4 \textwidth}
				\centering
				\labellist
				\small\hair 2pt
				\pinlabel{$n_1=3$} [c] at 120 350
				\pinlabel{$\overbrace{\hspace{1.5cm}}$} [c] at 100 320
				\pinlabel{$n_2=4$} [c] at 240 350
				\pinlabel{$\overbrace{\hspace{1.9cm}}$} [c] at 240 320
				\pinlabel{$\scriptstyle ca$} at 70 290
				\pinlabel{$\scriptstyle ca$} [c] at 105 270
				\pinlabel{$\scriptstyle ca$} at 140 250
				\pinlabel{$\scriptstyle c$} at 170 185
				\pinlabel{$\scriptstyle (ca)^{-1}$} at 172 80
				\pinlabel{$\scriptstyle (ca)^{-1}$} at 210 65
				\pinlabel{$\scriptstyle (ca)^{-1}$} at 247 50
				\pinlabel{$\scriptstyle (ca)^{-1}$} at 280 35
				\pinlabel{$\scriptstyle b^{-1}$} at 330 120
				\endlabellist
				\includegraphics[width=4cm]{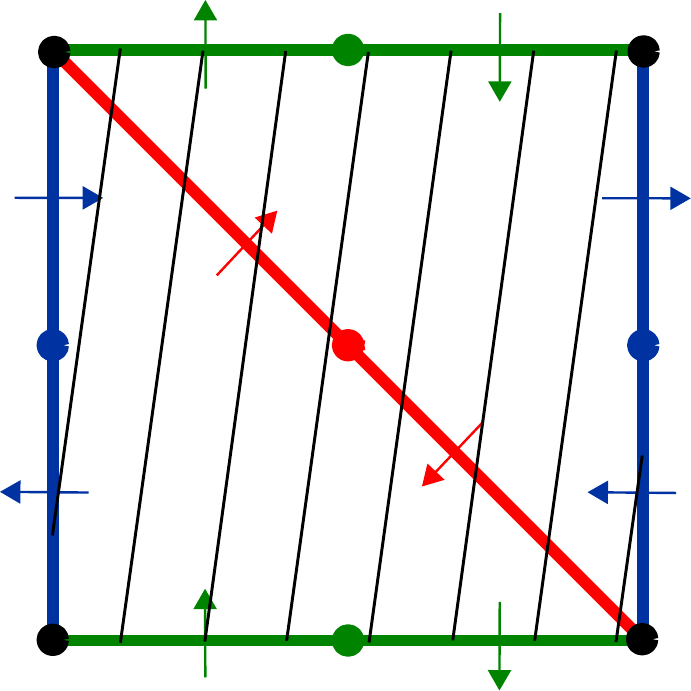}
				\caption{Case $n_2=n_1+1$,\\ we read $(ca)^3c(ca)^{-4}b^{-1}$}
				\label{fig:n2>n1}
			\end{subfigure}
			\caption{A line in the square between two intersections with the vertical side}
			\label{fig:line_square}
		\end{figure}
		
		\begin{itemize}
			\item If $n$ is even :
			\begin{itemize}
				\item If $n_1+n_2=n$, then necessarily $n_1=n_2=\frac{n}{2}$, so we recover the subword $w(\t{c},\t{b})$ defined in the lemma.
				\item If $n_1+n_2=n+1$, then necessarily $\{n_1,n_2\}=\{\frac{n}{2},\frac{n}{2}+1\}$, so we recover the subwords $w'(\t{c},\t{b},0)$ and $w'(\t{c},\t{b},1)$ defined in the lemma.\\
			\end{itemize}
			\item If $n$ is odd :
			\begin{itemize}
				\item If $n_1+n_2=n$, then necessarily $\{n_1,n_2\}=\{\frac{n-1}{2},\frac{n+1}{2}\}$, so we recover the subwords $w(\t{c},\t{b},0)$ and $w(\t{c},\t{b},1)$ defined in the lemma.
				\item If $n_1+n_2=n+1$, then necessarily $n_1=n_2=\frac{n+1}{2}$, so we recover the subword $w'(\t{c},\t{b})$ defined in the lemma.
			\end{itemize}
		\end{itemize}
	\end{proof}
	
	Let us now specify the previous result in the case where the slope is an integer.
	\begin{Lemma} \label{slope-[n]}
		Let $\gamma \in \simple$. Suppose that $\slope(\gamma) = [n]$, with $n \in \N$. Then (up to conjugation and inversion) :
		\begin{equation*}
			\gamma =  \left\{ \begin{array}{cc}
				(ca)^{\frac{n}{2}}c(ca)^{-\frac{n}{2}}b  & \text{ if $n$ is even} \\
				(ca)^{\frac{n+1}{2}}c^{-1}(ca)^{-\frac{n-1}{2}}b & \text{ if $n$ is odd}
			\end{array} \right.
		\end{equation*}
	\end{Lemma}
	
	\begin{proof}
		Since the slope of $\gamma$ is an integer $n\in \N$, the letter $b$ (or $b^{-1}$) appears exactly once in the word $\gamma$. Thus, by Lemma \ref{slope-n1}, the word $\gamma$ (up to cyclic permutation) is one of the 24 words given in the statement of the lemma. Moreover, since the slope is $n$, the letter $a^{\pm 1}$ appears exactly $n$ times. Then $\gamma$ must be one of the following words : $w(\t{c},\t{b}),w(\t{c},\t{b},0),w(\t{c},\t{b},1)$.  
		\begin{itemize}
			\item Suppose that $n$ is even, then $\gamma$ is (up to permutation and inversion) of the form : $w(\t{c},\t{b})=(ca)^{\frac{n}{2}}\t{c}(ca)^{-\frac{n}{2}}\t{b}$. Up to taking the inverse, we can assume that the power on $b$ is $+1$. It remains to determine the power of the letter $c$ in the middle of the word.  Recall that we can see the curve $\gamma$ as a straight line of slope $n$ in the square $[0,2]$, that is, successive parallel segments evenly spaced of slope $n$. It is easy to check on a drawing that if the power on $b$ is $+1$ (which means that the first segment starts on the upper-half of the left vertical side of the square and that the last segment ends on the upper-half of the right vertical side of the square), then the ($\frac{n}{2}+1$)-th segment in the square must cross the diagonal in its first half. This means that the power on the letter $c$ in the middle is $+1$, hence $\gamma=(ca)^{\frac{n}{2}}c(ca)^{-\frac{n}{2}}b$ (up to permutation and inversion).
			
			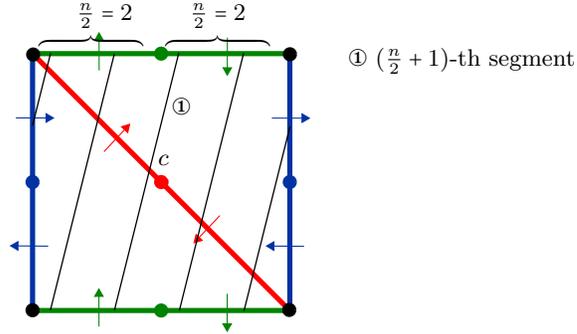
\begin{figure}[h]
				\centering
				\labellist
				\small \hair 2pt
				\pinlabel{$\frac{n}{2}=2$} [c] at 105 350
				\pinlabel{$\overbrace{\hspace{1.4cm}}$} [c] at 90 320
				\pinlabel{$\frac{n}{2}=2$} [c] at 230 350
				\pinlabel{$\overbrace{\hspace{1.4cm}}$} [c] at 230 320
				\pinlabel{$c$} at 170 190
				\pinlabel{\1} at 190 250
				\pinlabel{\1 ($\frac{n}{2}+1$)-th segment} at 500 300
				\endlabellist
				\includegraphics[width=4cm]{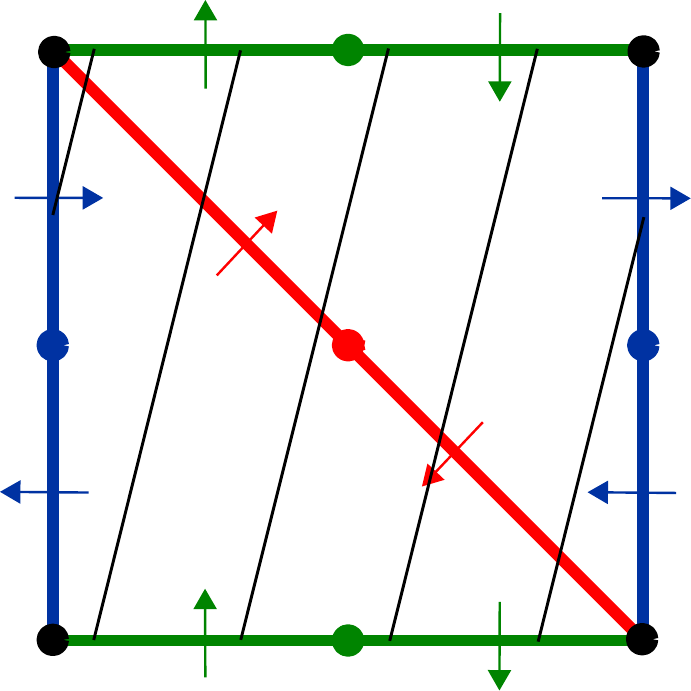}
				\caption{A line of slope $n=4$ in the square}
				\label{fig:slope-even}
			\end{figure}
			
			\item Suppose that $n$ is odd, then $\gamma$ is (up to permutation and inversion) of the form $w(\t{c},\t{b},0)$ or $w(\t{c},\t{b},1)$. Up to taking the inverse, we can assume that the power on $b$ is $+1$. It remains to check that the power on $(ca)$ is $\frac{n+1}{2}$ and that the power on the letter $c$ in the middle is $-1$. Notice that since the first segment hits the left vertical side of the square in its top half, then it forces the $(\frac{n+1}{2})$-th segment to hit the top horizontal side in its first half and the $(\frac{n+1}{2}+1)$-th segment to hit the diagonal in its second half. At the word level, this exactly means $\gamma=(ca)^{\frac{n+1}{2}}c^{-1}(ca)^{-\frac{n-1}{2}}b$ (up to permutation and inversion).
			
			\begin{figure}[h]
				\centering
				\labellist
				\small \hair 2pt
				\pinlabel{$\frac{n+1}{2}=3$} [c] at 115 350
				\pinlabel{$\overbrace{\hspace{1.8cm}}$} [c] at 105 320
				\pinlabel{$\frac{n-1}{2}=2$} [c] at 240 350
				\pinlabel{$\overbrace{\hspace{1.4cm}}$} [c] at 250 320
				\pinlabel{$c^{-1}$} at 205 160
				\pinlabel{\1} at 170 250
				\pinlabel{\1 ($\frac{n+1}{2}$)-th segment} at 520 300
				\pinlabel{\2} at 223 230
				\pinlabel{\2 ($\frac{n+1}{2}+1$)-th segment} at 540 260
				\pinlabel{$h$} at -10 150 
				\endlabellist
				\includegraphics[width=4cm]{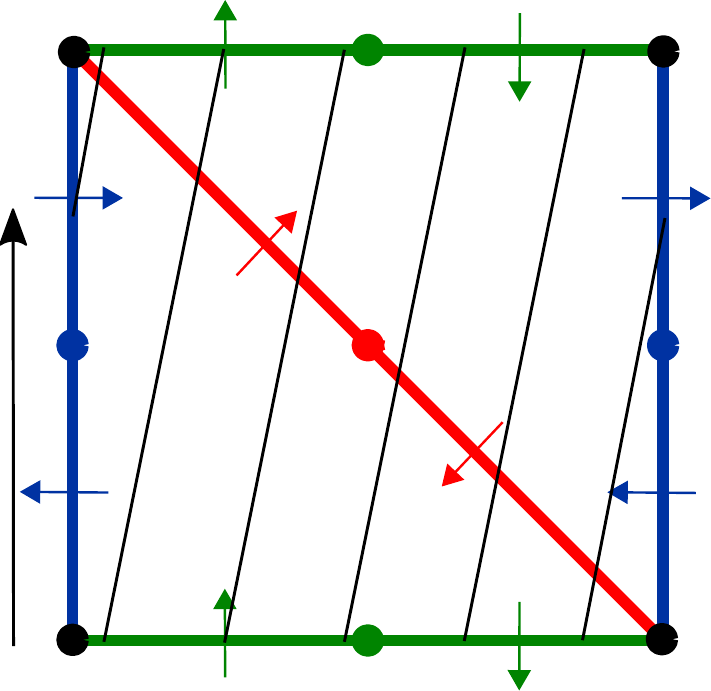}
				\caption{A line of slope $n=5$ in the square}
				\label{fig:slope-odd}
			\end{figure}
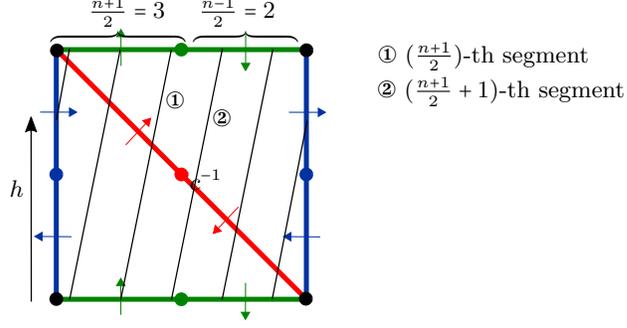
		\end{itemize}
	\end{proof}

	\begin{Corollary} \label{gamma^n-simple}
		Let $\gamma \in \pi_1(\sph)$ be a cyclically reduced element representing a simple closed curve on $\sph$. There exists two words $\delta_1,\delta_2 \in \groupe$ such that for all $n \in \N$, the element $\gamma^n\delta_1\gamma^{-n}\delta_2$ belongs to $\simple$.
	\end{Corollary}
	
	\begin{proof}
		First notice that if the Lemma is true for some $\gamma$, it is still true for the inverse of $\gamma$ and for a conjugate of $\gamma$. Hence, since any two elements of $\F_3$ (induced by simple closed curves) with the same slope are equal up to conjugacy and inversion, it is sufficient to show the lemma for any element with the same slope as $\gamma$. \\  
		Let $\frac{p}{q}=\slope(\gamma)$. There exists another rational $\frac{p'}{q'}$ such that $|pq'-p'q|=\pm 1$. Now use Lemma \ref{exists-phi} to show the existence of a mapping class $\vphi \in \mcg$ such that $\slope(\vphi(ca))=\frac{p}{q}$ and $\slope(\vphi(cb))=\frac{p'}{q'}$. Choose a representative of $\vphi$ in $\mathrm{Homeo}(\sph)$ and with a slight abuse of notation still denote it by $\vphi$. Let $\eta=\vphi(ca)$. Then $\eta \in \simple$ and $\slope(\eta)=\slope(\gamma)$. Moreover, by Lemma \ref{slope-[n]}, we deduce that for all $n \in \N$, $(ca)^nc(ca)^{-n}b \in \simple$. Then, since $\vphi \in \mathrm{Homeo}(\sph)$, we also obtain that $\vphi((ca)^nc(ca)^{-n}b) \in \simple$ for all $n \in \N$. But we have $\vphi((ca)^nc(ca)^{-n}b)=\eta^n \vphi(c)\eta^{-n}\vphi(b)$ hence the corollary with $\delta_1=\vphi(c)$ and $\delta_2=\vphi(b)$.
	\end{proof}
	
	We now end this section by writing any element $\gamma$ in $\simple$ as a concatenation of "building blocks", each of them using some "approximation" $\gamma_i$ of $\gamma$. We can approximate $\gamma$ at different levels, indexed by the integer $i$, corresponding to the successive steps in the continued fraction expansion of the slope of $\gamma$.
	
	\begin{Lemma} \label{general-form-gamma}
		Let $i \geq 1$ and let $[n_1,\hdots,n_i]$ be the continued fraction expansion of a rational number. Then there exists a simple word $\gamma_i \in \simple$ of slope $[n_1,\hdots,n_i]$, and two words $\delta_1, \delta_2 \in \groupe$, such that every simple word $\gamma \in \simple$ with continued fraction expansion $[n_1,\hdots,n_r]$, with $r >i$, can be written as a (cyclic-permutation of a) concatenation of subwords of the form :
		\begin{equation*}
			(\gamma_i)^{m_1}\t{\delta_1} (\gamma_i)^{-m_2}\t{\delta_2}
		\end{equation*}
		with $m_1,m_2 \in \{ \lf \frac{n_{i+1}}{2} \rf, \lf \frac{n_{i+1}}{2} \rf +1 \}$ and $\t{\delta_1} \in \{\delta_1,\delta_1^{-1}\}, \t{\delta_2} \in \{\delta_2,\delta_2^{-1} \}$.\\
		Moreover, if $r=i+1$, that is $\slope(\gamma)=[n_1,\hdots,n_i,n_r]$, then (up to conjugation and inversion) :
		\begin{equation*}
			\gamma =  \left\{ \begin{array}{cc}
				(\gamma_i)^{\frac{n_r}{2}}\delta_1(\gamma_i)^{-\frac{n_r}{2}}\delta_2  & \text{ if $n_r$ is even} \\
				(\gamma_i)^{\frac{n_r+1}{2}}\delta_1^{-1}(\gamma_i)^{-\frac{n_r-1}{2}}\delta_2 & \text{ if $n_r$ is odd}
			\end{array} \right.
		\end{equation*}
	\end{Lemma}
	
	\begin{proof}
		By permuting the three elements of the basis, $a,b,c$, we can assume that the slope of $\gamma$ is non-negative. \\ 
		
		We start by using Lemma \ref{exists-phi} to obtain the existence of a mapping class $\vphi \in \mcg$ such that $\slope(\vphi(ca))=[n_1,\hdots,n_i]$ and $\slope(\vphi_i(cb))=[n_1,\hdots,n_i+1]$. Choose a representative of $\vphi  \in \mcg \subset \out$ and with a slight abuse of notation still denote it by $\vphi \in \aut$. Denote $a'=\vphi_i(a),b'=\vphi_i(b),c'=\vphi_i(c)$ and $\gamma_i=\vphi(ca)$. Then $a',b',c'$ is a basis of $\F_3$ since $\vphi$ is an automorphism of $\F_3$ and in this new basis the slope of $\gamma$ is $[n_{i+1},\hdots,n_r]$ (see Remark \ref{change-basis}). Hence, by the Lemma \ref{slope-n1} applied to the new basis $\F_3=\la a',b',c' \ra$, the word $\gamma$ can be written as a concatenation of subwords of the form $(c'a')^{m_1} \t{c'} (c'a')^{-m_2} \t{b'}$, with $m_1,m_2 \in \{\frac{n_{i+1}}{2},\frac{n_{i+1}}{2}+1\}$ if $n_{i+1}$ is even, $m_1,m_2 \in \{\frac{n_{i+1}-1}{2},\frac{n_{i+1}+1}{2}\}$ if $n_{i+1}$ is odd. But we have :
		\begin{align*}
			(c'a')^{m_1} \t{c'} (c'a')^{-m_2} \t{b'} = (\vphi_i(ca))^{m_1}\t{\vphi_i(c)} (\vphi_i(ca))^{-m_2} \t{\vphi_i(b)} = (\gamma_i)^{m_1}\t{\vphi_i(c)} (\gamma_i)^{-m_2}\t{\vphi_i(b)}.
		\end{align*}
		Hence we obtain the first part of the lemma with $\delta_1=\vphi_i(c)$ and $\delta_2=\vphi_i(b)$. \\ 
		For the second part, now suppose that $r=i+1$, that is $\slope(\gamma)=[n_1,\hdots,n_i,n_r]$. Then, in the new basis $\F_3=\la a',b',c'\ra$, the slope of $\gamma$ is simply $[n_r]=n_r \in \N$ (again see Remark \ref{change-basis}). Thus we can use Lemma \ref{slope-[n]} to ensure that (up to permutation and inversion) :
		\begin{align*}
			\gamma & =  \left\{ \begin{array}{cc}
				(c'a')^{\frac{n_r}{2}}c'(c'a')^{-\frac{n_r}{2}}b'=(\vphi(ca))^{\frac{n_r}{2}}\vphi(c)(\vphi(ca))^{-\frac{n_r}{2}}\vphi(b) & \text{ if $n_r$ is even} \\
				(c'a')^{\frac{n_r+1}{2}}c'^{-1}(c'a')^{-\frac{n_r-1}{2}}b'=(\vphi(ca))^{\frac{n_r+1}{2}}\vphi(c)^{-1}(\vphi(ca))^{-\frac{n_r-1}{2}}\vphi(b) & \text{ if $n_r$ is odd}
			\end{array} \right. \\
			& = \left\{ \begin{array}{cc} (\gamma_i)^{\frac{n_r}{2}}\delta_1(\gamma_i)^{-\frac{n_r}{2}}\delta_2  & \text{ if $n_r$ is even.} \\
				(\gamma_i)^{\frac{n_r+1}{2}}\delta_1^{-1}(\gamma_i)^{-\frac{n_r-1}{2}}\delta_2 & \text{ if $n_r$ is odd.}
			\end{array} \right.
		\end{align*}
	\end{proof}
	
	\section{Study of the redundancy of subwords of simple words} \label{sec:redundancy}
	
	\subsection{Some generalities about lattices in $\R^2$} \label{lattices}
	In section \ref{magic-sphere}, we will prove the main result of this section (Theorem \ref{magic-len-S_{0,4}}), which studies the redundancy of subwords of some specific lengths in a simple word. We will adopt a geometric approach using lattices in $\R^2$. Thus, in this section, we need to set up some notations and to state a few facts about the geometry of lattices in $\R^2$ that will be useful for us in the following. 
	
	\subsubsection{Rectangles adapted to a basis of a lattice and tiling of $\R^2$}
	Let us fix $\Lambda$ a lattice in $\R^2$. We consider $\R^2$ both endowed with its usual euclidean structure and frame and with the lattice $\Lambda$. Let $(u,v)$ be a basis of $\Lambda$ and $x \in \Lambda$. \\
	In coordinates, write $x=(x_1,x_2), u=(u_1,u_2)$ and $v=(v_1,v_2)$. Let us assume that the basis $(u,v)$ satisfies $u_1>0,v_1>0$ and $u_2<0<v_2$. In particular this requires that $\displaystyle \slope(u)=\frac{u_2}{u_1} <0<\frac{v_2}{v_1}=\slope(v)$. Let us now define : 
	\begin{equation}\label{S}
		S(x,u,v)=[x_1,x_1+u_1+v_1]\times [x_2+u_2,x_2+v_2].
	\end{equation}
	
	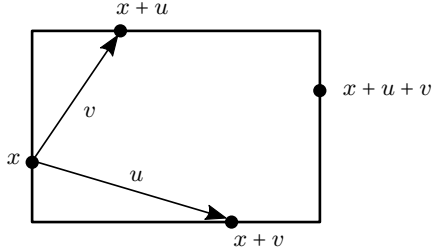
\begin{figure}[h!]
		\centering
		\labellist
		\small \hair 2pt
		\pinlabel{$x$} at -10 60 
		\pinlabel{$x+u$} at 100 190 
		\pinlabel{$x+v$} at 200 -10 
		\pinlabel{$x+u+v$} at 310 120
		\pinlabel{$u$} at 95 45
		\pinlabel{$v$} at 55 100
		\endlabellist
		\includegraphics[width=4cm]{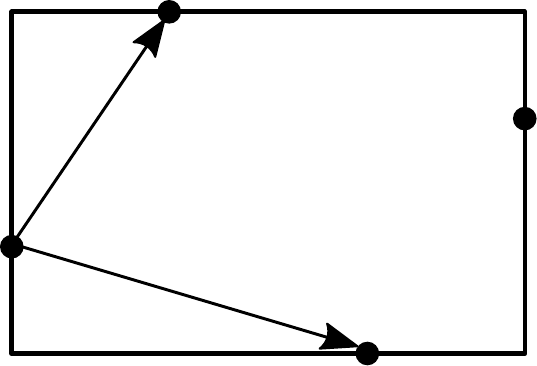}
		\caption{The rectangle $S(x,u,v)$.}
		\label{fig:S(x,u,v)}
	\end{figure}
	
	Thus, $S(x,u,v)$ is the only rectangle with horizontal and vertical sides containing on each of its sides exactly one point of the set $\{x,x+u,x+v,x+u+v \}$. Note the important fact that since $(u,v)$ is a basis of $\Lambda$, there is no other point of $\Lambda$ in $S(x,u,v)$. In particular, there is no point of~$\Lambda$ in the interior of $S(x,u,v)$. \\
	
	We will also need to consider the same square without its right vertical side. Then we choose the following notation :
	\begin{equation} \label{S*}
		S_*(x,u,v)=[x_1,x_1+u_1+v_1[\times [x_2+u_2,x_2+v_2].
	\end{equation}

	In the same way that we had previously defined the type of a point of $\Z^2$ as its equivalence class modulo $2\Z^2$, we can now define the \emph{type} of a point in $\Lambda$ as its equivalence class modulo $2\Lambda$. Hence they are $4$ types of points. Note that the four points $x,x+u,x+v$ and $x+u+v$ of the lattice $\Lambda$ are all of different types. If $t$ is a type, we define the $(u,v)$-\emph{opposite type} of $t$ as the type of the point $x+u+v$, where $x$ is any point of type $t$ (this definition does not depend on the choice of the point $x$ of type $t$). We denote by $\Lambda(t)$ the set of points in $\Lambda$ of type $t$. \\
	
	We now want to cover the plane $\R^2$ with the rectangles $S(x,u,v)$, for $x \in \Lambda(t) \cup \Lambda(\inv{t})$. The situation will not be the same according to the signs of $\slope(u+v)$ and $\slope(v-u)$. 
	When $\slope(u+v)$ and $\slope(v-u)$ are of the same sign, the rectangles $S(x,u,v)$, for $x \in \Lambda(t)\cup \Lambda(\inv{t})$ cover the whole $\R^2$. However, there is some overlap. This is the purpose of the following lemma :
	\begin{Lemma} \label{tile-same-sign}
		Assume that $\slope(u+v) \neq 0, \slope(v-u) \neq + \infty$ and $\slope(u+v)$ and $\slope(v-u)$ are of the same sign. Let $t$ be a type and $\inv{t}$ its $(u,v)$-opposite type. Then :
		\begin{equation*}
			\R^2=\underset{x \in \Lambda(t) \cup \Lambda(\inv{t})} \bigcup S(x,u,v)
		\end{equation*}
		Moreover,
		\begin{itemize}
			\item if $x,x' \in \Lambda$ are of $(u,v)$-opposite type, then $S(x,u,v) \cap S(x',u,v) \neq \emptyset$ if and only if $x'=x\pm (v-u)$ or $x'=x \pm (u+v)$ and in this case $S(x,u,v)$ and $S(x',u,v)$ only intersect along one of their sides. Thus, for all $x \in \Lambda(t),x'\in \Lambda(\inv{t})$, $S_*(x,u,v) \cap S_*(x',u,v)=\emptyset$.
			\item if $x,x' \in \Lambda$ are of the same type, then we distinguish according to the sign of $\slope(u+v)$ and $\slope(v-u)$ :
			\begin{itemize}
				\item If $\slope(u+v) > 0$ and $\slope(v-u) > 0$, then there exists $N \in \N$ such that : \\ 
				$S(x,u,v) \cap S(x',u,v) \neq \emptyset$ if and only if $x'=x \pm 2ku$, with $|k| \leq N$. In this case : 
				\begin{align*}
					S(x,u,v) \bigcap (\underset{1 \leq k \leq N} \bigcup S(x+2ku,u,v) )&  = [x_1+2u_1,x_1+u_1+v_1]\times [x_2+u_2,x_2+2u_2+v_2] \\
					S(x,u,v) \bigcap ( \underset{1 \leq k \leq N} \bigcup S(x-2ku,u,v) ) & = [x_1,x_1-u_1+v_1]\times [x_2-u_2,x_2+v_2]
				\end{align*}
				\item If $\slope(u+v) < 0$ and $\slope(v-u) < 0$, then there exists $N \in \N$ such that : \\
				$S(x,u,v) \cap S(x',u,v) \neq \emptyset$ if and only if $x'=x \pm 2kv$, with $|k| \leq N$. In this case : : 
				\begin{align*}
					S(x,u,v) \bigcap (\underset{1 \leq k \leq N} \bigcup S(x+2kv,u,v)) &  = [x_1+2v_1,x_1+u_1+v_1]\times [x_2+u_2+2v_2,x_2+v_2] \\
					S(x,u,v) \bigcap (\underset{1 \leq k \leq N} \bigcup S(x-2kv,u,v) ) & = [x_1,x_1+u_1-v_1]\times [x_2+u_2,x_2-v_2]
				\end{align*}
			\end{itemize} 
		\end{itemize}
	\end{Lemma}
	
	\begin{proof} 
		Since the fundamental quadrilateral $(x,x+u,x+v,x+u+v)$ of $\R^2 / \Lambda$ is contained in $S(x,u,v)$, then the rectangles $S(x,u,v)$, for $x \in \Lambda$, tile $\R^2$ (we mean that $ \displaystyle \R^2 = \underset{x \in \Lambda} \bigcup S(x,u,v)$). So we need to show that, if $x \in~\Lambda \setminus~\Lambda(t) \cup~\Lambda(\inv{t})$, we can cover $S(x,u,v)$ by a union of rectangles $S(y,u,v)$, with $y \in \Lambda(t) \cup \Lambda(\inv{t})$. It is easy to check that, when $\slope(u+v)$ and $\slope(v-u)$ are of the same sign, we have :
		\begin{equation*}
			S(x,u,v) \subset S(x-u,u,v) \cup S(x+u,u,v) \cup S(x+v,u,v) \cup S(x-v,u,v)
		\end{equation*}
		as represented in figure \ref{fig:tiling-same1} on page \pageref{fig:tiling-same1}.
		\begin{figure}[h!]
			\centering
			\labellist
			\small\hair 2pt
			\pinlabel{$x$} at 190 150 
			\pinlabel{$x-u$} at 130 220
			\pinlabel{$x+u$} at 290 140 
			\pinlabel{$x-v$} at -30 60
			\pinlabel{$x+v$} at 450 300
			\pinlabel{$S(x,u,v)$} at 340 200
			\pinlabel{$S(x-v,u,v)$} at 120 90
			\pinlabel{$S(x+v,u,v)$} at 550 320
			\pinlabel{$S(x-u,u,v)$} at 280 305
			\pinlabel{$S(x+u,u,v)$} at 385 90
			\endlabellist
			\includegraphics[width=6.7cm]{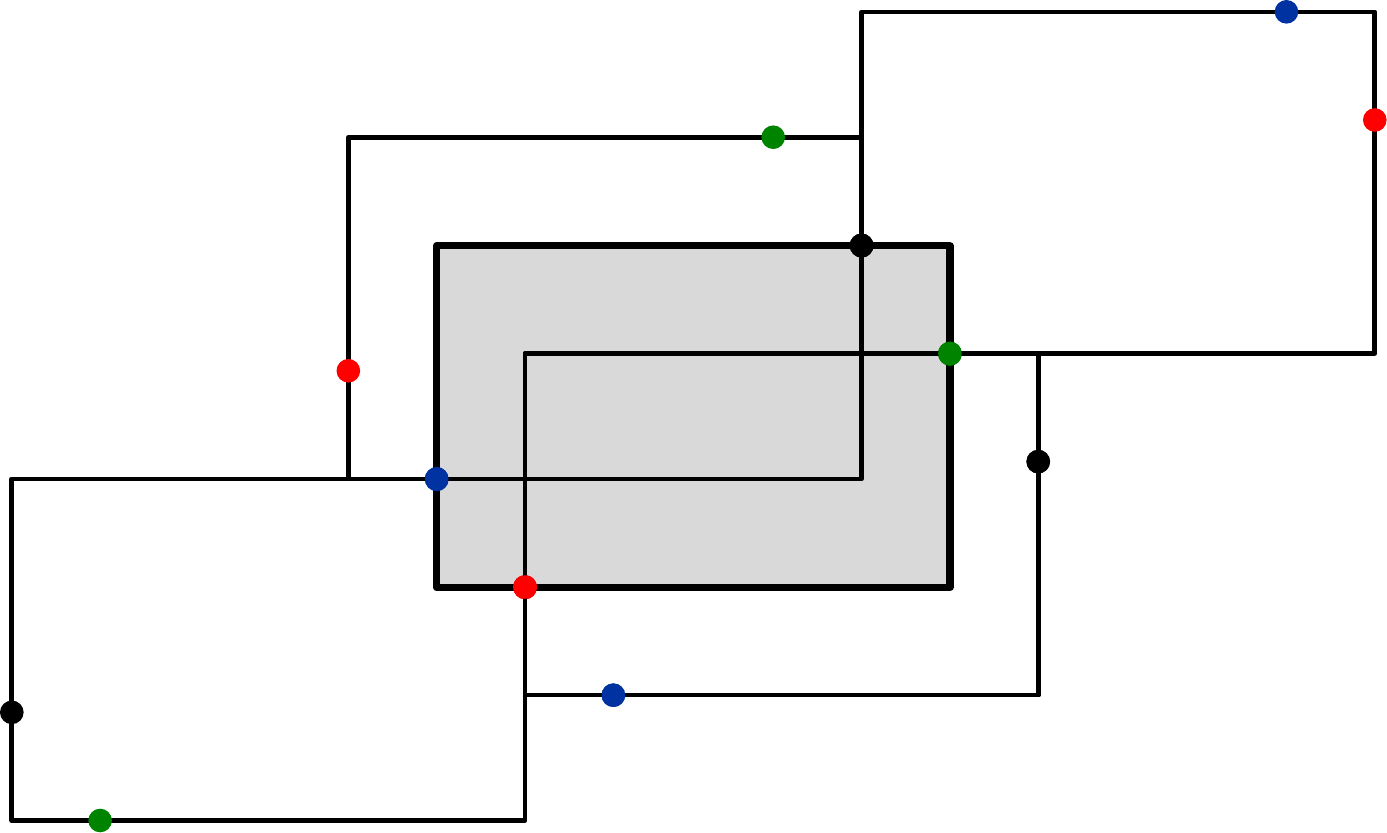}
			\caption{The rectangle $S(x,u,v)$ is covered by rectangles of type $t$ and $\inv{t}$.}
			\label{fig:tiling-same1}
		\end{figure}
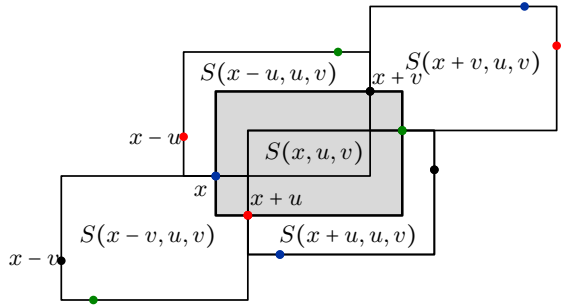
		But notice that when $x \in \Lambda \setminus \Lambda(t) \cup \Lambda(\inv{t})$, the points $x+u,x-u,x+v$ and $x-v$ all four belong to $\Lambda(t)\cup \Lambda(\inv{t})$, hence the claim. \\ 
		
		Now, we investigate the intersection between rectangles. Let $x \in \Lambda$. Let us determine which rectangles intersects $S(x,u,v)$. For $S(y,u,v)$ to intersect $S(x,u,v)$, the point $y$ must be contained in $\mathcal{Z}=[x_1-u_1-v_1,x_1+u_1+v_1]\times [x_2+u_2-v_2,x_2+v_2-u_2]$. Let $t$ be the type of $x$. See figure \ref{fig:tiling-same2} on page \pageref{fig:tiling-same2}.
		
		\begin{figure}[h!]
			\centering
			\labellist
			\small \hair 2pt
			\pinlabel{$x$} at 350 257
			\pinlabel{$x-2u$} at 270 319
			\pinlabel{$x-4u$} at 206 380
			\pinlabel{$x+2u$} at 390 200
			\pinlabel{$x+4u$} at 450 142
			\pinlabel{$x+u+v$} at 675 365
			\pinlabel{$x+v-u$} at 615 435
			\pinlabel{$x-u-v$} at 70 155
			\pinlabel{$x+u-v$} at 130 80
			\pinlabel{$S(x,u,v)$} at 500 310
			\pinlabel{$\color{purple} \mathcal{Z}$} at 130 445
			\pinlabel{points of type $t$} at 1020 590
			\pinlabel{points of type $\inv{t}$} at 1020 553
			\endlabellist
			\includegraphics[width=12cm]{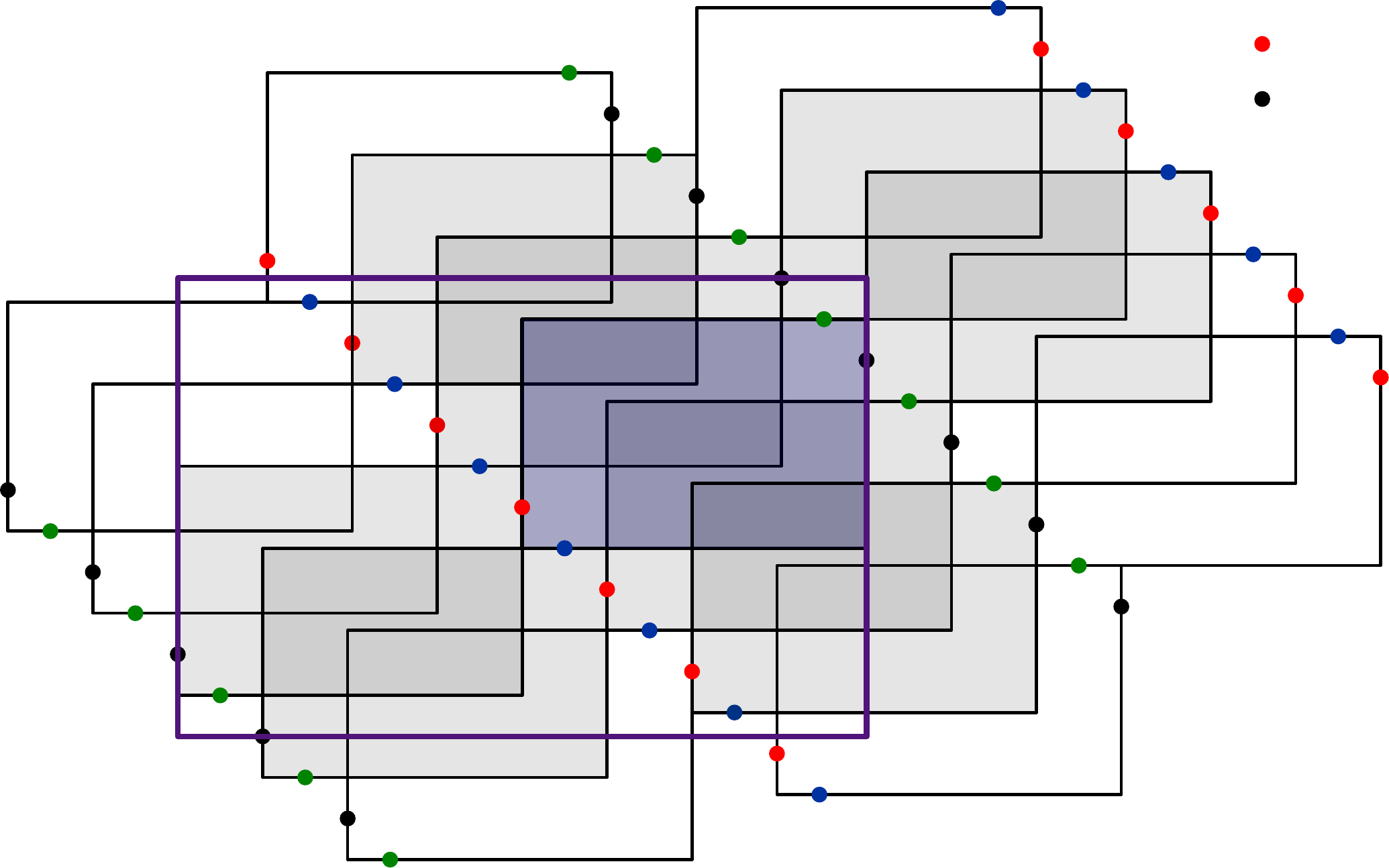}
			\caption{Rectangles of type $t$ and $\inv{t}$ (in gray) which intersect $S(x,u,v)$ (in the center). In this case, $\slope(u+v)>0$ and $N=2$.}
			\label{fig:tiling-same2}
		\end{figure}
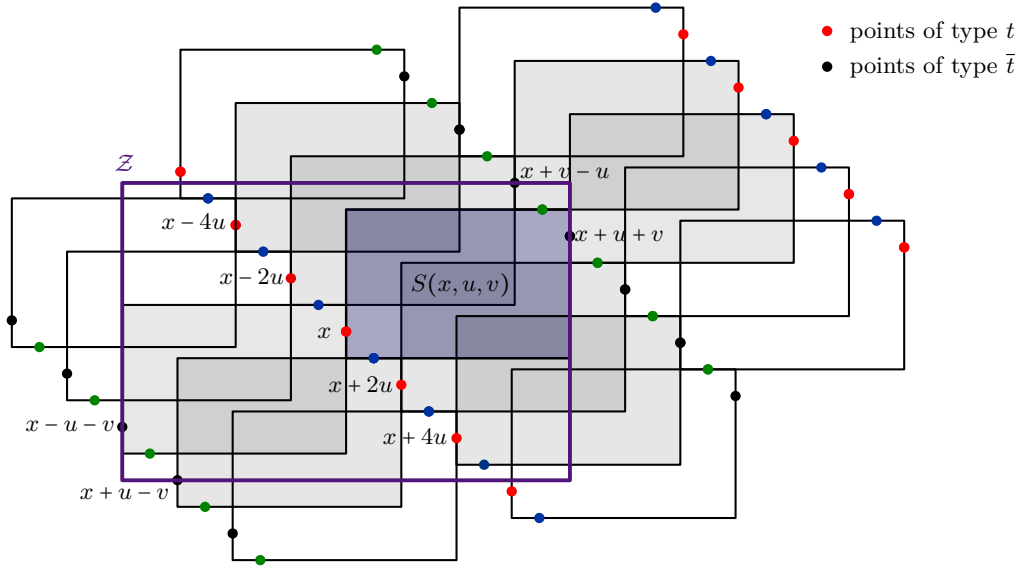
		
		\begin{itemize}
			\item The only points of the lattice $\Lambda$ of type $\inv{t}$ which are contained in $\mathcal{Z}$ are $x+u+v,x-u-v,x+v-u$ and $x-v+u$. Moreover, the rectangle $S(x,u,v)$ intersect $S(x+u+v,u,v)$ only along one of their side, and ditto for $S(x-u-v,u,v)$, $S(x+v-u,u,v)$ and $S(x-v+u,u,v)$. \\ 
			Therefore $S_*(x,u,v) \cap S_*(x+u+v,u,v) =S_*(x,u,v) \cap S_*(x-u-v,u,v)=\emptyset$.
			\item The only points of the lattice $\Lambda$ of type $t$ which can be contained in $\mathcal{Z}$ are of the form $x+2ku$ and $x+2kv$, for $k \in \Z$. 
			\begin{itemize}
				\item If $\slope(u+v) > 0$ and $\slope(v-u) > 0$, then $0< u_1<v_1$ and $0<-u_2<v_2$. So we deduce that $x+2kv \in \mathcal{Z}$ if and only if $k=0$. Therefore we deduce the existence of an integer $N \in \N$ such that the only points of the lattice $\Lambda$ of type $t$ which are contained in $\mathcal{Z}$ are of the form $x+2ku$, with $|k|\leq N$. We also deduce that for $1 \leq k \leq N$, $S(x,u,v) \cap S(x+2ku,u,v) \subset S(x,u,v) \cap S(x+2u,u,v)$ and $S(x,u,v) \cap S(x-2ku,u,v) \subset S(x,u,v) \cap S(x-2u,u,v)$, hence the formula for the intersection.
				\item If $\slope(u+v)<0$ and $\slope(v-u)<0$, then $0<v_1<u_1$ and $0<v_2<-u_2$. So we do the same reasoning as in the previous case exchanging $u$ and $v$ to obtain the analogous result. \qedhere
			\end{itemize}
		\end{itemize}
	\end{proof}
	Let us denote : 
	\begin{align*}
		S^+(x,u,v) & = \left\{ \begin{array}{cc}
			S(x,u,v) \cap S(x+2u,u,v)  & \text{ if } \slope(u+v) \geq 0 \text{ and } \slope(v-u) \geq 0\\
			S(x,u,v) \cap S(x+2v,u,v)  & \text{ if } \slope(u+v) \leq 0 \text{ and } \slope(v-u) \leq 0 
		\end{array}\right. \\
		S^-(x,u,v) & = \left\{ \begin{array}{cc}
			S(x,u,v) \cap S(x-2u,u,v)  & \text{ if } \slope(u+v) \geq 0 \text{ and } \slope(v-u) \geq 0\\
			S(x,u,v) \cap S(x-2v,u,v)  & \text{ if } \slope(u+v) \leq 0 \text{ and } \slope(v-u) \leq 0 
		\end{array}\right. \\
		\text{ and } S^\pm (x,u,v) & = S^+(x,u,v) \cup S^-(x,u,v).
	\end{align*}
	\begin{figure}[!h!]
		\centering
		\labellist
		\pinlabel{$S(x,u,v)$} at 480 140
		\pinlabel{$\scriptstyle S^+(x,u,v)$} at 623 105
		\pinlabel{$\scriptstyle S^-(x,u,v)$} at 341 175
		\pinlabel{$x$} at 270 110 
		\pinlabel{$x+u$} at 430 55 
		\pinlabel{$x+2u$} at 520 45 
		\pinlabel{$x-2u$} at -45 180
		\pinlabel{$x+v$} at 560 225
		\pinlabel{$x+u+v$} at 750 175
		\pinlabel{$S(x-2u,u,v)$} at 160 210
		\pinlabel{$S(x+2u,u,v)$} at 800 65 
		\small \hair 2pt
		\endlabellist
		\includegraphics[width=11cm]{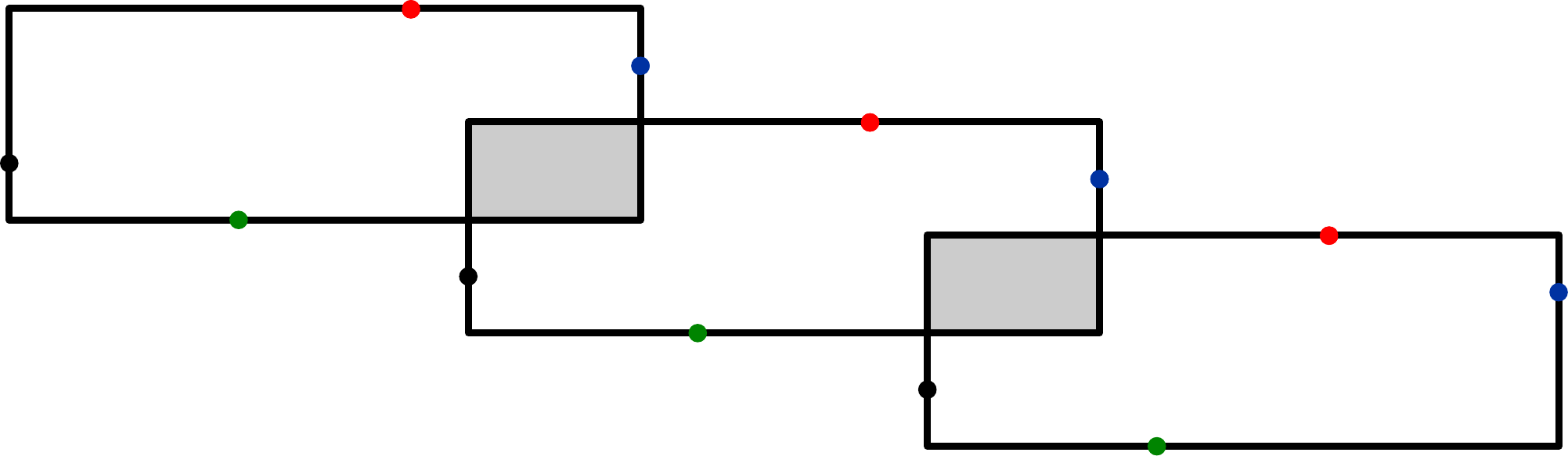}
		\caption{The rectangles $S(x,u,v),S(x-2u,u,v),S(x+2u,u,v)$ and their intersections $S^-(x,u,v)$ and $S^+(x,u,v)$.}
		\label{fig:Spm(x,u,v)}
	\end{figure}
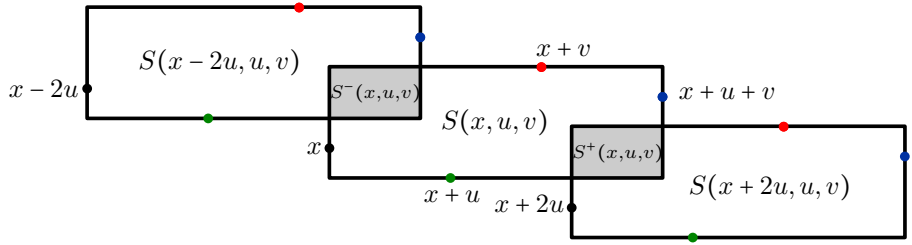
	Hence $S^{\pm}(x,u,v)$ is the subset of the tile $S(x,u,v)$ in which the overlap occurs. \\ 
	
	Let us now deal with the case where $\slope(u+v)$ and $\slope(v-u)$ are of opposite signs. In this case, the rectangles $S(x,u,v)$ for $x \in \Lambda$ of type $t$ or $\inv{t}$ do not cover the whole plane $\R^2$. We have to add some "rest", which we define below and denote by $R(x,u,v)$ (see figure \ref{fig:tiling-opposite1} on page \pageref{fig:tiling-opposite1}).  
	
	\begin{equation}\label{R}
		R(x,u,v)  = \left\{ \begin{array}{ll}
			[x_1,x_1+u_1-v_1] \times [x_2-v_2,x_2+u_2]  & \text{ if } \slope(u+v)\geq 0 \text{ and } \slope(v-u) \leq 0 \\
			{[ x_1,x_1-u_1+v_1 ]} \times [x_2+v_2,x_2-u_2] & \text{ if } \slope(u+v)\leq 0 \text{ and } \slope(v-u) \geq 0
		\end{array}\right. 
	\end{equation}
	
	\begin{equation}\label{R*}
		R_*(x,u,v)  = \left\{ \begin{array}{ll}
			[x_1,x_1+u_1-v_1[ \times [x_2-v_2,x_2+u_2]  & \text{ if } \slope(u+v)\geq 0 \text{ and } \slope(v-u) \leq 0 \\
			{[ x_1,x_1-u_1+v_1 }[ \times [x_2+v_2,x_2-u_2] & \text{ if } \slope(u+v)\leq 0 \text{ and } \slope(v-u) \geq 0
		\end{array}\right. 
	\end{equation}
	
	In this context, there will be no overlap. Now we can state the corresponding lemma : 
	\begin{Lemma}
		\label{tile-opposite-sign}
		Assume that $\slope(u+v) \neq 0$, $\slope(v-u) \neq \infty$ and that $\slope(u+v)$ and $\slope(v-u)$ are of opposite signs. 
		Let $t$ be a type and $\inv{t}$ its $(u,v)$-opposite type. Then :
		\begin{equation*}
			\R^2 = \underset{x \in \Lambda(t)\cup \Lambda(\inv{t})} \bigcup (S(x,u,v) \cup R(x,u,v)).
		\end{equation*}
		Moreover, any two rectangles of this tiling can only intersect along one of their sides. 
	\end{Lemma}
	
	\begin{proof}
		As in the proof of Lemma \ref{tile-same-sign}, first remark that the rectangles $S(x,u,v)$, for $x \in \Lambda$, tile $\R^2$ (we mean that $\R^2 = \underset{x \in \Lambda} \bigcup S(x,u,v)$). So we need to cover $S(x,u,v)$, when $x \notin \Lambda(t) \cup \Lambda(\inv{t})$, by a union of rectangles $S(y,u,v)$ and $R(y,u,v)$, with $y \in \Lambda(t) \cup \Lambda(\inv{t})$. It is easy to check that, when $\slope(u+v)$ and $\slope(v-u)$ are of opposite signs, we have : 
		\begin{equation*}
			S(x,u,v) \subset S(x+u,u,v) \cup S(x-u,u,v) \cup S(x+v,u,v) \cup S(x-v,u,v) \cup R(x+v,u,v) 
		\end{equation*}
		
		\begin{figure}[h]
			\centering
			\labellist
			\small\hair 2pt
			\pinlabel{$x$} at 160 160 
			\pinlabel{$x-u$} at -35 225
			\pinlabel{$x+u$} at 385 105 
			\pinlabel{$x-v$} at 60 60
			\pinlabel{$x+v$} at 285 300
			\pinlabel{$\scriptstyle R(x+v,u,v)$} at 298 200
			\pinlabel{$S(x-v,u,v)$} at 270 145
			\pinlabel{$S(x+v,u,v)$} at 340 255
			\pinlabel{$S(x-u,u,v)$} at 170 230
			\pinlabel{$S(x+u,u,v)$} at 430 180
			\endlabellist
			\includegraphics[width=9cm]{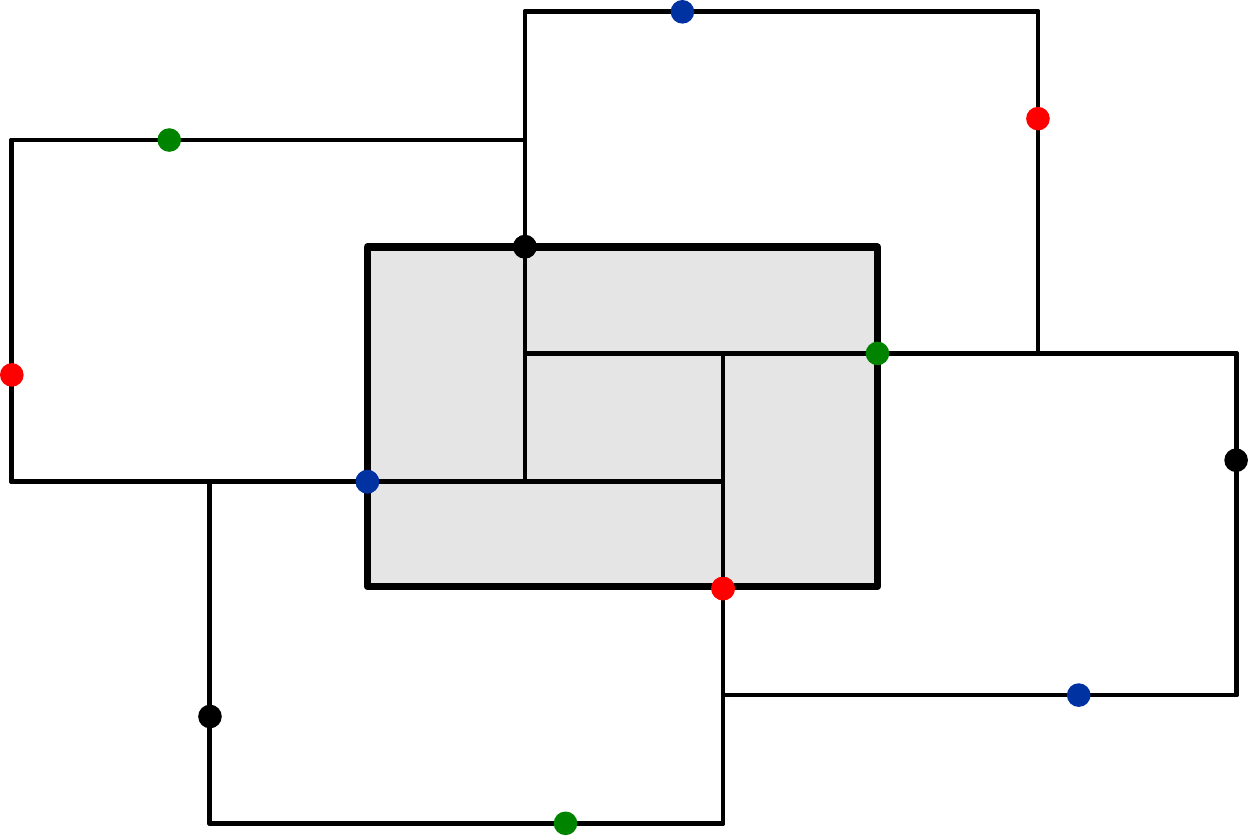}
			\caption{$S(x,u,v)$ is covered by rectangles $S(y,u,v)$ and $R(y,u,v)$ of type $t$ and $\inv{t}$}
			\label{fig:tiling-opposite1}
		\end{figure}
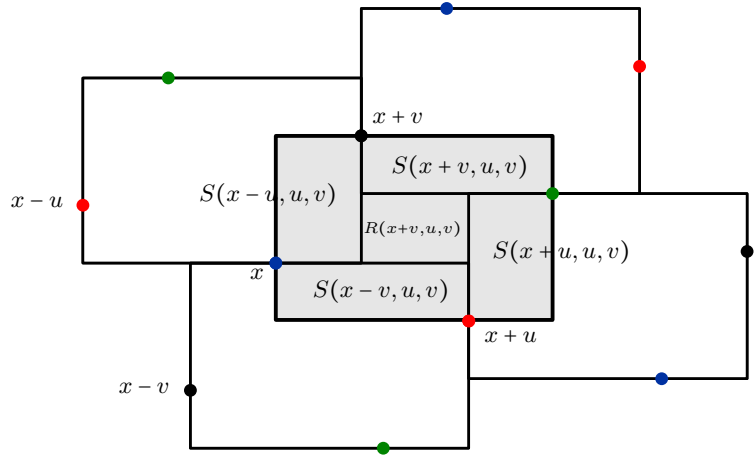
		We conclude by noting, as for the proof of Lemma \ref{tile-same-sign}, that when $x \in \Lambda \setminus (\Lambda(t) \cup \Lambda(\inv{t}))$, the points $x+u,x-u,x+v$ and $x-v$ all four belong to $\Lambda(t) \cup \Lambda(\inv{t})$, hence the claim. \\
		
		Now we investigate the intersection between rectangles. Let $x \in \Lambda$. As in the proof of Lemma \ref{tile-same-sign}, for $S(y,u,v)$ to intersect $S(x,u,v)$, we have that the point $y$ must belong to the rectangle $\mathcal{Z}=[x_1-u_1-v_1,x_1+u_1+v_1]\times [x_2+u_2-v_2,x_2+v_2-u_2]$. Let $t$ be the type of $x$.
		
		\begin{figure}[h]
			\centering
			\labellist
			\small \hair 2pt
			\pinlabel{$x$} at 360 265
			\pinlabel{$S(x,u,v)$} at 470 310
			\pinlabel{$S(x+u+v,u,v)$} at 720 370
			\pinlabel{$S(x-u-v,u,v)$} at 220 250
			\pinlabel{$S(x+v-u,u,v)$} at 370 475
			\pinlabel{$S(x-v+u,u,v)$} at 570 150
			\pinlabel{$\color{purple} \mathcal{Z}$} at 70 100
			\pinlabel{points of type $t$} at 970 605
			\pinlabel{points of type $\inv{t}$} at 970 568
			\pinlabel{\1} at 300 365
			\pinlabel{\2} at 545 425
			\pinlabel{\3} at 390 200
			\pinlabel{\4} at 640 260
			\pinlabel{\1 $R(x+v-u,u,v)$} at 1047 500
			\pinlabel{\2 $R(x+2v,u,v)$} at 1030 465
			\pinlabel{\3 $R(x,u,v)$} at 1000 430
			\pinlabel{\4 $R(x+u+v,u,v)$} at 1047 395
			\endlabellist
			\includegraphics[width=10cm]{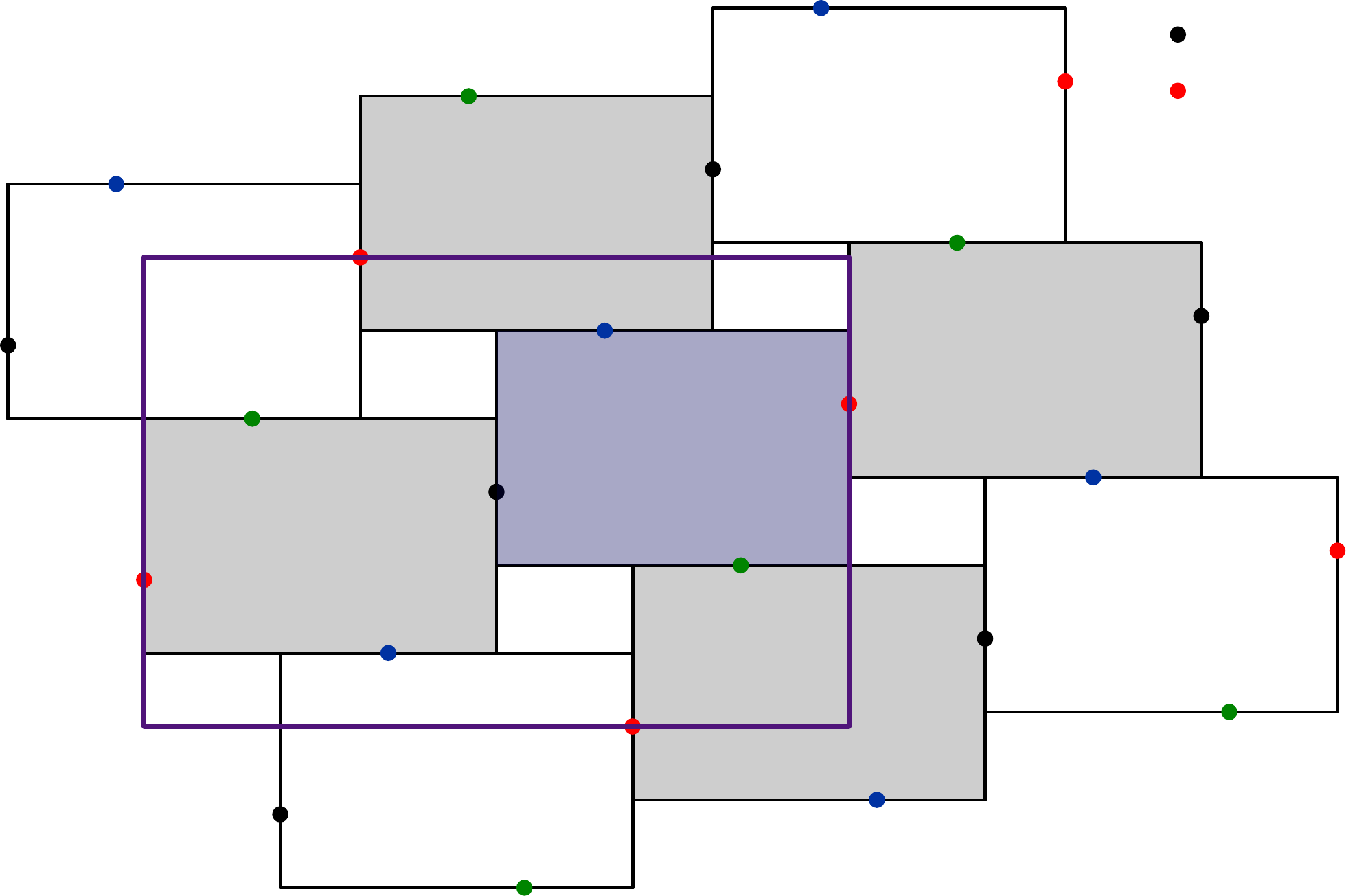}
			\caption{The tiling of $\R^2$ by tiles $S(y,u,v)$ and $R(y,u,v)$, with $y \in \Lambda(t) \cup \Lambda(\inv{t})$.}
			\label{fig:tiling-opposite2}
		\end{figure}
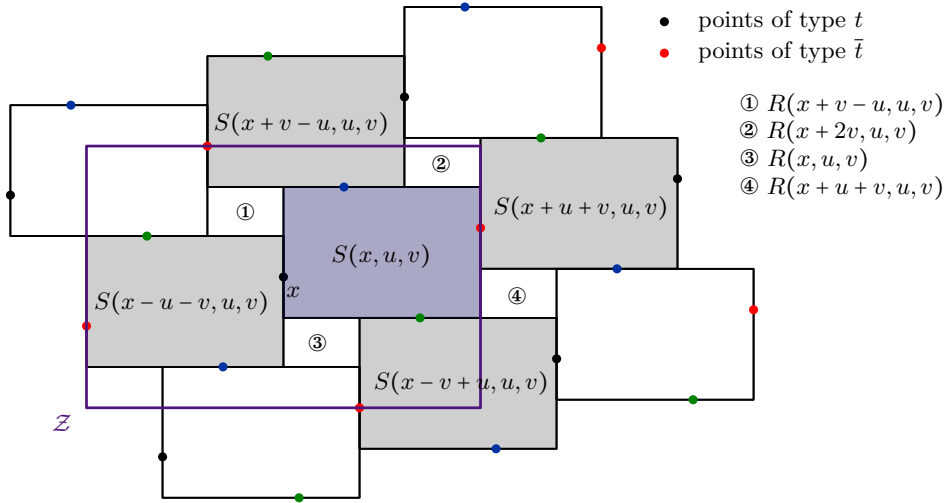
		
		\begin{itemize}
			\item Again as in the proof of Lemma \ref{tile-same-sign}, we note that the only points of the lattice $\Lambda$ of type $\inv{t}$ in $\mathcal{Z}$ are $x+u+v,x-u-v,x+v-u$ and $x-v+u$, and the rectangle $S(x,u,v)$ intersects $S(x+u+v,u,v)$ only along one of its sides, and ditto for $S(x-u-v,u,v)$, $S(x+v-u,u,v)$ and $S(x-v+u,u,v)$.
			\item Again as in the proof of Lemma \ref{tile-same-sign}, we note that the only points of the lattice $\Lambda$ of type $t$ in $\mathcal{Z}$ are of the form $x+2ku$ and $x+2kv$. But using our hypothesis on the slopes of $u+v$ and $v-u$, we obtain that either $0<u_1<v_1$ and $0<v_2<-u_2$ or $0<v_1<u_1$ and $0<-u_2<v_2$. From these inequalities we deduce that $x+2ku$ belongs to $\mathcal{Z}$ if and only if $k=0$ and $x+2kv$ belongs to $\mathcal{Z}$ if and only if $k=0$. Therefore $S(x,u,v)$ intersects no other rectangle $S(y,u,v)$ of type~$t$.
		\end{itemize} 
		At last, note that by construction, $R(x,u,v)$ intersects only $S(x,u,v), S(x-u-v,u,v),S(x-2v,u,v)$ and $S(x+u-v,u,v)$, and the intersection occurs only along one of their sides.

	\end{proof}

	Notice that since $(u,v)$ is a basis of $\Lambda$, so are $(u,u+v)$ and $(u+v,v)$. The next lemma, in the case where $\slope(u+v)$ and $\slope(v-u)$ are of opposite signs, aims at showing that an horizontal segment of length $l=\max(u_1,v_1)$ is necessarily included in some $S(x,u,v), S(x,u,u+v)$ or $S(x,u+v,v)$. Denote $(u'v') = \left\{ \begin{array}{cc}
		(u,u+v)    &  \text{ if } \slope(u+v) \geq 0  \\
		(u+v,v)    &  \text{ if } \slope(u+v) \leq 0
	\end{array} \right.$, \\
	$S'(x,u,v) = S(x,u',v') \qquad S_*'(x,u,v) = S_*(x,u',v') \qquad \text{ and } \quad S'^\pm(x,u,v) = S^\pm(x,u',v'). $ \\ 
	See figure \ref{fig:S'(x,u,v)} on page \pageref{fig:S'(x,u,v)} for $S'(x,u,v)$, in the case $\slope(u+v)\geq 0$ and $\slope(v-u)\leq 0$. \\
	
	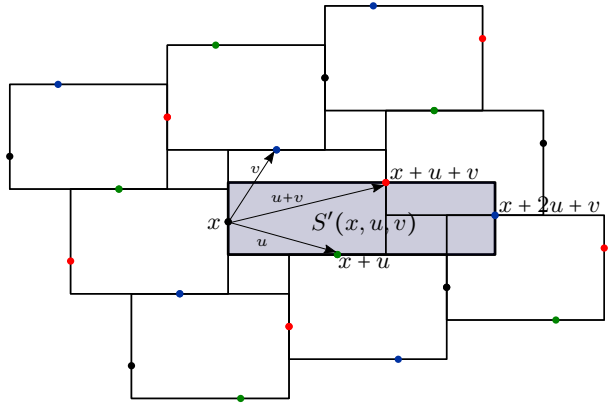
\begin{figure}[h!]
		\centering
		\labellist
		\small\hair 2pt
		\pinlabel{$S'(x,u,v)$} at 560 280
		\pinlabel{$x$} at 325 280 
		\pinlabel{$\scriptstyle u$} at 400 250
		\pinlabel{$\scriptstyle v$} at 390 365 
		\pinlabel{$\scriptstyle u+v$} at 440 320 
		\pinlabel{$x+u$} at 560 220
		\pinlabel{$x+u+v$} at 670 360
		\pinlabel{$x+2u+v$} at 850 310
		\endlabellist
		\includegraphics[width=8cm]{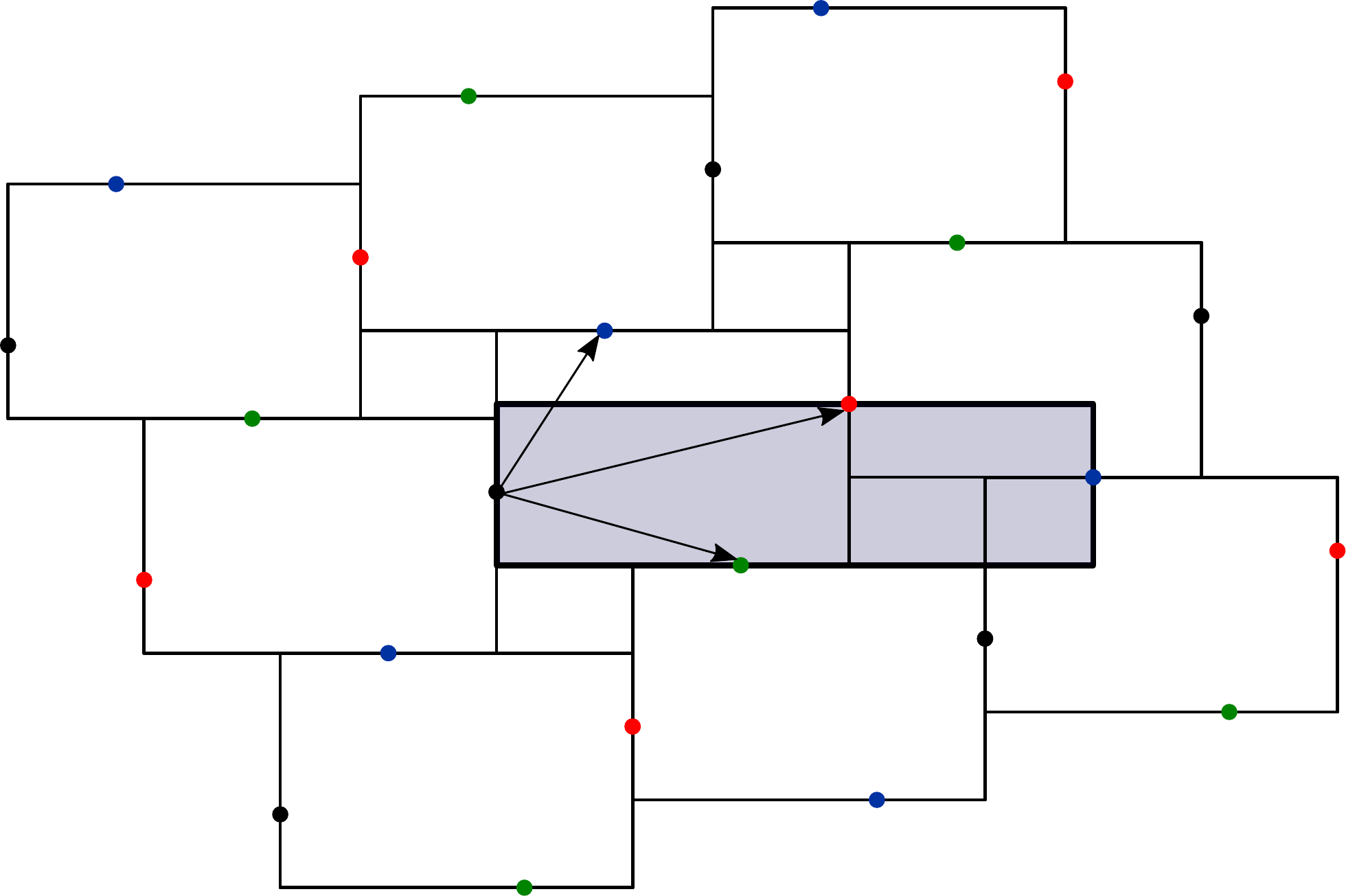}
		\caption{The rectangle $S'(x,u,v)$ in the case $\slope(u+v) \geq 0$ and $\slope(v-u)\leq 0.$}
		\label{fig:S'(x,u,v)}
	\end{figure}
	
	Now in this case ($\slope(u+v)$ and $\slope(v-u)$ of opposite signs), denote :
	
	\begin{equation}\label{T}
		T(x,u,v)  = \left\{ \begin{array}{ll}
			[x_1+v_1,x_1+2u_1] \times [x_2+u_2,x_2+u_2+v_2]  & \text{ if } \slope(u+v)\geq 0 \text{ and } \slope(v-u) \leq 0 \\
			{[ x_1+u_1,x_1+2v_1 ]} \times [x_2+u_2+v_2,x_2+v_2] & \text{ if } \slope(u+v)\leq 0 \text{ and } \slope(v-u) \geq 0
		\end{array}\right. 
	\end{equation}
	\begin{equation}\label{T*}
		T_*(x,u,v)  = \left\{ \begin{array}{ll}
			[x_1+v_1,x_1+2u_1[ \times [x_2+u_2,x_2+u_2+v_2]  & \text{ if } \slope(u+v)\geq 0 \text{ and } \slope(v-u) \leq 0 \\
			{[ x_1+u_1,x_1+2v_1 [} \times [x_2+u_2+v_2,x_2+v_2] & \text{ if } \slope(u+v)\leq 0 \text{ and } \slope(v-u) \geq 0
		\end{array}\right. 
	\end{equation}
	See figure \ref{fig:T(x,u,v)} on page \pageref{fig:T(x,u,v)}. Remark that we have : 
	\begin{align}
		T(x,u,v) & \subset S'(x,u,v) \label{T-in-S'}  \\ 
		T_*(x,u,v) & \subset S'_*(x,u,v) \label{T*-in-S'*} \\ 
		S'^\pm(x,u,v) & \subset S'(x,u,v) \setminus T(x,u,v) \label{S'pm-in-S'-minus-T} 
	\end{align}
	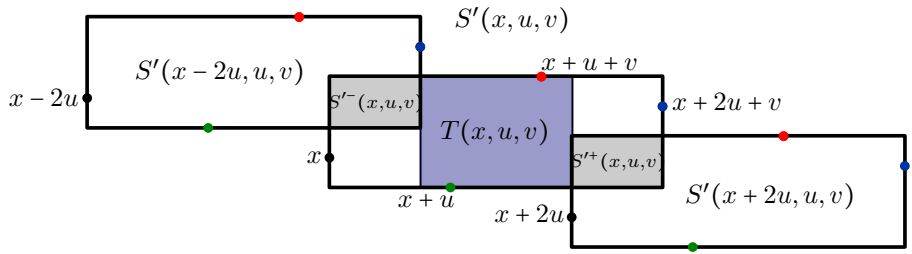
\begin{figure}[h!]
		\centering
		\labellist
		\pinlabel{$T(x,u,v)$} at 480 140
		\pinlabel{$\scriptstyle S'^{+}(x,u,v)$} at 623 105
		\pinlabel{$\scriptstyle S'^{-}(x,u,v)$} at 341 175
		\pinlabel{$x$} at 270 110 
		\pinlabel{$x+u$} at 400 60 
		\pinlabel{$x+2u$} at 520 45 
		\pinlabel{$x-2u$} at -45 180
		\pinlabel{$x+u+v$} at 590 220
		\pinlabel{$x+2u+v$} at 750 175
		\pinlabel{$S'(x-2u,u,v)$} at 160 210
		\pinlabel{$S'(x,u,v)$} at 500 270
		\pinlabel{$S'(x+2u,u,v)$} at 800 65 
		\small \hair 2pt
		\endlabellist
		\includegraphics[width=11cm]{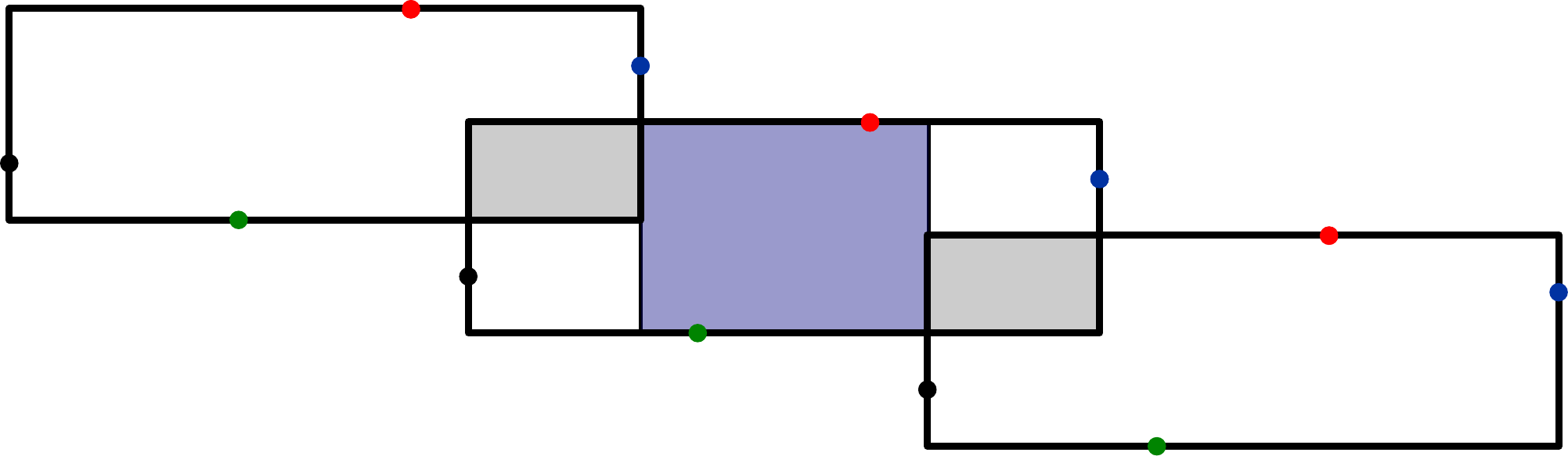}
		\caption{The rectangle $T(x,u,v)$ is included in $S'(x,u,v)$ and disjoint from $S'^\pm(x,u,v)$.}
		\label{fig:T(x,u,v)}
	\end{figure}
	
	We now deduce that the rectangles $T(x,u,v)$ are disjoints : 
	
	\begin{Lemma} \label{T-disjoint}
		Assume that $\slope(u+v)$ and $\slope(v-u)$ are of opposite signs. Let $t \in \Lambda/2\Lambda $ be a type and $\inv{t}$ its $(u,v)$-opposite type. Let $x,y \in \Lambda(t)\cup \Lambda(\inv{t})$. Then the rectangles $T(x,u,v)$ and $T(y,u,v)$ are disjoints : $T(x,u,v) \cap T(y,u,v) = \emptyset$. 
	\end{Lemma}
	
	\begin{proof}
		First notice that since $T(x,u,v) \subset S'(x,u,v)$ and $T(y,u,v) \subset S'(y,u,v)$ (see \eqref{T-in-S'}), intersection between $T(x,u,v)$ and $T(y,u,v)$ can only occur when $S'(x,u,v)$ and $S'(y,u,v)$ intersect and then $T(x,u,v) \cap T(y,u,v) \subset S'(x,u,v) \cap S'(y,u,v)$. But $S'(x,u,v) \cap S'(y,u,v) \subset S'^{\pm}(x,u,v)$, so by \eqref{S'pm-in-S'-minus-T}, $T(x,u,v) \cap T(y,u,v) \subset S'(x,u,v)\setminus T(x,u,v)$, which implies that $T(x,u,v) \cap T(y,u,v)$ is empty.
	\end{proof}
	
	The last lemma of this section shows that every horizontal segment of some prescribe length is included in a rectangle $S(x,u,v)$ or $T(x,u,v)$.  
	\begin{Lemma} \label{segment-in-square}
		Assume that $\slope(u+v)$ and $\slope(v-u)$ are of opposite signs. Let $l=\max(u_1,v_1)$ (thus $l=u_1$ if $\slope(u+v)\geq 0$ and $l=v_1$ if $\slope(u+v)\leq 0$) and $I$ be an horizontal segment of length $l$ in $\R^2$ which does not intersect the lattice $\Lambda$. 
		
		Then, there exists $x \in \Lambda$ such that $I \subset S_*(x,u,v)$ or $I \subset T_*(x,u,v)$.
	\end{Lemma}
	
	\begin{proof} We write the proof in the case where $\slope(u+v)\geq 0$. Thus $\slope(v-u) \leq 0$ and $l=u_1$. The other case is identical (up to a reflection across the horizontal axis). \\ 
		Let $y=(y_1,y_2)$ be the point in $\R^2$ such that $I$ is the segment which joins $y$ to $(y_1+u_1,y_2)$. 
		Then, by Lemma \ref{tile-opposite-sign}, there exists a point $x \in \Lambda$ such that $y \in S_*(x,u,v)$ or $y \in R_*(x,u,v)$. 
		
		\begin{itemize}
			\item If $y \in S(x,u,v)$, then we have $x_1 \leq y_1 < x_1+u_1+v_1$ and $x_2+u_2 \leq y_2 \leq x_2 + v_2$. We distinguish according to the zone to which $y$ belong, represented in the figure \ref{fig:segmentI} on page \pageref{fig:segmentI}.
			The proof for each case is illustrated on figure \ref{fig:proofI}, page \pageref{fig:proofI}. 
			
			\begin{figure}[!h]
				\centering
				\labellist
				\small\hair 2pt
				\pinlabel{\1} at 40 140
				\pinlabel{\2} at 210 110
				\pinlabel{\3} at 130 110
				\pinlabel{\4} at 120 195
				\pinlabel{\5} at 205 195
				\pinlabel{$x$} at -10 110
				\endlabellist
				\includegraphics[width=5cm]{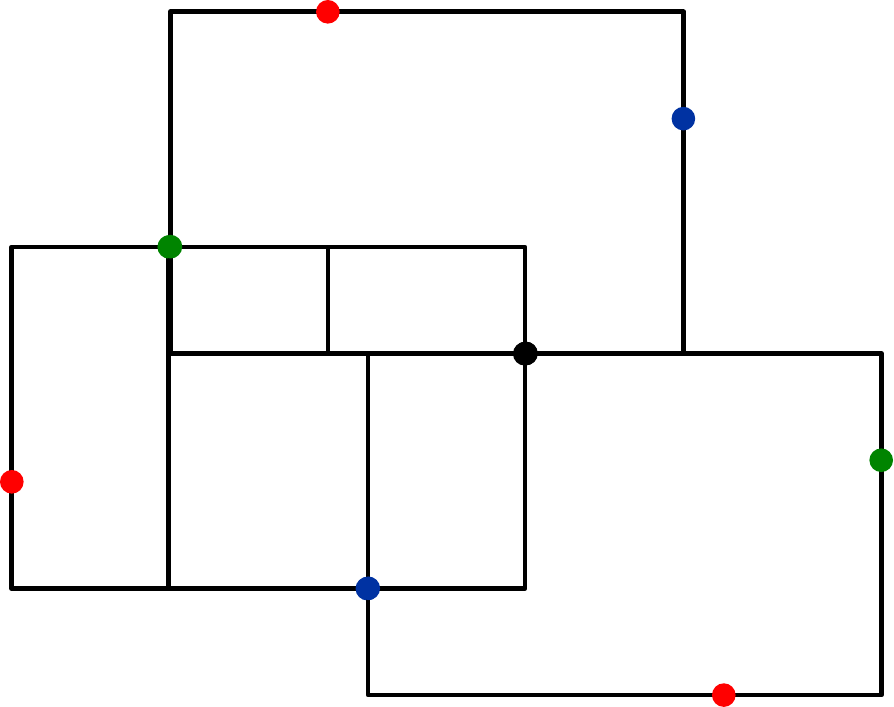}
				\caption{The five cases in the proof of Lemma \ref{segment-in-square}.}
				\label{fig:segmentI}
			\end{figure}
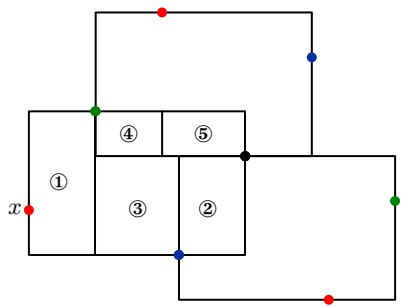
			
				\begin{figure}[!h!]
				\centering
				\begin{subfigure}{0.4 \textwidth}
					\centering
					\labellist
					\small\hair 2pt
					\pinlabel{$\color{purple} I$} at 130 150 
					\pinlabel{$x$} at -15 110
					\pinlabel{$S(x,u,v)$} at -60 230 
					\endlabellist
					\includegraphics[width=5cm]{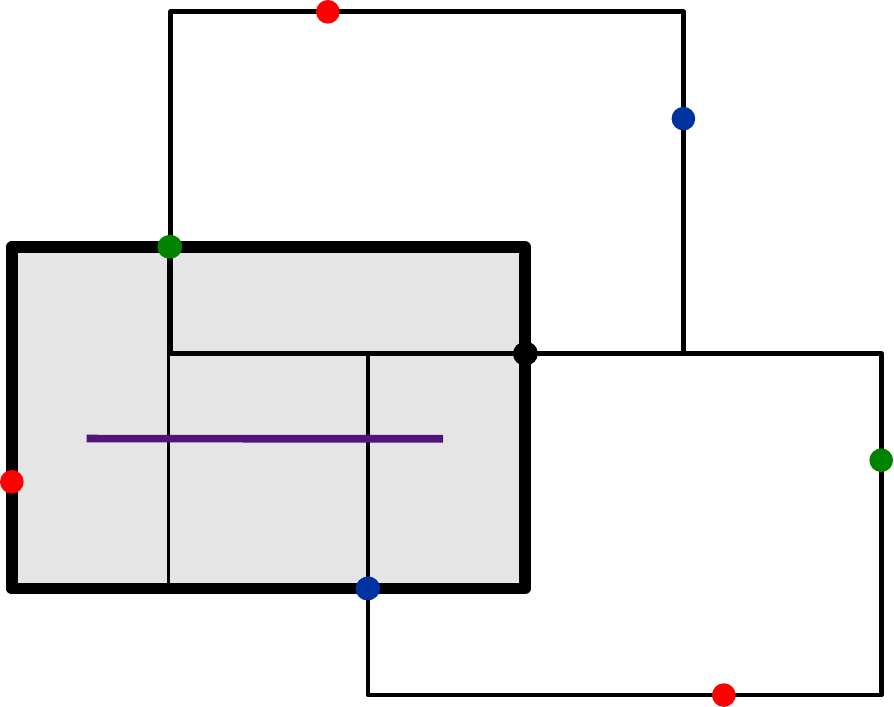}
					\caption{Case \1}
				\end{subfigure}
				\begin{subfigure}{0.4 \textwidth}
					\centering
					\labellist
					\small\hair 2pt
					\pinlabel{$\color{purple} I$} at 300 135 
					\pinlabel{$x$} at -15 110
					\pinlabel{$S(x+u,u,v)$} at 520 30 
					\endlabellist
					\includegraphics[width=5cm]{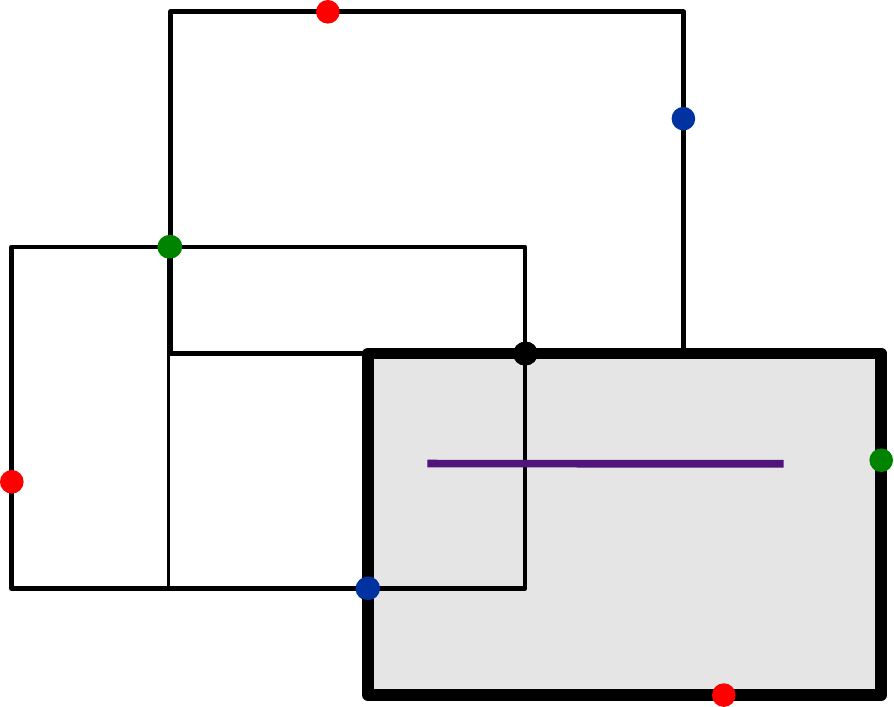}
					\caption{Case \2}
				\end{subfigure} \\ 
				\begin{subfigure}{\textwidth}
					\centering
					\labellist
					\small\hair 2pt
					\pinlabel{$\color{purple} I$} at 560 190 
					\pinlabel{$x$} at 330 160
					\pinlabel{$S'(x,u,v)$} at 840 230 
					\pinlabel{$S'(x-2u,u,v)$} at 120 350
					\pinlabel{$S'(x+2u,u,v)$} at 1020 140
					\endlabellist
					\includegraphics[width=13cm]{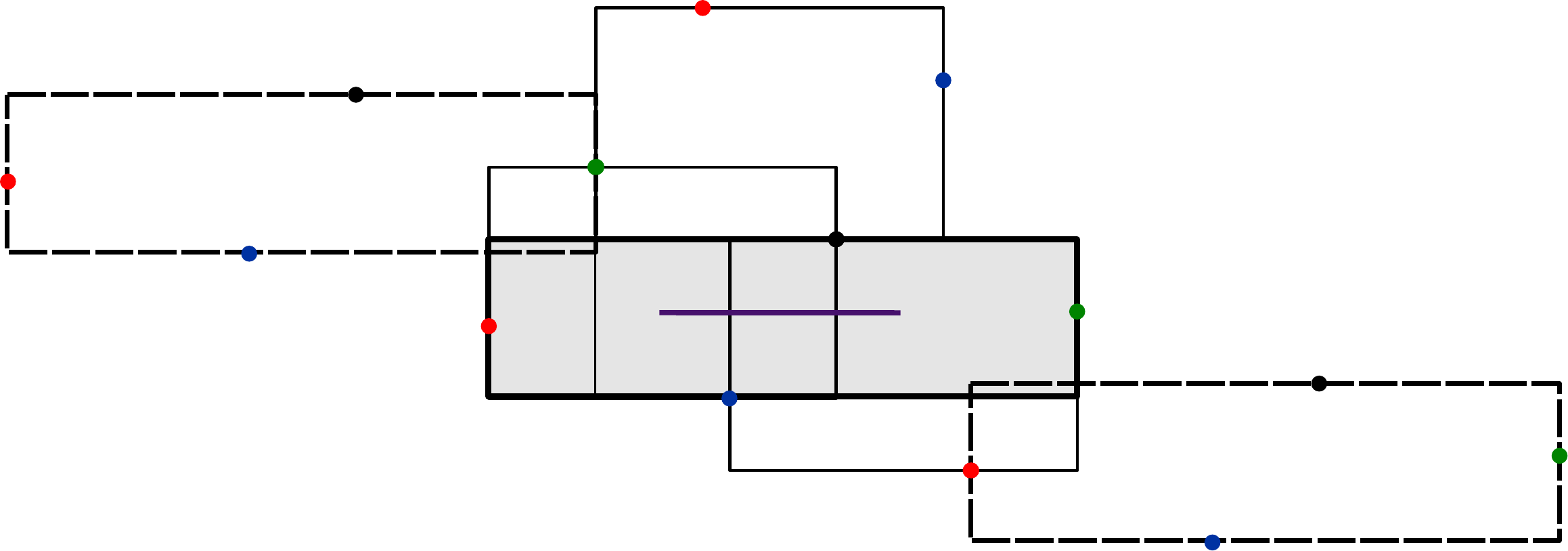}
					\caption{Case \3}
				\end{subfigure} \\ \vspace{0.15cm}
				\begin{subfigure}{0.25 \textwidth}
					\centering
					\labellist
					\small\hair 2pt
					\pinlabel{$\color{purple} I$} at 200 220 
					\pinlabel{$x$} at -15 110
					\pinlabel{$S(x+v,u,v)$} at -15 350 
					\endlabellist
					\includegraphics[width=4.5cm]{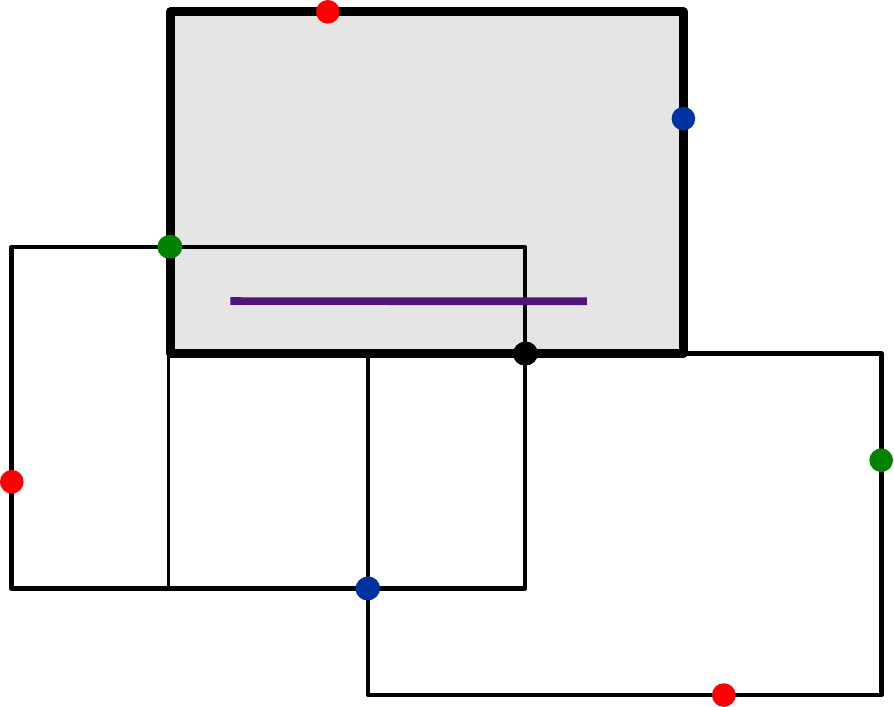}
					\caption{Case \4}
				\end{subfigure}
				\begin{subfigure}{0.7 \textwidth}
					\centering
					\labellist
					\small\hair 2pt
					\pinlabel{$\color{purple} I$} at 560 220 
					\pinlabel{$x$} at 250 110
					\pinlabel{$S'(x+v,u,v)$} at 870 280 
					\endlabellist
					\includegraphics[width=12cm]{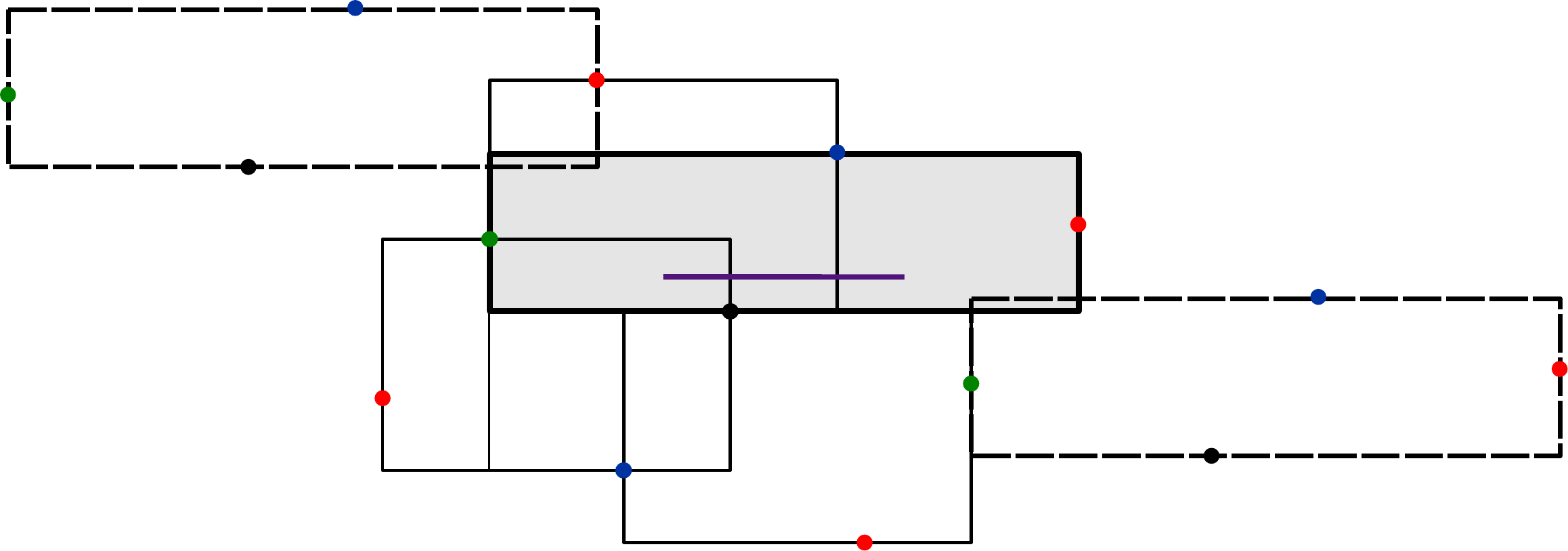}
					\caption{Case \5}
				\end{subfigure}
				\caption{Illustration of the proof of Lemma \ref{segment-in-square} : The segment $I$ is included in a rectangle.}
				\label{fig:proofI}
			\end{figure}

			\begin{enumerate}
				\item If $y \in \text{\1}$, that is $x_1 \leq y_1 < x_1+v_1$, then $x_1+u_1 \leq y_1+u_1 < x_1+u_1+v_1$, and so $I \subset S_*(x,u,v)$.
				\item If $y \in \text{\2}$, that is $x_1+u_1 \leq y_1 < x_1+u_1+v_1 $ and $x_2+u_2 \leq y_2 \leq x_2+u_2+v_2$, then notice that $y \in S_*(x+u,u,v)$ with $x_1+u_1 \leq y_1 < x_1+u_1+v_1 $, so by the previous case we have $I \subset S_*(x+u,u,v)$. 
				\item If $y \in \text{\3}$, that is $x_1+v_1 \leq y_1 < x_1+u_1$ and $x_2+u_2 \leq y_2 \leq x_2+u_2+v_2$, then we have $x_1+u_1+v_1 \leq y_1+u_1 < x_1 +2u_1$, and then we deduce that $I \subset T_*(x,u,v)$. 
				\item If $y \in \text{\4}$ or $y \in \text{\5}$, that is $x_1+v_1 \leq y_1 < x_1+u_1+v_1$ and $x_2+u_2+v_2 \leq y_2 \leq x_2+v_2$, then notice that $y \in S_*(x+v,u,v)$ with $x_2+u_2+v_2 \leq y_2 \leq x_2+v_2$ and then, if $y \in \text{\4}$, that is $y_1 < x_1+2v_1$, we use the first case to deduce that $I \subset S_*(x+v,u,v)$, and if $y \in \text{\5}$, that is $x_1+2v_1 \leq y_1 < x_1+u_1+v_1$, then we use the third case to deduce that $I \subset T_*(x+v,u,v)$ with $I \cap S^\pm(x+v,u,u+v)$ is at most a point.
			\end{enumerate}
			Hence we have covered every case. 
			
			\item If $y \in R_*(x,u,v)$, recall that $R_*(x,u,v) \subset S_*(x-v,u,v)$, so $y \in S_*(x-v,u,v)$ and then we can use what we have previously done. (We can even be a bit more precise : if $y \in R_*(x,u,v)$, then $y \in S_*(x-v,u,v)$ with $x_1 \leq y_1 < x_1+u_1-v_1$ and $x_2-v_2 \leq y_2 \leq x_2+u_2$, so we are in the third case : $I \subset T_*(x-v,u,v)$). \qedhere
		\end{itemize}
	\end{proof}
	
	\subsubsection{Segments at height $h$ in the rectangles $S(x,u,v)$ and $R(x,u,v)$}
	\label{subsubsec:segments-rectangle}
	
	This section is intended to apply the "tilings" obtained above to decompose an horizontal segment into sub-segments. \\ 
	
	Let $0 < h < 1$ and consider the horizontal segment at height $h$ in $S(x,u,v)$ (that is the segment from the point $(x_1,x_2+u_2+h)$ to the point $(x_1+u_1+v_1,x_2+u_2+h)$). Denote it by $I_S(x,u,v,h)$ and note that $I_S(x,u,v,h)$ is a segment of length $u_1+v_1$. Similarly, let $0 < h < u_2+v_2$ and consider the horizontal segment at height $h$ in $R(x,u,v)$ (that is the segment from the point $(x_1,x_2-v_2+h)$ to the point $(x_1+u_1+v_1,x_2-v_2+h))$. Denote it by $I_R(x,u,v,h)$ and note that $I_R(x,u,v,h)$ is a segment of length $|u_1-v_1|$. Finally, we also define the \emph{type} of a segment $I_S(x,u,v,h)$ (resp. $I_R(x,u,v,h)$) as the type of the point $x$. See figure \ref{fig:I_S&I_R} on page \pageref{fig:I_S&I_R}. \\
	
	\begin{figure}[!h!]
		\centering
		\begin{subfigure}{0.4 \textwidth}
			\centering
			\labellist
			\small\hair 2pt
			\pinlabel{$h$} at -10 65
			\pinlabel{$I_S(x,u,v,h)$} at 150 120
			\pinlabel{$S(x,u,v)$} at 280 190
			\pinlabel{$x$} at 55 55
			\pinlabel{$x+u$} at 230 -10
			\pinlabel{$x+v$} at 140 185
			\pinlabel{$x+u+v$} at 340 120
			\endlabellist
			\includegraphics[width=3cm]{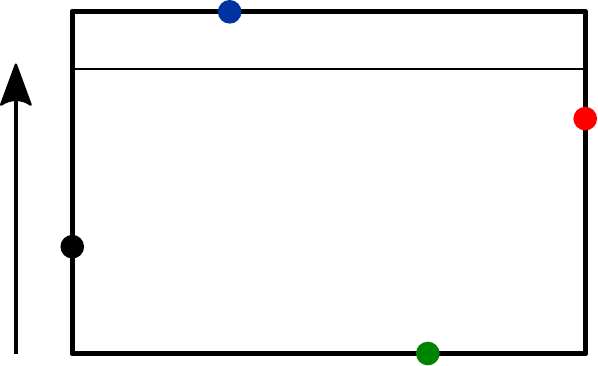}
			\caption{Segment of heigth $h$ in $S(x,u,v)$}
			\label{fig:I_S}
		\end{subfigure}
		\begin{subfigure}{0.4 \textwidth}
			\centering
			\labellist
			\small\hair 2pt
			\pinlabel{$h$} at 40 80
			\pinlabel{$I_R(x,u,v,h)$} at 145 80
			\pinlabel{$R(x,u,v)$} at 180 135
			\pinlabel{$x$} at 60 170
			\pinlabel{$x+u$} at 290 120
			\pinlabel{$x-v$} at -30 60
			\pinlabel{$x+u-v$} at 230 10 
			\endlabellist
			\includegraphics[width=3cm]{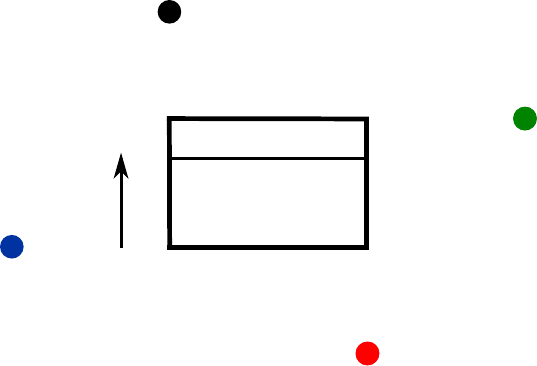}
			\caption{Segment of heigth $h$ in $R(x,u,v)$}
			\label{fig:I_R}
		\end{subfigure}
		\caption{Segment at height $h$, $I_S(x,u,v,h)$ and $I_R(x,u,v,h)$}
		\label{fig:I_S&I_R}
	\end{figure}
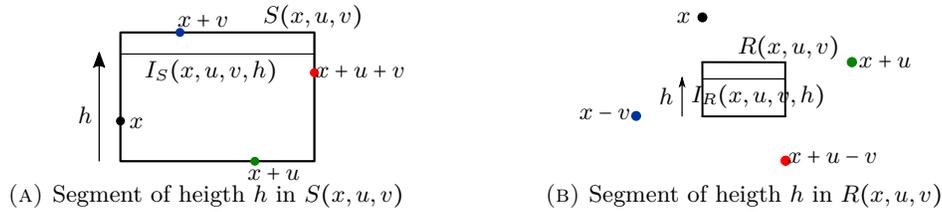
	
	We now decompose any horizontal segment $I$ of any length using the tiling in the case where $\slope(u+v)$ and $\slope(v-u)$ are of opposite signs.
	\begin{Lemma} \label{intervals} Assume that $\slope(u+v) \neq 0$, $\slope(v-u) \neq \infty$ and that $\slope(u+v)$ and $\slope(v-u)$ are of opposite signs. Let $t$ be a type and let $I$ be a horizontal segment which does not intersect the lattice $\Lambda$. \\
		There exist $p \in \N$ such that for all $0 \leq j \leq p+1$, there exists a point $x_j \in \Lambda(t)\cup \Lambda(\inv{t})$ and a real $0<h_j<1$ such that,
		\begin{equation*}
			I=I_0 \cup I_1 \cup \hdots \cup I_p \cup I_{p+1}
		\end{equation*}
		with :
		\begin{itemize}
			\item Either $I_0 \subset I_S(x_0,u,v,h_0)$ or $I_0 \subset I_R(x_0,u,v,h_0)$, and either $I_{p+1} \subset I_S(x_{p+1},u,v,h_{p+1})$ or $I_{p+1} \subset I_R(x_{p+1},u,v,h_{p+1})$.
			\item For $1 \leq j \leq p$, either $I_j=I_S(x_j,u,v,h_j)$ or $I_j=I_R(x_j,u,v,h_j)$.
			\item If $I_j=I_R(x_j,u,v,h_j)$, then $I_{j+1}=I_S(x_{j+1},u,v,h_{j+1})$ with $x_{j+1}=\left\{ \begin{array}{cc}
				x_j+u-v   & \text{ if } \slope(u+v)>0 \\
				x_j-u+v   & \text{ if } \slope(u+v)<0
			\end{array} \right.$. 
			\item If $I_j=I_S(x_j,u,v,h_j)$ and $I_{j+1}=I_R(x_{j+1},u,v,h_{j+1})$, then $x_{j+1}=x_j+u+v$.
			\item If $I_j=I_S(x_j,u,v,h_j)$ and $I_{j+1}=I_S(x_{j+1},u,v,h_{j+1})$, then $x_{j+1}=x_j+u+v$.
		\end{itemize}
		Moreover, suppose $j<j'$, then $I_j \cap I_{j'} \neq \emptyset$ if and only if $j'=j+1$ and in this case $I_j\cap I_{j+1}$ is a point (which is the endpoint of $I_j$ and the starting point of $I_{j+1}$). 
	\end{Lemma}
	
	\begin{figure}[h]
		\centering
		\labellist
		\small\hair 2pt
		\pinlabel{$x_0$} at -15 220
		\pinlabel{$x_1$} at 235 285
		\pinlabel{$x_2$} at 475 345
		\pinlabel{$x_3$} at 575 180
		\pinlabel{$x_4$} at 820 240
		\pinlabel{$h_0$} at 90 210
		\pinlabel{$h_1$} at 295 243
		\pinlabel{$h_2$} at 550 243
		\pinlabel{$h_3$} at 665 190
		\pinlabel{$h_4$} at 880 220
		\pinlabel{$\color{purple} I_0$} at 150 273
		\pinlabel{$\color{pink} I_1$} at 380 273
		\pinlabel{$\color{purple} I_2$} at 550 273
		\pinlabel{$\color{pink} I_3$} at 710 273
		\pinlabel{$\color{purple} I_4$} at 940 273
		\endlabellist
		\includegraphics[width=13cm]{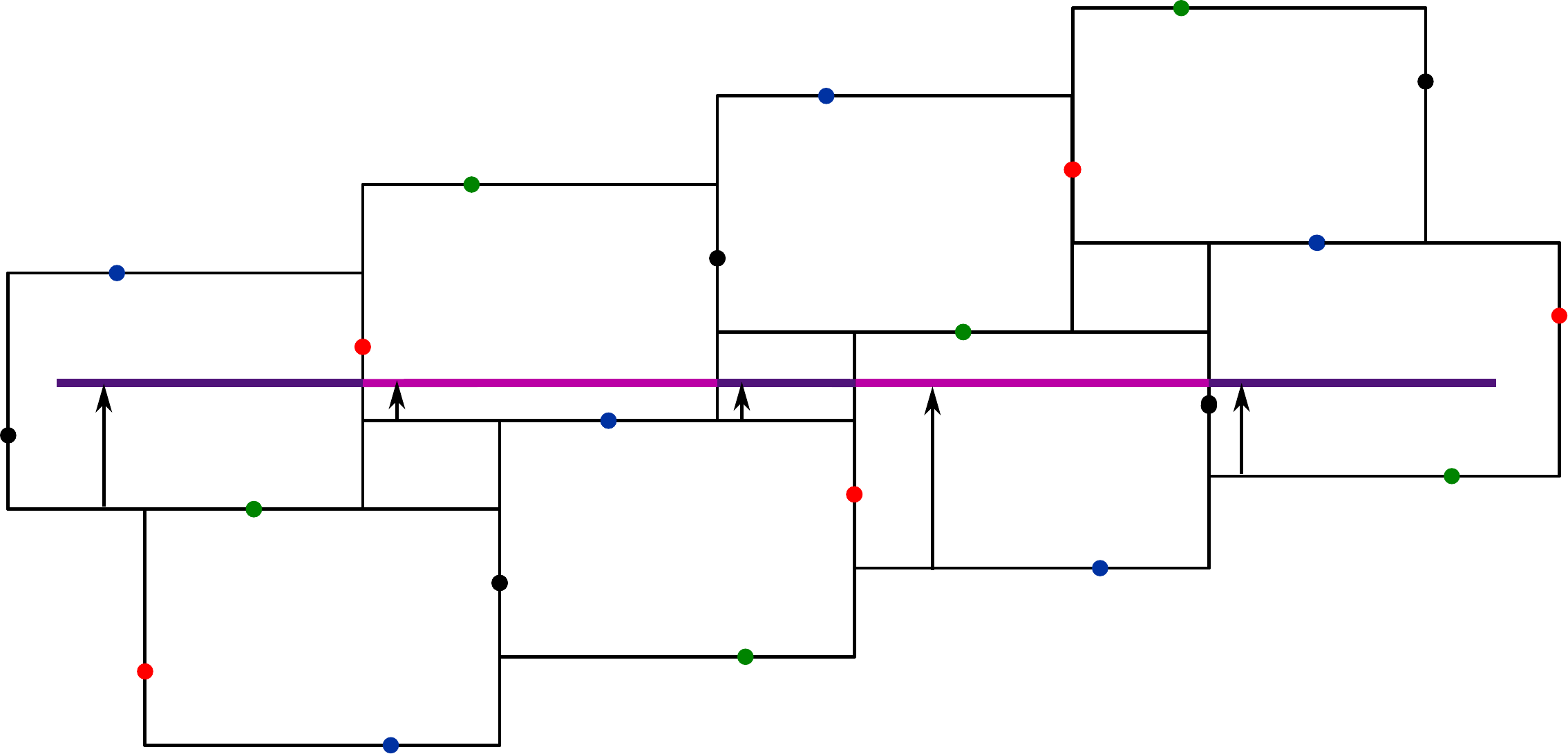}
		\caption{Decomposition of a horizontal segment}
		\label{fig:segment-opposite}
	\end{figure}
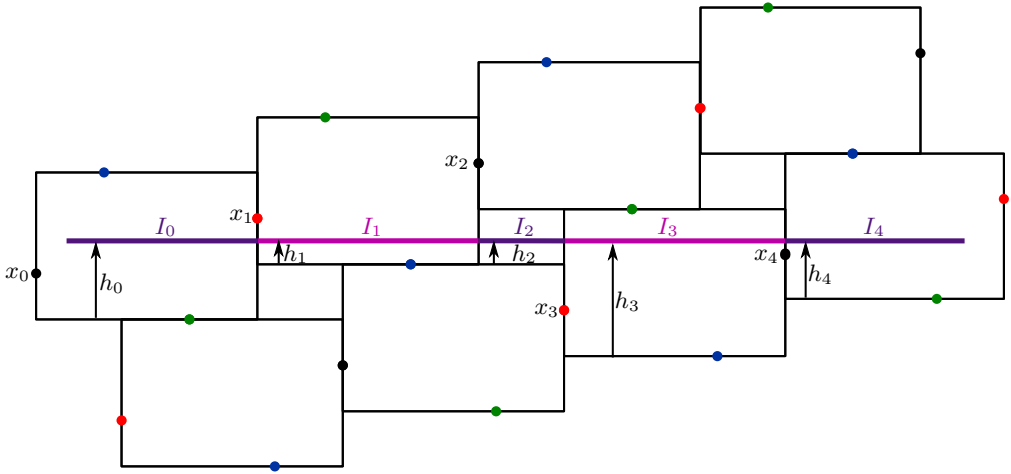
	
	See figure \ref{fig:segment-opposite} on page \pageref{fig:segment-opposite} for an illustration of Lemma \ref{intervals}.
	\begin{proof}
		This lemma follows directly from Lemma \ref{tile-opposite-sign}. Lemma \ref{tile-opposite-sign} gives a tiling of $\R^2$ using the rectangles $S(x,u,v)$ and $R(x,u,v)$ for $x \in \Lambda(t) \cup \Lambda(\inv{t})$. Then we consider the intersection of $I$ with this tiling, and this gives the decomposition $I=I_0 \cup I_1 \cup \hdots \cup I_p \cup I_p$. \\ 
		The third point comes from the fact that a rectangle $R(x,u,v)$ is surrounded by four rectangle of the form $S(y,u,v)$, so a segment of the form $I_R(x,u,v,h)$ cannot be followed by a segment of the form $I_R(x',u,v,h')$.
	\end{proof}
	
	We easily deduce from Lemma \ref{intervals} the following properties :
	
	\begin{Corollary} \label{coro-interval}
		With the notations of the previous Lemma, we have :
		\begin{enumerate}
			\item $\# \left\{j \in \{1,\hdots,p\} \mid I_j=I_S(x_j,u,v,h_j) \right\} \geq \# \{j \in \{1,\hdots,p\} \mid I_j=I_R(x_j,u,v,h_j)\}$
			\item The type of $x_j$ is alternating. Precisely : if 
			$x_0$ is of type $t$, then $x_j$ is of type $t$ if $j$ is even  and of type $\inv{t}$ if $j$ is odd, and if $x_0$ is of type $\inv{t}$, then $x_j$ is of type $\inv{t}$ of $j$ is even and of type $t$ if $j$ is odd. 
		\end{enumerate}
	\end{Corollary}
	
	Finally, we decompose any horizontal segment $I$ using the tiling in the case where $\slope(u+v)$ and $\slope(v-u)$ are of the same sign.
	\begin{Lemma} \label{int-same-sign}Assume that $\slope(u+v) \neq 0$, $\slope(v-u) \neq \infty$ and that $\slope(u+v)$ and $\slope(v-u)$ are of the same sign. \\
		Let $t$ be a type and let $I$ be a horizontal segment which does not intersect the lattice $\Lambda$. \\ 
		There exist $p \in \N$ such that for all $0 \leq j \leq p+1$, there exists a point $x_j \in \Lambda(t)\cup \Lambda(\inv{t})$ and a real $0<h_j<1$ such that :
		\begin{equation*}
			I=I_0 \cup I_1 \cup \hdots \cup I_p \cup I_{p+1}
		\end{equation*}
		with 
		\begin{itemize}
			\item $I_0 \subset I_S(x_0,u,v,h_0)$ and $I_{p+1} \subset I_S(x_{p+1},u,v,h_{p+1})$.
			\item For all $1 \leq j \leq p$, $I_j=I_S(x_j,u,v,h_j)$.
			\item Either $x_{j+1}=x_j+u+v$, and in this case $I_j \cap I_{j+1}$ is a point (which is the right endpoint of $I_j$ and the left endpoint of $I_{j+1}$),
			\item Or $x_{j+1}= \left\{ \begin{array}{cc}
				x_j+2u & \text{ if } \slope(u+v)>0 \\
				x_j+2v & \text{ if } \slope(u+v)<0 
			\end{array} \right. $, and in this case $I_j \cap I_{j+1} = I_j \cap S^+(x_j,u,v)$.
		\end{itemize}
	\end{Lemma}
	See figure \ref{fig:segment-same} on page \pageref{fig:segment-same} for an illustration of Lemma \ref{int-same-sign}.
	\begin{figure}[h]
		\centering
		\labellist
		\small\hair 2pt
		\pinlabel{$x_0$} at 60 165
		\pinlabel{$x_1$} at 310 225
		\pinlabel{$x_2$} at 560 285
		\pinlabel{$x_3$} at 640 190
		\pinlabel{$x_4$} at 890 220
		\pinlabel{$x_5$} at 1190 295
		\pinlabel{$h_0$} at 180 188
		\pinlabel{$h_1$} at 380 215
		\pinlabel{$h_2$} at 630 243
		\pinlabel{$h_3$} at 740 190
		\pinlabel{$h_4$} at 960 220
		\pinlabel{$h_5$} at 1190 230
		\pinlabel{$ \color{purple} I_0$} at 230 293
		\pinlabel{$\color{pink} I_1$} at 470 293
		\pinlabel{$\overbrace{\hspace{2.7cm}}$ } at 715 270
		\pinlabel{$\color{purple} I_2$} at 715 305
		\pinlabel{$\underbrace{\hspace{2.7cm}}$} at 790 245
		\pinlabel{$\color{pink} I_3$} at 790 210
		\pinlabel{$\color{purple} I_4$} at 1050 275
		\pinlabel{$\color{pink} I_5$} at 1290 280
		\endlabellist
		\includegraphics[width=\textwidth]{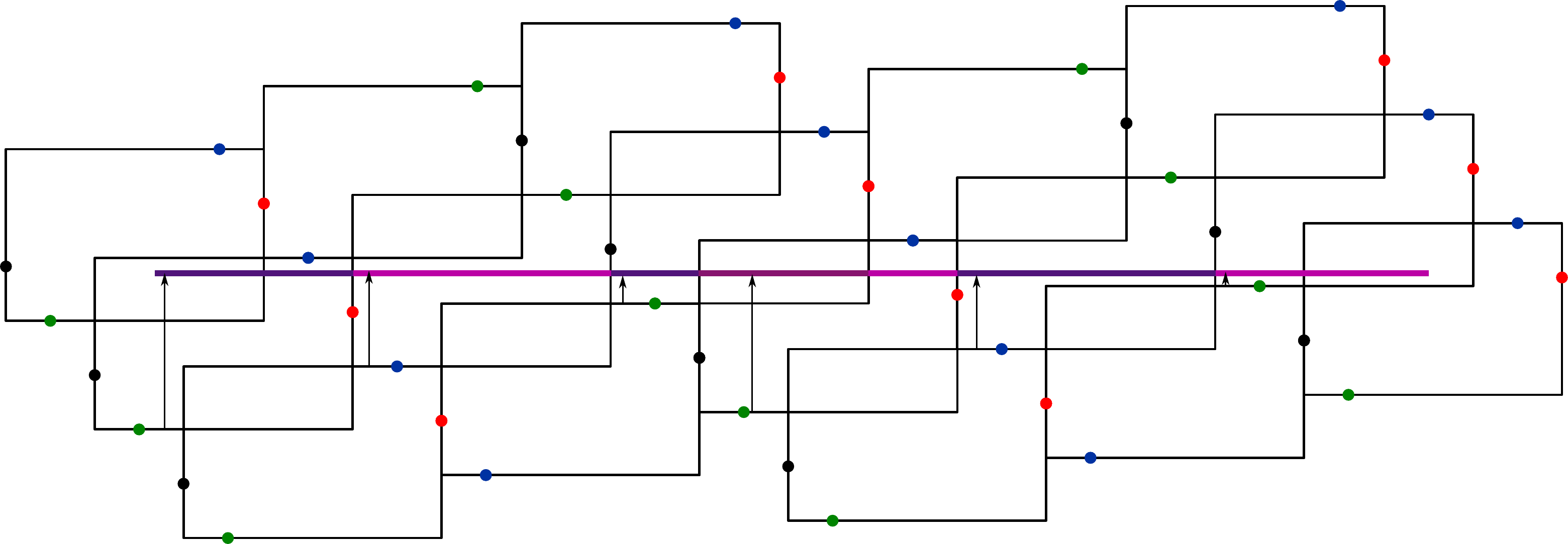}
		\caption{Decomposition of an horizontal segment}
		\label{fig:segment-same}
	\end{figure}
	
	\begin{proof}
		This Lemma follows directly from Lemma \ref{tile-same-sign}. Lemma \ref{tile-same-sign} gives a "tiling" of $\R^2$ (the term is improper in this context because the tiles can intersect in their interiors) using the rectangles $S(x,u,v)$ for $x \in \Lambda(t) \cup \Lambda(\inv{t})$. Then we consider the intersection of $I$ with this tiling, and this gives the desired decomposition $I=I_0 \cup I_1 \cup \hdots \cup I_p \cup I_{p+1}$. 
		
	\end{proof}
	
	\subsection{A few reminders about \emph{Farey neighbours}} 
	\label{farey}
	In the next section, we will in particular consider a lattice in $\R^2$ defined using the continued fraction expansion of a rational corresponding to a slope of a curve. Then we now need to say a few things about fractions, \emph{Farey neightbours} and the link with continued fraction expansion. \\ 
	
	In this section, we will consider fractions $\displaystyle \frac{p}{q}$, and we will always assume that they are written as irreducible fractions (which means that $p$ and $q$ are coprime) with non-negative denominator. By convention, we also consider the fraction $\displaystyle \frac{1}{0}$ which we will sometimes refer to as $+ \infty$. \\
	
	We say that two distinct fractions $\displaystyle \frac{p}{q}$ and $\displaystyle \frac{p'}{q'}$, are \emph{Farey neighbours} whenever $|pq'-pq|=1$. This is equivalent to saying that the two vectors $(q,p)$ and $(q',p')$ form a basis of $\Z^2$. In this case, we denote by $\displaystyle \frac{p}{q}\oplus \frac{p'}{q'}$ the \emph{Farey sum} of the two fractions $\displaystyle \frac{p}{q}$ and $\displaystyle \frac{p'}{q'}$, which is defined by $\displaystyle \frac{p}{q}\oplus \frac{p'}{q'}=\frac{p+p'}{q+q'}$ (note that this last fraction is automatically irreducible with positive denominator). The fractions $\displaystyle \frac{p}{q}$ and $\displaystyle \frac{p}{q} \oplus \frac{p'}{q'}$ are also Farey neighbours, and so are $\displaystyle \frac{p'}{q'}$ and $\displaystyle \frac{p}{q}\oplus \frac{p'}{q'}$. Suppose that $qq' \neq 0$ and $\displaystyle \frac{p}{q} \leq \frac{p'}{q'}$. Then we have :
	\begin{equation*}
		\frac{p}{q} \leq \frac{p}{q} \oplus \frac{p'}{q'} \leq \frac{p'}{q'}.
	\end{equation*}
	
	Now define $\displaystyle \frac{p}{q}\ominus \frac{p'}{q'}=\frac{p-p'}{q-q'}$. The fractions $\displaystyle \frac{p}{q}\ominus \frac{p'}{q'}$ and $\displaystyle \frac{p}{q}$ are Farey neighbours, and so are $\displaystyle \frac{p}{q}\ominus \frac{p'}{q'}$ and $\displaystyle \frac{p'}{q'}$. Note that with this definition $\ominus$ is commutative. \\
	We have either $\displaystyle (\frac{p}{q}\ominus \frac{p'}{q'})\oplus \frac{p'}{q'} = \frac{p}{q}$ or $\displaystyle (\frac{p'}{q'} \ominus \frac{p}{q}) \oplus \frac{p}{q} = \frac{p'}{q'}$. \\ ~ \\
	
	\textbf{Link with continued fraction expansion.} \\
	Let $n_1 \in \Z, n_2, \hdots, n_r \in \N^*$ with $n_r \geq 2$. Denote $\displaystyle \frac{p}{q} = [n_1,\hdots,n_r]$, $\displaystyle \frac{p_0}{q_0}=\frac{1}{0}$ and for all $1 \leq i \leq r$, $\displaystyle \frac{p_i}{q_i}=[n_1,\hdots,n_i]$. Then, it is not hard to check that for all $1 \leq i \leq r, \quad \displaystyle \frac{p_i}{q_i}$ and $\displaystyle \frac{p_{i-1}}{q_{i-1}}$ are Farey-neighbours. We can compute their Farey sum and Farey difference and we obtain : 
	\begin{align*}
		\frac{p_i}{q_i} \oplus \frac{p_{i-1}}{q_{i-1}}  & = [n_1,\hdots,n_i+1] \\
		\frac{p_i}{q_i} \ominus \frac{p_{i-1}}{q_{i-1}}  & = \left\{ \begin{array}{cc}
			[n_1,\hdots,n_i-1]  &  \text{ if } n_i \geq 2 \text{ or } i=1  \\
			{[n_1, \hdots, n_{i-2}]}  & \text{ if } n_i=1 \text{ and  } i \geq 3 \\ 
			\frac{1}{0}=\infty  & \text{ if } n_i=1 \text{ and } i=2
		\end{array} [n_1,\hdots,n_i-1] \right.
	\end{align*}
	Notice that in this case we always have $\displaystyle (\frac{p_i}{q_i}\ominus \frac{p_{i-1}}{q_{i-1}})\oplus \frac{p_{i-1}}{q_{i-1}} = \frac{p_i}{q_i}$. \\
	Let us make two observations : 
	\begin{align}
		\frac{p_i}{q_i} & \oplus \frac{p_{i-1}}{q_{i-1}} \neq \frac{p}{q}  \label{sum-farey} \\
		\frac{p_i}{q_i} & \ominus \frac{p_{i-1}}{q_{i-1}} = -1 \text{ if and only if } i=1 \text{ and } n_i=0 \label{diff-farey}
	\end{align}
	Moreover, since $\displaystyle \frac{p}{q} = [n_1,\hdots,n_r]$, we either have $\displaystyle [n_1,\hdots,n_i+1] \leq [n_1,\hdots,n_r] \leq [n_1,\hdots,n_i]$ or $\displaystyle [n_1,\hdots,n_i] \leq [n_1,\hdots,n_r] \leq [n_1,\hdots,n_i+1]$, that is :
	\begin{align*}
		\text{ either } \qquad \frac{p_i}{q_i}\oplus \frac{p_{i-1}}{q_{i-1}} \leq \frac{p}{q} \leq \frac{p_i}{q_i} \qquad
		\text{ or } \qquad \frac{p_i}{q_i} \leq \frac{p}{q} \leq \frac{p_i}{q_i}\oplus \frac{p_{i-1}}{q_{i-1}}.
	\end{align*}
	\begin{itemize}
		\item If $\displaystyle \frac{p_i}{q_i}\oplus \frac{p_{i-1}}{q_{i-1}} \leq \frac{p}{q} \leq \frac{p_i}{q_i}$, this forces $\displaystyle \frac{p_{i-1}}{q_{i-1}} \leq \frac{p_i}{q_i}\oplus \frac{p_{i-1}}{q_{i-1}} \leq \frac{p_i}{q_i}$ and so $\displaystyle \frac{p_i}{q_i} \ominus \frac{p_{i-1}}{q_{i-1}} \geq \frac{p_{i}}{q_{i}} \geq \frac{p}{q}$. \\
		Hence $\displaystyle \frac{p_i}{q_i} \oplus \frac{p_{i-1}}{q_{i-1}} \leq \frac{p}{q} \leq \frac{p_i}{q_i}\ominus \frac{p_{i-1}}{q_{i-1}}$.
		\item If $\displaystyle \frac{p_i}{q_i}\leq \frac{p}{q} \leq \frac{p_i}{q_i}\oplus \frac{p_{i-1}}{q_{i-1}} $, this forces $\displaystyle \frac{p_{i}}{q_{i}} \leq \frac{p_i}{q_i}\oplus \frac{p_{i-1}}{q_{i-1}} \leq \frac{p_{i-1}}{q_{i-1}}$ and so $\displaystyle \frac{p_i}{q_i} \ominus \frac{p_{i-1}}{q_{i-1}} \leq \frac{p_{i}}{q_{i}} \leq \frac{p}{q}$. \\
		Hence $\displaystyle \frac{p_i}{q_i} \ominus \frac{p_{i-1}}{q_{i-1}} \leq \frac{p}{q} \leq \frac{p_i}{q_i}\oplus \frac{p_{i-1}}{q_{i-1}}$.
	\end{itemize}
	Thus, in every case, the rationals $\displaystyle \frac{p_i}{q_i} \ominus \frac{p_{i-1}}{q_{i-1}}$ and $\displaystyle \frac{p_i}{q_i}\oplus \frac{p_{i-1}}{q_{i-1}}$ are on either side of $\displaystyle \frac{p}{q}$.
	
	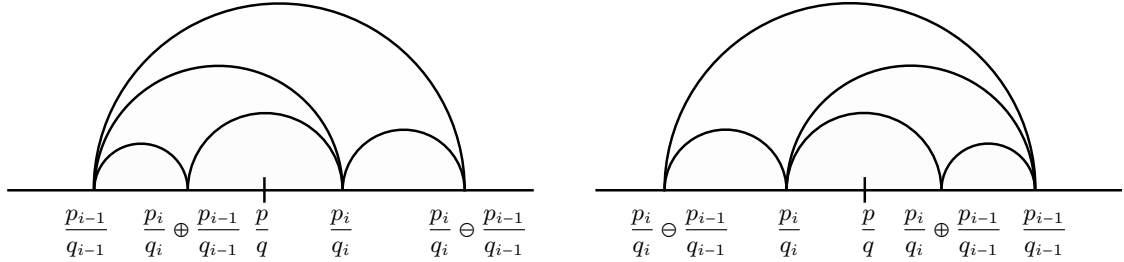
\begin{figure}[h]
		\centering
		\labellist
		\small \hair 2pt
		\pinlabel{$\displaystyle \frac{p_{i-1}}{q_{i-1}}$} at 18 -7
		\pinlabel{$\displaystyle \frac{p_i}{q_i} \oplus \frac{p_{i-1}}{q_{i-1}}$} at 42 -7
		\pinlabel{$\displaystyle \frac{p}{q}$} at 58 -7
		\pinlabel{$\displaystyle \frac{p_i}{q_i}$} at 76 -7 
		\pinlabel{$\displaystyle \frac{p_i}{q_i}\ominus \frac{p_{i-1}}{q_{i-1}}$} at 107 -7
		\pinlabel{$\displaystyle \frac{p_{i-1}}{q_{i-1}}$} at 236 -7
		\pinlabel{$\displaystyle \frac{p_i}{q_i} \oplus \frac{p_{i-1}}{q_{i-1}}$} at 215 -7
		\pinlabel{$\displaystyle \frac{p}{q}$} at 196 -7
		\pinlabel{$\displaystyle \frac{p_i}{q_i}$} at 178 -7 
		\pinlabel{$\displaystyle \frac{p_i}{q_i}\ominus \frac{p_{i-1}}{q_{i-1}}$} at 153 -7
		\endlabellist
		\includegraphics[width=7cm]{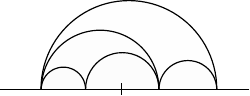} \qquad
		\includegraphics[width=7cm]{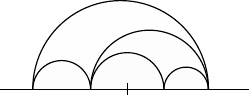}
		\vspace{1cm}
		\caption{The rational $\displaystyle \frac{p}{q}$ and its approximants $\displaystyle \frac{p_i}{q_i}, \frac{p_{i-1}}{q_{i-1}}, \frac{p_i}{q_i} \oplus \frac{p_{i-1}}{q_{i-1}}$ and $\displaystyle \frac{p_i}{q_i}\ominus \frac{p_{i-1}}{q_{i-1}}$ ordered on the real line.}
		\label{fig:Farey}
	\end{figure}
	
	\newpage
	\subsection{Special-lengths Theorem} 
	\label{magic-sphere}
	
	For $\gamma \in \simple$ a simple closed curve on $S_{0,4}$ of slope $[n_1(\gamma),\dots,n_{r(\gamma)}(\gamma)]$, recall that we have defined the lengths $l_i(\gamma)$ and $l'_i(\gamma)$ in Definition \ref{def:magic-length}. These lengths will play an important role in our study of the redundancy of arbitrary subwords of simple words. Indeed, the following theorem, to which this section is dedicated, says that  these lengths  are "special" in some sense :  we prove that the subwords of $\gamma \in \simple$ of these specific lengths (in fact precisely $l_i(\gamma)-5$) are redundant in the sense that they and their inverses can be found many other times in the word $\gamma$, and their appearances can be quantified by a uniform constant. 
	
	\begin{Theorem}[Special-lengths]\label{magic-len-S_{0,4}}
		Set $\alpha=\frac{1}{30}$. Let $\gamma \in \simple$. Consider $\slope(\gamma)=[n_1(\gamma),\hdots,n_{r(\gamma)}(\gamma)]$ the continued fraction expansion of $\slope(\gamma)$. Fix $0 \leq i \leq r(\gamma)$ such that $l_i(\gamma) \geq 10$ (in particular, note that by inequality \eqref{i<li}, this condition is fulfilled whenever $i \geq 10$). Let $W \in \F_3$ be a subword of (a cyclic-permutation of) $\gamma$ (or its inverse) such that $|W| \geq 3(l'_i(\gamma)+l_i(\gamma)+1)$ and $w \in \F_3$ be a subword of (a cyclic-permutation of) $\gamma$ (or its inverse) of length $l_i(\gamma)-5$.\\
		Then we can decompose the word~$W$ as a concatenation $W=u_1 \hdots u_q$ such that there exists a subset $\cI \subset \{1,\hdots,q\}$ satisfying the following :
		\begin{enumerate}
			\item \label{change-letter} For all $k \in \cI$, $u_k\in \{w,w^{-1}\}$.
			\item \label{proportion} $\displaystyle \sum_{k\in \cI} |u_k| \geq \alpha |W|$.
		\end{enumerate}
	\end{Theorem}
	
	We will prove Theorem \ref{magic-len-S_{0,4}} in section \ref{proof-magic-len}. For this purpose, given a simple closed curve $\gamma \in \simple$ and a fixed integer $0 \leq i \leq r(\gamma)$, we will define in section \ref{mapping} a lattice $\Lambda$ of $\R^2$ together with a basis $(u,v)$ of $\Lambda$ which will be particularly convenient for studying subwords of $\gamma$ of length $l_i(\gamma)$. This will allow us to "read" the word $\gamma$ by following a horizontal line in $\R^2$ (section \ref{reading-gamma}) and then we will consider for this specific choice of $\Lambda$ and $(u,v)$ the tiles $S(x,u,v)$ of $\R^2$ defined in section \ref{lattices} to study some specific subwords of $\gamma$ : the one read in the rectangle $S(x,u,v)$, which will be of lengths approximately $l'_i(\gamma)$ (section \ref{subwords-square}).
	
	\subsubsection{A tiling of $\R^2$}  
	\label{mapping}
	In this section, we associate to a simple closed curve $\gamma \in \simple$ and an integer $1 \leq i \leq r(\gamma)$ a lattice $\Lambda$ of $\R^2$. \\ 
	
	Fix $\gamma \in \simple$ a simple closed curve. Let us assume that $\slope(\gamma) \geq 0$ and write $[n_1,\hdots,n_r]=\slope(\gamma)$ the continued fraction expansion of $\slope(\gamma)$. For $1 \leq i < r$, we write $\displaystyle \frac{p_i}{q_i}=[n_1,\hdots,n_i]$ and for $i=0$, $\displaystyle \frac{p_{0}}{q_{0}}=\frac{1}{0}$. Now we fix $1 \leq i <r$ for all this section and if $i=1$, we assume that $n_i \neq 0$. \\
	Denote $u=(q_i,p_i) \in \Z^2$ and $v=(q_{i-1},p_{i-1}) \in \Z^2$. Notice that, because of the assumption $i<r$, we have $\slope(u) \neq \slope(\gamma)$ and $\slope(v) \neq \slope(\gamma)$. After possibly exchanging $u$ and $v$, suppose that $\slope(u) <\slope(v)$. Then, we have :
	\begin{enumerate}
		\item $\slope(u) < \slope(\gamma) < \slope(v)$.
		\item $(u,v)$ form a basis of $\Z^2$
	\end{enumerate}
	Moreover, using the equalities \eqref{sum-farey} and \eqref{diff-farey} of section \ref{farey} and the hypothesis $(i,n_i) \neq (1,0)$, we have :
	\begin{enumerate}[resume]
		\item $\displaystyle \slope(u+v)= \frac{p_i}{q_i}\oplus \frac{p_{i-1}}{q_{i-1}} \neq \slope(\gamma)$
		\item $\displaystyle \slope(v-u)=\frac{p_i}{q_i}\ominus \frac{p_{i-1}}{q_{i-1}} \neq -1$ \\
	\end{enumerate}
	
	Let us write as irreducible fractions :
	\begin{align*} 
		\frac{p}{q} =\slope(\gamma), \quad \frac{p_i}{q_i}= \slope(u), \quad \frac{p'_i}{q'_i}= \slope(v)
	\end{align*}
	
	\textbf{The map $f_{\gamma,u,v}$}. ~ \\
	Consider $f_{\gamma,u,v}$ an element of $\mathrm{GL}(2,\R)$ which sends the lines of slope $\slope(\gamma)$ in $\R^2$ to the lines of slope $0$ and the lines of slope $-1$ in $\R^2$ to the lines of slope $+\infty$. Choose $f_{\gamma,u,v}$ such that $f_{\gamma,u,v}$ preserves the orientation, and scale it so that the abscissa of the vector $f_{\gamma,u,v}(u+v)$ is $1$ and the ordinate of the vector $f_{\gamma,u,v}(v-u)$ is also $1$ (such a map always exists). Let us give $f_{\gamma,u,v}$ explicitly. The map $f_{\gamma,u,v}$ is the product of a diagonal matrix $D$ with a matrix $M$ which sends lines of slope $-1$ to lines of slope $+ \infty$ and lines of slope $\displaystyle \frac{p}{q}$ to lines of slope $0$. Therefore for $M$ we can choose the inverse of the matrix $\begin{pmatrix}
		q & -1 \\ p & 1
	\end{pmatrix}$ which gives $ \displaystyle
	M=\frac{1}{p+q}\begin{pmatrix}
		1 & 1 \\ -p & q 
	\end{pmatrix}
	$. Then $ \displaystyle Mu=\frac{1}{p+q} \begin{pmatrix}
		p_i+q_i \\ -pq_i+qp_i
	\end{pmatrix}$ and $ \displaystyle Mv=\frac{1}{p+q} \begin{pmatrix}
		p'_i+q'_i \\ -pq'_i+qp'_i
	\end{pmatrix}$
	and in order to get the right rescaling we need to impose : $\displaystyle D = \displaystyle (p+q) \begin{pmatrix}
		\frac{1}{p_i+q_i+p'_i+q'_i} & 0 \\ 0 & \frac{1}{p(q_i-q'_i)-q(p_i-p'_i)}
	\end{pmatrix} $. Hence we obtain : 
	\begin{equation}
		f_{\gamma,u,v}=\begin{pmatrix}
			\frac{1}{p_i+q_i+p'_i+q'_i} & \frac{1}{p_i+q_i+p'_i+q'_i} \\
			\frac{-p}{p(q_i-q'_i)-q(p_i-p'_i)} & \frac{q}{p(q_i-q'_i)-q(p_i-p'_i)}
		\end{pmatrix}
	\end{equation}
	Also note that $(p_i+q_i,p'_i+q'_i)\in \{ (l_i(\gamma),l_{i-1}(\gamma)),(l_{i-1}(\gamma),l_i(\gamma)) \}$ and then $p_i+q_i+p'_i+q'_i=l_i(\gamma)+l_{i-1}(\gamma)=l'_i(\gamma)$. Now we can compute $f_{\gamma,u,v}(u)$ and $f_{\gamma,u,v}(v)$ : 
	\begin{align*}
		f_{\gamma,u,v}(u) & = \left( \frac{p_i+q_i}{p_i+q_i+p'_i+q'_i},\frac{-pq_i+qp_i}{p(q_i-q'_i)-q(p_i-p'_i)} \right) \\
		f_{\gamma,u,v}(v) & = \left( \frac{p'_i+q'_i}{p_i+q_i+p'_i+q'_i},\frac{-pq'_i+qp'_i}{p(q_i-q'_i)-q(p_i-p'_i)} \right)
	\end{align*}
	We deduce : 
	\begin{equation} \label{max-u-v}
		\max((f_{\gamma,u,v}(u))_1,(f_{\gamma,u,v}(v))_1) = \frac{\max(p_i+q_i,p'_i+q'_i)}{p_i+q_i+p'_i+q'_i}=\frac{\max(l_i(\gamma),l_{i-1}(\gamma))}{l'_i(\gamma)}=\frac{l_i(\gamma)}{l'_i(\gamma)}
	\end{equation}
	
	Since we have $\slope(u) < \slope(\gamma) < \slope(v)$, we deduce that $\slope(f_{\gamma,u,v}(u)) < 0 < \slope(f_{\gamma,u,v}(v))$. We had also noticed that $\slope(u+v) \neq \slope(\gamma)$ and $\slope(v-u)\neq -1$, hence we can also deduce that 
	\begin{align}
		\label{slope(u+v)}
		\slope(f_{\gamma,u,v}(u+v)) & \neq 0 \\
		\slope(f_{\gamma,u,v}(v-u)) & \neq \infty
	\end{align}
	Note also that the slope of $f_{\gamma,u,v}(u+v)$ and $f_{\gamma,u,v}(v-u)$ are of opposite signs. Indeed, before applying the map $f_{\gamma,u,v}$, the slope of $u+v$ is $\displaystyle \frac{p_i}{q_i}\oplus \frac{p_{i-1}}{q_{i-1}}$ and the slope of $v-u$ is $\displaystyle \frac{p_i}{q_i}\ominus \frac{p_{i-1}}{q_{i-1}}$. By the result of the end of section \ref{farey}, the rationals $\displaystyle \frac{p_i}{q_i}\oplus \frac{p_{i-1}}{q_{i-1}}$ and $\displaystyle \frac{p_i}{q_i}\ominus \frac{p_{i-1}}{q_{i-1}}$ are on either side of $\displaystyle \frac{p}{q}$, (with $\displaystyle \frac{p}{q}$ the slope of $\gamma$), hence after applying the map $f_{\gamma,u,v}$, we deduce that these two slopes are of opposite signs, since the slope $\displaystyle \frac{p}{q}$ is sent to $0$ by $f_{\gamma,u,v}$. \\
	
	The map $f_{\gamma,u,v}$ sends the lattice $\Z^2$ to another lattice in $\R^2$, denote it by $\Lambda=f_{\gamma,u,v}(\Z^2)$. Hence $2\Lambda=f_{\gamma,u,v}(2\Z^2)$. Recall that we had defined the type of a point in $\Z^2$, we can then also define the \emph{type} of a point in $\Lambda$ as the type of the corresponding point in $\Z^2$. We obtain four types of points in $\Lambda$ given by the four classes of the elements of $\Lambda$ mod $2\Lambda$. We had also defined three sets of lines~: $L_A,L_B$ and $L_C$ so we can look at their images by $f_{\gamma,u,v}$. Since the lines in $L_C$ have slope $-1$, the lines in $f_{\gamma,u,v}(L_C)$ have slope $+ \infty$. Also note that a line of slope $\slope(\gamma)$ has for image a line of slope $0$. Thus the lines in $f_{\gamma,u,v}(L_A)$ have a negative slope and the lines in $f_{\gamma,u,v}(L_B)$ have a positive slope (recall that the lines in $L_A$ have slope $0$ and the lines in $L_B$ have slope $+ \infty$). Two non parallel lines in $f_{\gamma,u,v}(L_A)\cup f_{\gamma,u,v}(L_B) \cup f_{\gamma,u,v}(L_C)$ always intersect in a point of $\Lambda$. \\
	
	In order to simplify the notation, from now on we simply denote by $L_A,L_B,L_C,u,v$ the images of $L_A,L_B,L_C,u,v$ by $f_{\gamma,u,v}$. \\
	
	Now that the lattice $\Lambda$ is fixed as well as the basis $(u,v)$ of $\Lambda$ satisfying $\slope(u) < 0 < \slope(v)$, we can use the notations of section \ref{lattices} to talk about the rectangle $S(x,u,v)$ and $R(x,u,v)$, for a point $x \in \Lambda$.
	Note that, by the choice of the scaling of $f_{\gamma,u,v}$, the lengths of the sides of $S(x,u,v)$ are $1$. 
	
	\subsubsection{Reading the word $\gamma$ and its subwords}
	\label{reading-gamma}
	
	We can still read the word corresponding to $\gamma$ in the fundamental group $\groupe$ using the lattice $\Lambda$ and the three new sets of lines $L_A, L_B$ and $L_C$ : Follow a line $l_0$ of slope $0$ in $\R^2$ which avoids the lattice $\Lambda$ and record $a,b$ or $c$ each time the line $l_0$ crosses a line $l$ of $L_A,L_B$ or $L_C$ with power $\pm 1$ depending on the transverse orientation of $l$ at the intersection point with $l_0$. To obtain the full word $\gamma$ (more precisely a cyclic permutation of $\gamma$ or its inverse), we have to follow $l_0$ between two points which are identified by an element of $2\mathcal{P}(\Lambda)$, that is between two points $x$ and $y$ on $l_0$ such that $y=x+2\lambda$, with $\lambda \in \mathcal{P}(\Lambda)$ (here $\mathcal{P}(\Lambda)$ denotes the set of primitive elements in $\Lambda$, that is the elements in $\Lambda$ which can be completed to a basis of $\Lambda$). More generally, by following $l_0$ along a segment $I \subset [x,y] \subset l_0$, we read a subword of $\gamma$ (or a subword of a cyclic-permutation of $\gamma$ or its inverse). We denote by $W(I)$ this subword. We are now interested in the link between the length of the segment $I$ and the word length of the word $W(I)$ we read by following $I$. More precisely, we need the following lemma : 
	\begin{Lemma}
		\label{comp-word-interval}
		Let $W$ be a subword of $\gamma$ (or more generally a subword of a cyclic permutation of $\gamma$ or its inverse). There exists a horizontal segment $I$ in the plane $\R^2$ such that we read $W$ by following $I$ and, denoting $l(I)$ the length of the segment $I$ and $|W|$ the word length of $W$, we have~:
		\begin{equation*}
			|W|-3 \leq l(I)l'_i(\gamma) \leq |W|+1
		\end{equation*}
	\end{Lemma}
	Remember that the integer $i$ which appears in this Lemma \ref{comp-word-interval} has been fixed at the beginning of section \ref{mapping}. Everything that has been defined using $i$, therefore depends on it : the map $f_{\gamma,u,v}$, the lattice $\Lambda$, the three sets of lines $L_A,L_B$ and $L_C$, etc... We have chosen not to index these objects by $i$ in order to reduce the amount of notation.
	\begin{proof}
		Let us start by computing the spacing between two consecutive lines in $L_C$, and denote it by $E_C$. Recall that $L_C$ consists of parallel lines of slope $+ \infty$ which are the images of the lines of equations $y=-x+2k, k\in \Z$, under the map $f_{\gamma,u,v}$. Then, since the point $(x,y)=(\frac{2q}{p+q},\frac{2p}{p+q})$ belongs both to the line of equation $y=-x+2$ and to the line of equation $y=\frac{p}{q}x$, its image by $f_{\gamma,u,v}$ belongs both to the line of equation $y=0$ and to the line of equation $x=E_C$. We compute $f_{\gamma,u,v}(x,y)$ and we obtain $(\frac{2}{l'_i(\gamma)},0)$. Hence $E_C=\frac{2}{l'_i(\gamma)}$. \\
		Now, since $W$ is a subword of $\gamma$, we can read it somewhere following a subsegment $I$ of $l_0$, that is a horizontal segment. 
		Now remember that since $\slope(\gamma) \geq 0$, the letters of $\gamma$ alternate between letters in $\{c,c^{-1}\}$ and letters in $\{a,a^{-1},b,b^{-1}\}$. This means that $I$ crosses exactly one line in $L_A \cup L_B$ between two lines in $L_C$. Denote $n_C$ the number of intersection points between $I$ and $L_C$, or equivalently the number of $c$ and $c^{-1}$ in $W$. 
		\begin{itemize}
			\item     If $|W|$ is even, then $|W|=2n_C$, and we can choose $I$ such that :
			\begin{equation*}
				(n_C-1)E_C \leq l(I) \leq n_CE_C, \quad \text{ which is equivalent to } \quad (|W|-2)E_C \leq 2l(I) \leq |W|E_C. 
			\end{equation*}
			\item If $|W|$ is odd and $|W|=2n_C+1$, then we can choose $I$
			such that :
			\begin{equation*}
				(n_C-1)E_C \leq l(I) \leq (n_C+1)E_C, \quad \text{ which is equivalent to } \quad (|W|-3)E_C \leq 2l(I) \leq (|W|+1)E_C.
			\end{equation*}
			\item If $|W|$ is odd and $|W|=2n_C-1$, then we can choose $I$ such that :
			\begin{equation*}
				l(I)=n_CE_C, \quad \text{ which is equivalent to } \quad 2l(I)=(|W|+1)E_C.
			\end{equation*}
		\end{itemize}
		Therefore, in every cases we have 
		\begin{equation*}
			(|W|-3)E_C \leq 2l(I) \leq (|W|+1)E_C
		\end{equation*}
		and using the previously calculated value of $E_C$, we obtain the required inequality. 
	\end{proof}
	
	\subsubsection{Subwords associated to a square}
	\label{subwords-square}
	By the previous construction, following the horizontal line $l_0$ along the segment $I$ gives (a cyclic permutation of) the word $\gamma$ (or its inverse). Hence by following a subsegment of $I$, we obtain a subword of $\gamma$. We will now look at the subwords of $\gamma$ that we obtain by crossing a square $S(x,u,v)$ and a rectangle $R(x,u,v)$, that is by following the intervals $I_S(x,u,v,h)$ and $I_R(x,u,v,h)$ defined in section \ref{subsubsec:segments-rectangle}. \\
	
	Now assume that the real $h$ ($\in (0,1)$) is such that the segment $I_S(x,u,v,h)$ (respectively $I_R(x,u,v,h)$) does not intersect the lattice $\Lambda$ and therefore, let us define $W_S(x,u,v,h)$ (respectively $W_R(x,u,v,h)$) to be the subword of $\gamma$ read when following $I_S(x,u,v,h)$ (respectively $I_R(x,u,v,h)$). There is a small ambiguity that we must resolve : if one of the endpoints of $I_S(x,u,v,h)$ (resp. $I_R(x,u,v,h)$) is an intersection point of $I_S(x,u,v,h)$ (resp. $I_R(x,u,v,h)$) with a line in $L_A \cup L_B \cup L_C$, then we record the corresponding letter in $W_S(x,u,v,h)$ (resp. $W_S(x,u,v,h)$). Now let us make a simple but important observation~: because of the invariance of our setting by the action of $2\Lambda$, if $x$ and $x'$ are two points of $\Lambda$ of the same type, then $W_S(x,u,v,h)=W_S(x',u,v,h)$ and $W_R(x,u,v,h)=W_R(x',u,v,h)$. Hence the subword of $\gamma$ read in a square or a rectangle only depends on the type of the square and the height~$h$. So if $t \in \Lambda / 2\Lambda$ is the type of a point, we will write $W_S(t,u,v,h)$ and $W_R(t,u,v,h)$ without ambiguity. \\
	
	The following lemma states that the subwords of $\gamma$ read in some square (resp. rectangle) are the inverses of the subwords read in a square (resp. rectangle) of opposite type : 
	
	\begin{Lemma} \label{inverses}
		Let $t \in \Lambda / 2\Lambda$ be a type and $\inv{t}$ its $(u,v)$-opposite type. Let $0<h<1$ and $0 < h' < u_2+v_2$ such that if $x$ is a point of $\Lambda$ of type $\inv{t}$, then $I_S(x,u,v,h)$ and $I_R(x,u,v,h')$ does not intersect the lattice $\Lambda$. Then $I_S(x+u+v,u,v,1-h)$ and $I_R(x+u+v,u,v,u_2+v_2-h')$ does not intersect the lattice $\Lambda$ and ~:
		\begin{align*}      
			W_S(\inv{t},u,v,h)& =W_S(t,u,v,1-h)^{-1} \\
			W_R(\inv{t},u,v,h')& =W_R(t,u,v,u_2+v_2-h')^{-1}
		\end{align*}
	\end{Lemma}
	
	\begin{proof}
		
		\begin{figure}[h]
			\centering
			\labellist
			\small \hair 2pt
			\pinlabel{$x$} at -10 70
			\pinlabel{$ \color{purple} I_S(x,u,v,h)$} at 130 160 
			\pinlabel{$\color{purple} I_S(x+u+v,u,v,1-h)$} at 340 140 
			\pinlabel{$h$} at 50 100
			\pinlabel{$1-h$} at 270 105
			\endlabellist
			\includegraphics[width=6cm]{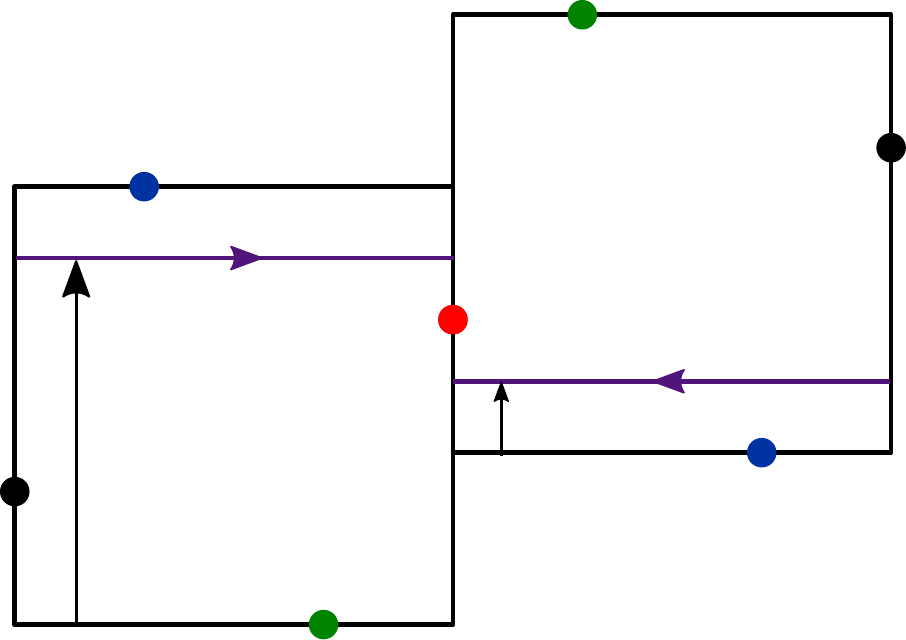}
			\caption{Proof of Lemma \ref{inverses}}
			\label{fig:word-inv}
		\end{figure}
		Let $x$ be a point of $\Lambda$ of type $\inv{t}$. Consider the reflection of the square $S(x,u,v)$ (resp. rectangle $R(x,u,v)$) across the point $x+u+v$ (resp. $x+u$). This gives the square $S(x+u+v,u,v)$ (resp. rectangle $R(x+u+v,u,v))$, and then this square (resp. rectangle) is of type $t$. Now consider the segment $I_S(x,u,v,h)$ in $S(x,u,v)$ (resp. $I_R(x,u,v,h')$ in $R(x,u,v)$) with its orientation from left to right. Then its reflection across the point $x+u+v$ (resp. $x+u$) is the segment $I_S(x+u+v,u,v,1-h)$ in $S(x+u+v,u,v)$ (resp. $I_R(x+u+v,u,v,u_2+v_2-h')$ in $R(x+u+v,u,v)$) with the reverse orientation : from right to left. Since our setting is invariant under a reflection across a point of the lattice $\Lambda$, we deduce that the subword read by following $I_S(x,u,v,h)$ (resp. $I_R(x,u,v,h')$) from left to right is the same as the subword read by following $I_S(x+u+v,u,v,1-h)$ (resp. $I_R(x+u+v,u,v,u_2+v_2-h')$) from right to left, thus it is the inverse of the subword read by following $I_S(x+u+v,u,v,1-h)$ (resp. $I_R(x+u+v,u,v,u_2+v_2-h')$) from left to right. Hence the lemma. 
	\end{proof}
	
	Now we want to investigate the dependence on the height $h$ of the word $W_S(t,u,v,h)$. Precisely, we want to show that, up to changing a few letters at the beginning and at the end of the word $W_S(t,u,v,h)$, it does not depend on the height $h$.
	\begin{Lemma} \label{W_S}
		Let $t \in \Lambda/2\Lambda$ be a type and $h,h'$ be two heights. Then, up to adding a letter at the beginning of $W_S(t,u,v,h)$, or removing the first letter of $W_S(t,u,v,h)$, or changing the two first letters of $W_S(t,u,v,h)$, and up to adding a letter at the end of $W_S(t,u,v,h)$, or removing the last letter of $W_S(t,u,v,h)$, or changing the two last letters of $W_S(t,u,v,h)$, we have :
		\begin{equation*}
			W_S(t,u,v,h)=W_S(t,u,v,h')
		\end{equation*}
	\end{Lemma}
	
	\begin{proof}
		Let $x$ be a point of type $t$ in $\Lambda$. We will work in the square $S(x,u,v)$. We want to understand the intersection of lines in $L_A \cup L_B \cup L_C$ which occurs in the square $S(x,u,v)$. Let us prove the following fundamental observations~:
		\begin{enumerate}
			\item If two lines of $L_A \cup L_B \cup L_C$ intersect in $S(x,u,v)$, then they intersect in $\{x,x+u,x+v,x+u+v\}$. 
			\item If a line of $L_A \cup L_B \cup L_C$ crosses the left side of the square $S(x,u,v)$, then it passes through the point $x$.
			\item If a line of $L_A \cup L_B \cup L_C$ crosses the right side of the square $S(x,u,v)$, then it passes through the point $x+u+v$.
		\end{enumerate}
		\begin{proof}[Proof of the observations]
			\begin{enumerate}
				\item Two lines of $L_A \cup L_B \cup L_C$ can only intersect in a point of the lattice $\Lambda$. But recall that the only points of the lattice $\Lambda$ contained in $S(x,u,v)$ are the four points $x,x+u,x+v$ and $x+u+v$. Hence the first observation. 
				\item Consider $L_{\infty}$ the set of vertical lines (of slope $\infty$) which passes through a point of $\Lambda$. Of course $L_C \subset L_{\infty}$. Now note that if $l\in L_{\infty}$ and $l' \in L_A \cup L_B$, then the intersection point between $l$ and $l'$ belongs to the lattice $\Lambda$ : $l\cap l' \subset \Lambda$. Since $x \in \Lambda$, we deduce that the left vertical side of the square $S(x,u,v)$ is included in a line $l \in L_{\infty}$. Therefore any line of $L_A \cup L_B$ which crosses the left side of the square $S(x,u,v)$ intersects it in a point of $\Lambda$. We conclude by recalling that $x$ is the only point of the lattice $\Lambda$ on the left vertical side of $S(x,u,v)$. To finish the proof, note that if a line of $L_C$ crosses the left side of the square $S(x,u,v)$, it is trivial that it passes through the point $x$, since in that case the line would be vertical and hence would contained the left side of the square $S(x,u,v)$.
				\item The proof is exactly the same as for the previous point, since the point $x+u+v$ of the lattice $\Lambda$ belongs to the right side of $S(x,u,v)$ and is the only point of the lattice $\Lambda$ on the right side of $S(x,u,v)$.
			\end{enumerate}
		\end{proof}
		From these three observations we can deduce the Lemma \ref{W_S}. Let us first draw a picture. Figure \ref{fig:W_S-full} on page \pageref{fig:W_S-full} shows the square $S(x,u,v)$ together with the lines of $L_A \cup L_B \cup L_C$ that cross it. Lines in $L_A$ are in green in the figure : they have negative slope and are regularly spaced. Lines in $L_B$ are in blue in the figure : they have positive slope and are regularly spaced. Lines in $L_C$ are in blue in the figure : they are vertical and regularly spaced. Note that between two lines in $L_A \cup L_B$ there is always exactly one line in $L_C$. \\
		
			\begin{figure}[h]
			\centering
			\labellist
			\small \hair 2pt
			\pinlabel{$-u_2$} at -10 35
			\pinlabel{$v_2$} at 265 80
			\pinlabel{$x$} at 10 85
			\pinlabel{$x+u$} at 170 -5
			\pinlabel{$x+v$} at 90 245
			\pinlabel{$x+u+v$} at 270 175
			\pinlabel{$\color{purple} W_S(t,u,v,h_1)$} at 130 205
			\pinlabel{$\color{purple} W_S(t,u,v,h_2)$} at 130 130
			\pinlabel{$\color{purple} W_S(t,u,v,h_3)$} at 130 55
			\pinlabel{$\color{purple}W_S(t,u,v,h_1) \color{black} =  c^{-1}b^{-1}\,(\t{c}b^{-1}\t{c}\t{b}\t{c}a\t{c}\t{b})\,c$.} at 377 195 
			\pinlabel{$\color{purple} W_S(t,u,v,h_2) \color{black} = c^{-1}b^{-1}\,(\t{c}b^{-1}\t{c}\t{b}\t{c}a\t{c}\t{b})\,c^{-1}$.} at 385 115 
			\pinlabel{$\color{purple} W_S(t,u,v,h_3) \color{black} =ba\,(\t{c}b^{-1}\t{c}\t{b}\t{c}a\t{c}\t{b})\,c^{-1}$.} at 375 45
			\endlabellist
			\includegraphics[width=6cm]{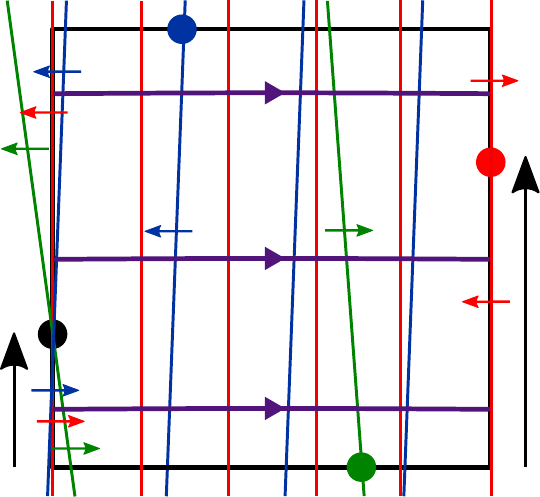} \hspace{3cm}
			\caption{Proof of Lemma \ref{W_S}. Reading the subword $W_S(t,u,v,h)$ in the square $S(x,u,v)$. \\ Lines in $L_A$ are in green, lines in $L_B$ in blue and lines in $L_C$ in red. We draw the case where $x$ is of type $abc$, $x+u$ of type $a$, $x+v$ of type $b$ and $x+u+v$ of type $c$. Up to changing the two first letters and inverting the last one, the word $W(t,u,v,h)$ does not depend on the height $h$.}
			\label{fig:W_S-full}
		\end{figure}
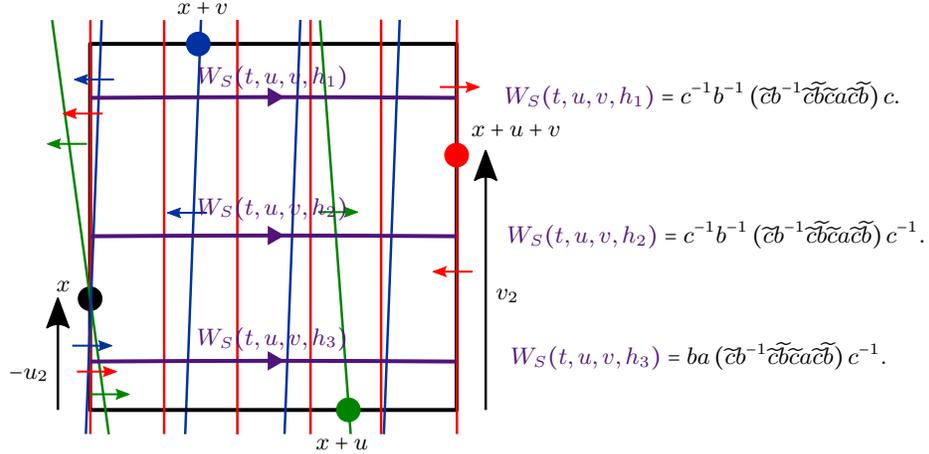
		
		Consider two heights $h$ and $h'$ and $I_S(x,u,v,h)$ and $I_S(x,u,v,h')$ the corresponding segments in $S(x,u,v)$. Let $l$ and $l'$ be two lines of $L_A \cup L_B \cup L_C$ which cross the square $S(x,u,v)$ but cross neither the left side nor the right side of $S(x,u,v)$. Then $l$ and $l'$ do not intersect in the interior of the square $S(x,u,v)$. Therefore, by following the segments $I_S(x,u,v,h)$ and $I_S(x,u,v,h')$ from left to right, we cross $l$ and $l'$ in the same order. Moreover, there is no point of the lattice $\Lambda$ in the intersection of $l$ and the interior of $S(x,u,v)$. So the segments $I_S(x,u,v,h)$ and $I_S(x,u,v,h')$ cross the line $l$ with the same orientation. We deduce that the only differences between the words $W(t,u,v,h)$ and $W(t,u,v,h')$ must occur at the beginning and at the end of these words (and correspond to lines of $L_A \cup L_B \cup L_C$ which passes through $x$ and $x+u+v$). Indeed, suppose $h>-u_2$ and $h'<-u_2$, and let us distinguish according to the type of $x$, which is $t$ :
		\begin{itemize}
			\item If $t=a$, then there is exactly one line of $L_A \cup L_B \cup L_C$ passing through $x$ and it is a line of $L_A$. Since the slope of the lines in $L_A$ is negative, we deduce that we must add a letter at the beginning of $W_S(x,u,v,h)$ (which will be $a^{-1}$) to obtain the same beginning as the word $W_S(x,u,v,h')$. 
			\item If $t=b$, then there is exactly one line of $L_A \cup L_B \cup L_C$ passing through $x$ and it is a line of $L_B$. Since the slope of the lines in $L_B$ is positive, we deduce that we must add a letter at the beginning of $W_S(x,u,v,h')$ (which will be $b$) to obtain the same beginning as the word $W_S(x,u,v,h)$. 
			\item If $t=c$, then there is exactly one line of $L_A \cup L_B \cup L_C$ passing through $x$ and it is a line of $L_C$. Since the slope of the lines in $L_C$ is infinite, we deduce that we must change the sign of the first letter of $W_S(x,u,v,h')$ (which will be $c^{-1}$) to obtain the same beginning as the word $W_S(x,u,v,h)$. 
			\item If $t=abc$, then there is exactly three lines of $L_A \cup L_B \cup L_C$ passing through $x$ : one in $L_A$, one in $L_B$ and one in $L_C$. Then we deduce that after changing the sign of the first letter of $W_S(x,u,v,h')$ (which will be $c$) and changing the second letter of $W_S(x,u,v,h')$ (from $a$ to $b^{-1}$), we recover the beginning of the word $W_S(x,u,v,h)$. 
		\end{itemize}
		Now suppose that $h>v_2$ and $h'<v_2$. By making the same distinction as before but this time on the type of $x+u+v$ (which is $\inv{t}$), we deduce as above that after possibly adding or removing a letter at the beginning of $W_S(t,u,v,h)$, or after possibly changing the two last letters of $W_S(t,u,v,h)$, the end of the word $W_S(t,u,v,h)$ is the same as the end of the word $W_S(t,u,v,h')$.    
	\end{proof}
	
	\begin{Remark}
	In fact, our poof gives a more precise version of Lemma \ref{W_S} by actually computing the first and last letters of $W_S(u,v,h)$ depending on the type of $t$ and $\inv{t}$ (and of course on the height $h$), but we don't need it here. 
	\end{Remark}

	Now we would like to write (a cyclic permutation of) $\gamma$ (or its inverse) as a concatenation of words of the form $W_S(x,u,v,h),W_R(x,u,v,h)$. However, there is a small technical issue to fix here. Since the rectangles $S(x,u,v)$ and $S(x+u+v,u,v)$ intersect along one of their vertical sides (the right vertical side of $S(x,u,v)$ and the left vertical side of $S(x+u+v,u,v)$, the word read by following a horizontal line in $S(x,u,v) \cup S(x+u+v,u,v)$ could not be the concatenation of $W_S(x,u,v,h)$ then $W_S(x+u+v,u,v,h')$ (for the corresponding heights $h$ and $h'$). In order to resolve this, let us consider $I^*_S(x,u,v,h)$ (respectively $I^*_R(x,u,v,h)$) the segment at height $h$ in $S_*(x,u,v)$ (respectively in $R_*(x,u,v)$) and $W^*_S(x,u,v,h)$ (resp. $W^*_R(x,u,v,h)$) the word read by following $I^*_S(x,u,v,h)$ (resp. $I^*_R(x,u,v,h)$). Recall that $S_*(x,u,v)$ and $R_*(x,u,v)$ have been defined respectively in \eqref{S*} and \eqref{R*}. Of course we also have that $W^*_S(x,u,v,h)$ (resp. $W^*_R(x,u,v,h)$) does not depend on the type of the point $x$ and thus we write $W^*_S(t,u,v,h)$ (resp. $W^*_R(t,u,v,h)$) for $t$ a type without ambiguity. Similarly, we consider $^*S(x,u,v)$ (resp.$^*R(x,u,v)$) the rectangle obtained by deleting the left vertical side of $S(x,u,v)$ (resp. $R(x,u,v)$) and we write $^*I_S(x,u,v,h)$ (resp.$^*I_R(x,u,v,h)$) for the segment at height $h$ in $^*S(x,u,v)$ (resp.$^*R(x,u,v)$) and $^*W_S(x,u,v,h)$ (resp. $^*W_R(x,u,v,h)$) the word read by following $^*I_S(x,u,v,h)$ (resp. $^*I_R(x,u,v,h)$).\\
	
	Now let us specify the lengths of subwords read in rectangles :
	\begin{Lemma} \label{length-W_S-W_R} Let $t$ be a type and $0<h<1$ a height. We have the following inequalities :
		\begin{enumerate} 
			\item \label{WS-l'i}
			$l'_i(\gamma)-1 \leq \len(W^*_S(t,u,v,h))\leq l'_i(\gamma)+1$ \quad and \quad
			$l'_i(\gamma)-1 \leq \len(^*W_S(t,u,v,h))\leq l'_i(\gamma)+1$ 
			
			\item $\len(W^*_R(t,u,v,h)) \leq l'_i(\gamma)+1$ \quad and \quad
			$\len(^*W_R(t,u,v,h)) \leq l'_i(\gamma)+1$
		\end{enumerate}
		Denote $
			(u',v') = \left\{ \begin{array}{cc}
				(u,u+v)  & \text{ if } \slope(u+v)\geq 0  \\
				(u+v,v)  &  \text{ if } \slope(u+v) \leq 0
			\end{array} \right. 
	$.
		\begin{enumerate}[resume]
			\item \label{len-W'_S} $l'_i(\gamma)+l_i(\gamma)-1 \leq \len(W^*_S(t,u',v',h))\leq l'_i(\gamma)+l_i(\gamma)+1$  and \\			$l'_i(\gamma)+l_i(\gamma)-1 \leq \len(^*W_S(t,u',v',h))\leq l'_i(\gamma)+l_i(\gamma)+1$
			\item \label{len-W'_R} $\len(W^*_R(t,u',v',h)) \leq  l'_i(\gamma)+l_i(\gamma)+1$ and \\ 
			$\len(^*W_R(t,u',v',h)) \leq  l'_i(\gamma)+l_i(\gamma)+1$
		\end{enumerate}
	\end{Lemma}
	
	\begin{proof}
		\begin{enumerate}
			\item Let $\max(-u_2,v_2) <h_1 <1$, $\min(-u_2,v_2) <h_2<\max(-u_2,v_2)$ and $0<h_3<\min(-u_2,v_2)$ be three heights. Note that the words $W_S(t,u,v,h_j), W^*_S(t,u,v,h_j)$ and $^*W_S(t,u,v,h_j)$ depend on $j$ but not on the specific height $h_j$. Then in order to simplify the notation, let us denote (see figure \ref{fig:length1} on page \pageref{fig:length1}):

				\begin{align*}
				W_j(t)&=W_S(t,u,v,h_j) & ^{*}W_j(t)&={}^{*}W_S(t,u,v,h_j) & 
				W^*_j(t)&=W^*_S(t,u,v,h_j) \\ 
				 W_j(\inv{t})&=W_S(\inv{t},u,v,h_j)
				 & ^{*}W_j(\inv{t})&={}^{*}W_S(\inv{t},u,v,h_j) & W^*_j(\inv{t})&=W^*_S(\inv{t},u,v,h_j)
			\end{align*}

			\begin{figure}[h]
				\centering
				\labellist
				\small \hair 2pt
				\pinlabel{$W_1(t)$} at 110 180
				\pinlabel{$W_2(t)$} at 110 110
				\pinlabel{$W_3(t)$} at 110 40
				\pinlabel{$W_1(\inv{t})$} at 320 270
				\pinlabel{$W_2(\inv{t})$} at 320 190
				\pinlabel{$W_3(\inv{t})$} at 320 120
				\endlabellist
				\includegraphics[width=6cm]{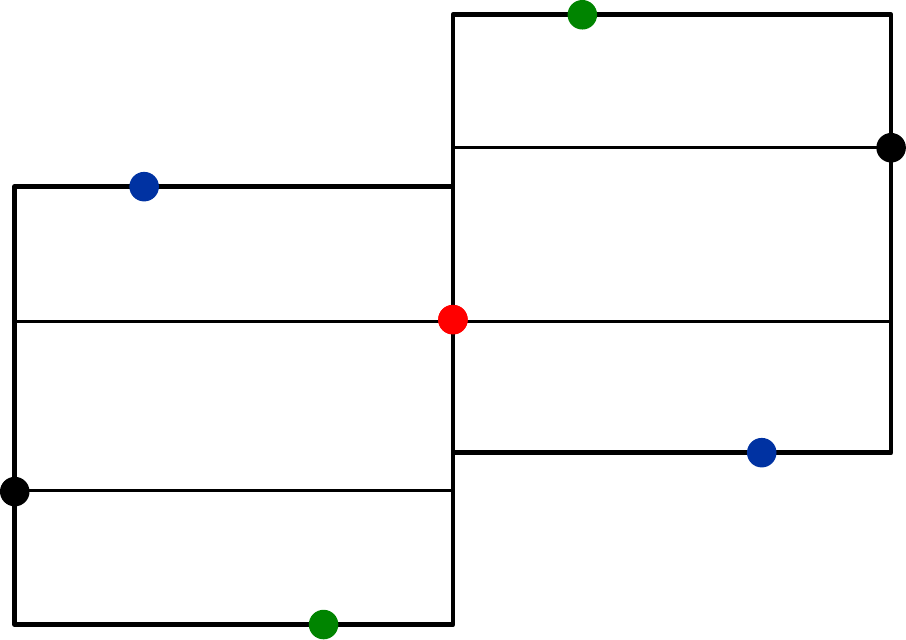}
				\caption{The six words $W_1(t),W_2(t),W_3(t),W_1(\inv{t}),W_2(\inv{t}),W_3(\inv{t})$. }
				\label{fig:length1}
			\end{figure}
			We have, using the property on inverses of Lemma \ref{inverses} : 
			\begin{equation} \label{inv-etoile}
				\begin{aligned} 
					W_1(t)^{-1}&=W_3(\inv{t}) \qquad & W_2(t)^{-1}&=W_2(\inv{t}) \qquad & W_3(t)^{-1}&=W_1(\inv{t}) \\
					W^*_1(t)^{-1}&={}^*W_3(\inv{t}) \qquad & W^*_2(t)^{-1}&={}^*W_2(\inv{t}) \qquad & W^*_3(t)^{-1}&={}^*W_1(\inv{t}) \\
					^*W_1(t)^{-1}&=W^*_3(\inv{t}) \qquad & ^*W_2(t)^{-1}&=W^*_2(\inv{t}) \qquad & ^*W_3(t)^{-1}&=W^*_1(\inv{t})
				\end{aligned}
			\end{equation}
			
			Let us now state and justify the following facts about the lengths of these words :
			\begin{enumerate}
				\item \label{*W-W*} $ \vert ^*W_j(t) \vert - \vert W^*_j(t) \vert \in \{-1,0,1\} $ for all $j \in \{1,2,3\}$. 
				\item \label{W1-W2} $ \vert W_1(t) -W_2(t) \vert = \vert W^*_1(t)-W^*_2(t) \vert = \vert ^*W_1(t) -^*W_2(t) \vert \in \{-1,0,1\},$ \\ 
				$ \vert W_3(t) -W_2(t) \vert = \vert W^*_3(t)-W^*_2(t) \vert = \vert ^*W_3(t)-^*W_2(t) \vert \in \{-1,0,1\}.$
				\item \label{W1*-*W3} $ \vert W^*_1(t)\vert +\vert ^*W_3(t) \vert =2l'_i(\gamma)$, \\
				$ \vert ^*W_1(t) \vert +\vert W^*_3(t)\vert =2l'_i(\gamma)$.
			\end{enumerate}
			
			\begin{proof}
				\begin{enumerate}
					\item This is a trivial consequence of the fact that $ ^*W_j(t) $ and $W^*_j(t)$ are both obtained from $W_j(t)$ by possibly deleting a letter (at the beginning or at the end).
					\item  Elaborating on the ideas of the proof of Lemma \ref{W_S}, we deduce that the words $W_1(t)$ and $W_2(t)$ can only differ by a few letters (at most 2) at their beginning (in the case where $\slope(u+v)<0$) or at their end (in the case where $\slope(u+v)>0$). This difference comes from the line(s) of $L_A \cup L_B \cup L_C$ which passes through the point $x$ (in the case $\slope(u+v)<0$) or $x+u+v$ (in the case $\slope(u+v)>0$) of the lattice $\Lambda$. In other words, the change depends on the type of $x$ (in the case $\slope(u+v)<0$) or in the type of $x+ u+v$ (in the case $\slope(u+v)>0$). Suppose now for simplicity that we are in the case $\slope(u+v)<0$ ; the other case is similar. If the type of $x$ is $abc$ or $c$, then  $\vert W_1(t) \vert =\vert W_2(t) \vert$. If the type of $x$ is $a$ or $b$, then the lengths of $W_1(t)$ and $W_2(t)$ differ by exactly 1. Finally note that the difference of these lengths does not depend on whether or not we consider the possible first or last letter of the word, that is whether or not we put a star $^*$ at the beginning or at the end of the word $W_j(t)$. 
					
					The proof of the second equality, when replacing $W_1(t)$ by $W_3(t)$, is the same.
					
					\item 	The key point is that the words $W^*_1(t)W^*_1(\inv{t})$ and $W^*_3(t)W^*_3(\inv{t})$ correspond to a simple closed curve of slope $\displaystyle \frac{p_i}{q_i}\oplus \frac{p_{i-1}}{q_{i-1}}$. Indeed, consider as in the figure \ref{fig:length} on page \pageref{fig:length} the segment $I$ from $y$ to $y+2(u+v)$, where $y$ is a point on the left vertical side of $S(x,u,v)$, with $x$ of type $t$, and such that $y_2 >x_2$.

		\begin{figure}[h]
						\centering
						\labellist
						\small\hair 2pt
						\pinlabel{$y$} at -10 100
						\pinlabel{$y+2(u+v)$} at 510 260
						\pinlabel{$ $} at -10 30
						\pinlabel{$\color{purple} W^*_1(t)$} at 670 195
						\pinlabel{$\color{purple} W^*_1(\inv{t})$} at 880 275
						\pinlabel{$\color{purple} I$} at 260 210
						\endlabellist
						\includegraphics[width=5cm]{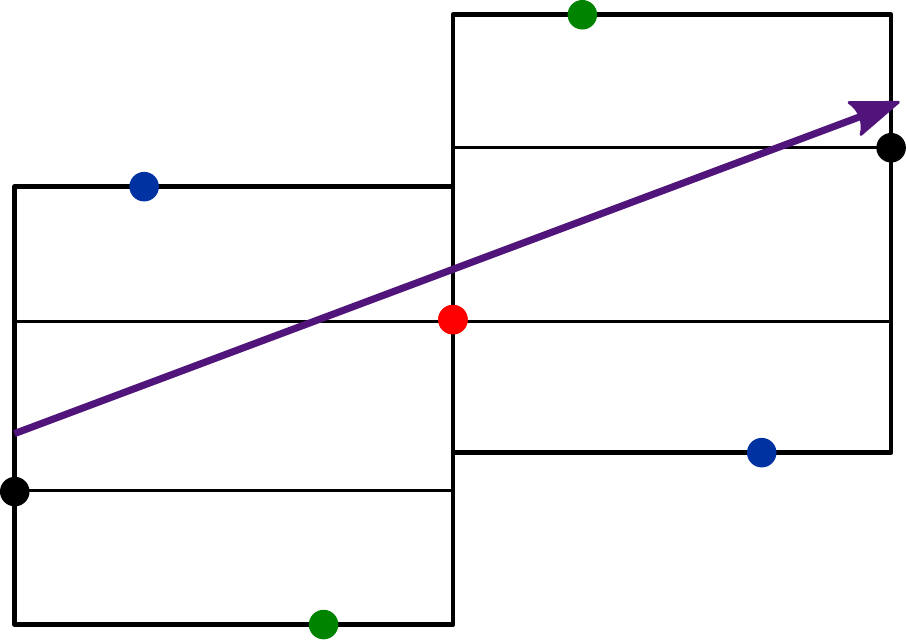} \hspace{1.5cm}
						\includegraphics[width=5cm]{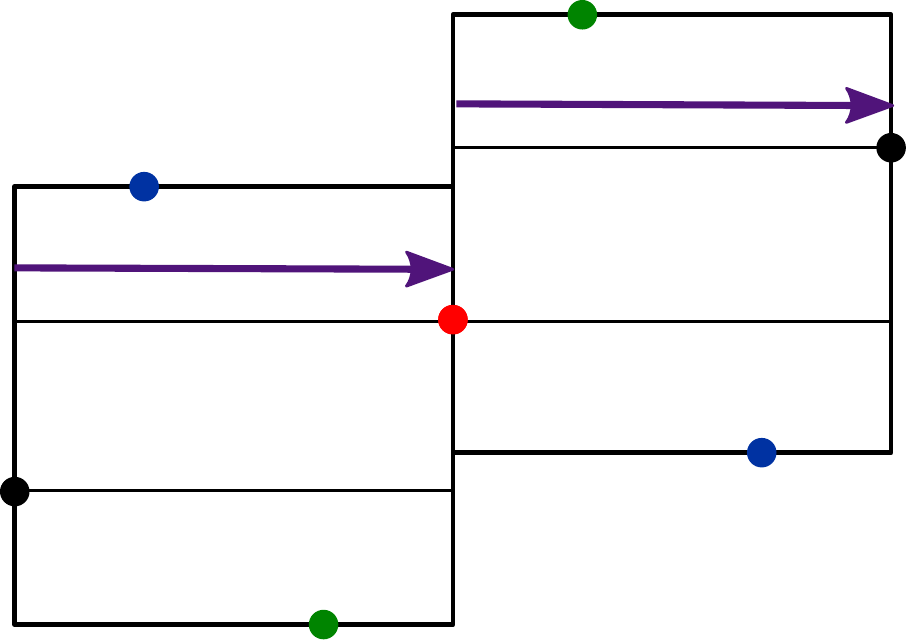}
						\caption{Reading the word of slope $\displaystyle \frac{p_i}{q_i}\oplus \frac{p_{i-1}}{q_{i-1}}$}
						\label{fig:length}
		\end{figure}
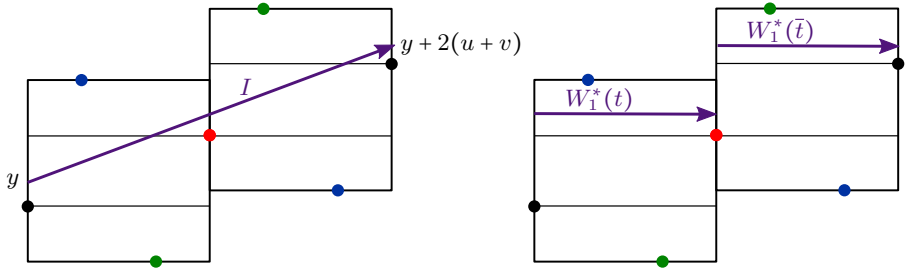

Then in the quotient by $(2\Lambda, \pm)$, this segment is a simple closed curve of slope $\slope(f^{-1}_{\gamma,u,v}(u))\oplus\slope(f^{-1}_{\gamma,u,v}(v))$. Moreover, by following the segment $I$, we read the word $W^*_1(t)W^*_1(\inv{t})$. Do the same but choose $y$ so that $y_2 <x_2$ in order to read the word $W^*_3(t)W^*_3(\inv{t})$. 

Since $W^*_1(t)W^*_1(\inv{t})$ and $W^*_3(t)W^*_3(\inv{t})$ are of slope $\displaystyle \frac{p_i}{q_i}\oplus \frac{p_{i-1}}{q_{i-1}}$, then they are of length $2l'_i(\gamma)$. From this we obtain the following equalities : 
		\begin{align}\label{W+Winv}
						\vert W^*_1(t) \vert + \vert W^*_1(\inv{t})\vert =2l'_i(\gamma)  \\ 
						\vert W^*_3(t) \vert + \vert W^*_3(\inv{t})\vert =2l'_i(\gamma)
		\end{align}
	and we conclude using the equalities \eqref{inv-etoile}.
  \end{enumerate}

			\end{proof}

			Now let us prove point \ref{WS-l'i} of Lemma \ref{length-W_S-W_R}. From \eqref{*W-W*} and \eqref{W1-W2}, we deduce $\vert  \,\, \vert ^*W_3(t) \vert - \vert W^*_1(t) \vert \, \,  \vert \leq 3 $ and $\vert  \,\, \vert W^*_3(t) \vert - \vert ^*W_1(t) \vert \, \,  \vert \leq 3$. Now using also the equalities \eqref{W1*-*W3}, we obtain the desired inequalities for $j=1,3$ :
			\begin{equation}
				l'_i(\gamma)-1 \leq \vert ^*W_1(t) \vert, \vert W^*_1(t) \vert, \vert ^*W_3(t) \vert, \vert W^*_3(t) \vert \leq l'_i(\gamma)+1. 
			\end{equation}
			Using these equalities together with \eqref{W1-W2}, we also have :
			\begin{equation}
					l'_i(\gamma)-2 \leq \vert ^*W_2(t) \vert, \vert W^*_2(t) \vert \leq l'_i(\gamma)+2. 
			\end{equation}
			Now suppose that $\vert W^*_2(t)\vert = l'_i(\gamma)-2$, then by \eqref{W1-W2} we have $\vert W^*_1\vert = l'_i(\gamma)-1$, so by \eqref{W1*-*W3}, $\vert ^*W_3(t) \vert = l'_i(\gamma)+1$. But then we would have $\vert W^*_2(t)\vert -\vert W^*_3(t) \vert = \vert W^*_2(t)\vert - \vert ^*W_3(t) \vert + \vert ^*W_3(t) \vert  -\vert W^*_3(t) \vert = -3 + \vert ^*W_3(t) \vert  -\vert W^*_3(t) \vert \leq -2 $ (using \eqref{*W-W*}), which contradicts \eqref{W1-W2}. \\
			The other cases ($\vert W^*_2(t)\vert = l'_i(\gamma)+2$, $\vert ^*W_2(t)\vert = l'_i(\gamma)-2$, $\vert ^*W_2(t)\vert = l'_i(\gamma)+2$) are similar.

			\item We use the fact that $R(x,u,v) \subset \left\{ \begin{array}{cc}
				S(x-v,u,v)  & \text{ if } \slope(u+v) \geq 0 \\
				S(x-u,u,v)  & \text{ if } \slope(u+v) \leq 0
			\end{array} \right.$. \\
			Then $W^*_R(x,u,v,h)$ is a subword of $W^*_S(x-v,u,v,h')$ or $W^*_S(x-u,u,v,h')$ for some height~$h'$. But by the previous point $\len(W^*_S(x-v,u,v,h'))\leq l'_i(\gamma)+1$ and $\len(W^*_S(x-u,u,v,h'))\leq l'_i(\gamma)+1$, (for all $h'$). Hence the second inequality.

			\item The proof is the same as above, working this time in the rectangle $S(x,u',v')=S'(x,u,v)$ and using the fact that the slope of $\displaystyle f^{-1}_{\gamma,u,v}(u') + f^{-1}_{\gamma,u,v}(v')$ is $ \displaystyle \frac{p_i}{q_i}\oplus 2 \frac{p_{i-1}}{q_{i-1}}$ (and $l'_i(\gamma)+l_i(\gamma)=p_i+q_i+2p_{i-1}+2q_{i-1}$).
			\item Same as above, using the previous point.
		\end{enumerate}
	\end{proof}
	
	At last, we end this section by giving a decomposition of $\gamma$, or more generally of any subword of $\gamma$, as a concatenation of words read in $S(x,u,v)$ and $R(x,u,v)$.
	
	\begin{Lemma} \label{concat-W}
		Let $t \in \Lambda/2\Lambda$ be a type and $\inv{t}$ its $(u,v)$-opposite type. Let $W$ be a subword of (a cyclic-permutation of) $\gamma$ or its inverse. \\
		There exists an integer $p \in \N$ such that for all $0 \leq j \leq p+1$, there exists a real $0<h_j<1$ and a type~$t_j \in \{t,\inv{t} \} $ such that we have the following decomposition of $W$ :
		\begin{equation*}
			W=\s{W_0} W_1  \hdots W_p \p{W_{p+1}}
		\end{equation*}
		with, for all $0 \leq j \leq p+1$, either $W_j=W^*_S(t_j,u,v,h_j)$ or $W_j=W^*_R(t_j,u,v,h_j)$, and $\s{W_0}$ is a suffix of $W_0$, $\p{W_{p+1}}$ is a prefix of $W_{p+1}$. Moreover :
		\begin{enumerate}
			\item \label{cond1} If $W_j=W^*_R(t_j,u,v,h_j)$, then $W_{j+1}=W^*_S(t_{j+1},u,v,h_{j+1})$.
			\item \label{cond2} The types $t_j$ are alternating. Precisely : if 
			$t_0=t$, then $t_j = \left\{ \begin{array}{cc}
				t  & \text{ if $j$ is even }  \\
				\inv{t}  & \text{ if $j$ is odd }
			\end{array} \right.$, \\
			and if $t_0=\inv{t}$, then $t_j = \left\{ \begin{array}{cc}
				\inv{t}  & \text{ if $j$ is even }  \\
				t  & \text{ if $j$ is odd }
			\end{array} \right.$. 
		\end{enumerate}
	\end{Lemma}

	\begin{proof}
		There exists a horizontal segment $I$ in $\R^2$ such that the word we read following $I$ is $W$. Now the proof is direct by applying Lemma \ref{intervals} to $I$. Since the segments $I_j$ of this lemma can only intersect in a point, we deduce that $W$ is the concatenation of the words $W_j$, where $W_j$ is the word read by following $I_j$, that is $W^*_S(t_j,u,v,h_j)$ or $W^*_R(t_j,u,v,h_j)$.  
		\begin{itemize}
			\item The first point of Lemma \ref{concat-W} comes from the third point of Lemma \ref{intervals}.
			\item The second point of Lemma \ref{concat-W} comes from the second point of Corollary \ref{coro-interval}.
		\end{itemize}
	\end{proof} 
	
	\subsubsection{Proof of Theorem \ref{magic-len-S_{0,4}} (special-lengths)}
	\label{proof-magic-len}
	\begin{proof}[Proof of Theorem \ref{magic-len-S_{0,4}}] ~ \\
		
		First note that without loss of generality, we can assume that $W$ is a subword of a cyclic-permutation of $\gamma$ : indeed, if it is a subword of a cyclic-permutation of $\gamma^{-1}$, just take the inverse of the concatenation obtained for $W^{-1}$ (which is a subword of a cyclic-permutation of $\gamma$). We can also assume that $w$ is a subword of a cyclic-permutation of $\gamma$ : indeed, if $w$ is a subword of a cyclic-permutation of $\gamma^{-1}$, $w^{-1}$ is a subword of a cyclic-permutation of $\gamma$ of the same length, and thus the same decomposition for $W$ holds. \\
		
		Also remark that if $i=0$, then $l_i(\gamma)=1<10$ and if $i=1$ and $n_1 = 0$, then $l_i(\gamma)=1<10$. Moreover, if $i=r(\gamma)$, then we have $|W|\geq 3(l'_{r(\gamma)}(\gamma)+l_{r(\gamma)}(\gamma)+1)>2l_{r(\gamma)}(\gamma)=|\gamma|$ and this is impossible since $W$ is supposed to be a subword of (a cyclic-permutation of) $\gamma$ (or its inverse). \\
		
		Therefore let us fix $1 \leq i < r(\gamma)$, such that if $i=1$, then $n_1 \neq 0$, and assume that $l_i(\gamma)\geq 10$. \\ 
		
		Fix $W$ a subword of (a cyclic permutation of) $\gamma$ (or its inverse) of length $|W| \geq 3(l'_i(\gamma)+l_i(\gamma)+1)$ and $w$ a subword of (a cyclic permutation of) $\gamma$ (or its inverse) of length $|w|=l_i(\gamma)-5$. \\ 
		Now that the integer $i$ is fixed, we can consider the map $f_{\gamma,u,v}$ as well as the lattice $\Lambda$ and its basis $(u,v)$ which have been defined in section \ref{mapping} (note that as in this section, we still denote by $(u,v)$ the image of the basis of $\Z^2$ after applying the map $f_{\gamma,u,v}$). Recall that with our definition, we automatically have that $\slope(u+v)$ and $\slope(v-u)$ are of opposite signs. 
		
		Consider $w'$ the subword of $\gamma$ obtained from $w$ by adding two letters on the left and two letters on the right. Thus, we can write $w'=pws$, such that $pws$ is  reduced, $|p|=2, |s|=2$, and $w'$ is a subword of (a cyclic-permutation of) $\gamma$ or its inverse of length $|w'|=l_i(\gamma)-1$. By Lemma \ref{comp-word-interval}, there exists a horizontal segment $I_{w'}$ such that we read $w'$ by following $I_{w'}$, and such that, denoting $l(I_{w'})$ the length of $I_{w'}$, we have : $l(I_{w'})l'_i(\gamma) \leq |w'|+1$. So $l(I_{w'}) \leq \frac{l_i(\gamma)}{l'_i(\gamma)}$. Then we can consider a horizontal segment $I$ of length $\frac{l_i(\gamma)}{l'_i(\gamma)}$ containing $I_{w'} \subset I$. But recall, by the computation \eqref{max-u-v}, that $\max(u_1,v_1)=\frac{l_i(\gamma)}{l'_i(\gamma)}$. Therefore, we can apply Lemma \ref{segment-in-square} and find a point $x \in \Lambda$ such that :
		\begin{enumerate}
			\item Either $I \subset S_*(x,u,v)$ 
			\item Or $I \subset T_*(x,u,v)$. 
		\end{enumerate}
		
		(Recall that $S_*(x,u,v)$ and $T_*(x,u,v)$ have been defined respectively in \eqref{S*} and \eqref{T*}). \\ 
		
		We treat the two cases separately. 
		\begin{enumerate}
			\item Suppose that $I \subset S_*(x,u,v)$. \\ 
			We use Lemma \ref{concat-W} to ensure that we can write $W$ in the following way :
			\begin{equation}\label{decomposition1}
				W=\s{W_0}W_1 \hdots W_p \p{W_{p+1}}
			\end{equation}
			with, for all $0 \leq j \leq p+1$, either $W_j=W^*_S(t_j,u,v,h_j)$ or $W_j=W^*_R(t_j,u,v,h_j)$, and $t_j$ and $h_j$ as defined in the Lemma. Recall that $\s{W_0}$ and $\p{W_{p+1}}$ denote respectively a suffix of $W_0$ and a prefix of $W_{p+1}$. Set $\cJ=\{j \in \{1, \hdots,p \} \mid W_j=W^*_S(t,u,v,h_j) \}, \inv{\cJ}=\{j \in \{1, \hdots,p \} \mid W_j=W^*_S(\inv{t},u,v,h_j) \}$ and $\mathcal{K}=\{j \in \{1, \hdots,p \} \mid W_j=W^*_R(t_j,u,v,h_j) \}$. Of course we have~$\{1,\hdots,p\} = \cJ \sqcup \inv{\cJ} \sqcup\mathcal{K}$. \\
			
			\begin{enumerate}
				\item \label{w-subword} Let us show that $w$ is a subword of $W_j$, for all $j \in \cJ$, and that $w^{-1}$ is a subword of $W_j$, for all $j \in \inv{\cJ}$ : \\ 
				
				Recall that by construction of the horizontal segment $I$, we have $I_{w'} \subset I$, so by our hypothesis $I \subset S_*(x,u,v)$, we also have $I_{w'} \subset S_*(x,u,v)$. Then we deduce that $w'$ is a subword of $W^*_S(t,u,v,h)$, for a well chosen $0<h<1$ and $t$ the type of~$x$. In addition, Lemma \ref{W_S} ensures that for all $j \in \cJ$, $W_j$ and $W_S(t,u,v,h)$ are equal up to adding or deleting a letter at the beginning or at the end of $W_j$, and up to changing its first two or last two letters. Then the subword $w'$, up to the previous changes, is always a subword of $W_j$, for $j \in \cJ$. But these changes affect at most the first two letters of $w'$ and the last two letters of $w'$ (in fact, never all four at the same time but this detail doesn't matter here), and since $w'=pws$, with $|p|=|s|=2$ (this decomposition is cyclically reduced), we deduce that these changes can only affect $p$ and $s$. Therefore, we can conclude that $w$ is a subword of $W_j$, for all $j \in \cJ$. \\
				A similar argument applies for $w^{-1}$. Since $w'$ is a subword of $W_S^*(t,u,v,h)$, we deduce by Lemma \ref{inverses} that $w'^{-1}$ is a subword of $W_S(\inv{t},u,v,1-h)$. Thus, as before, using Lemma \ref{W_S}, we deduce that for all $j \in \inv{\cJ}$, $W_j$ and $W_S(\inv{t},u,v,1-h)$ are equal up to some changes of letters affecting at most the first two letters and the last two letters. Since $w'^{-1}=s^{-1}w^{-1}p^{-1}$, with $|s^{-1}|=|p^{-1}|=2$, we deduce that $w^{-1}$ is itself a subword of $W_j$, for all $j \in \inv{\cJ}$. \\
				
				\item \label{dec-point1} Decomposition of $W$ and proof of point \ref{change-letter} of Theorem \ref{magic-len-S_{0,4}} : \\ 
				
				Let $q=\# \cJ + \# \inv{\cJ}$. By letting $u_2,u_4,\hdots,u_{2q}$ by the  occurrences of $w$ and $w^{-1}$ in $W_j$, for $j \in \cJ \cup \inv{\cJ}$, guaranteed by the previous point, we deduce that $u_2,u_4, \hdots,u_{2q}$ are pairwise disjoint. Also let $u_1,u_3,\hdots,u_{2q+1}$ be the remainders in between in the word $W$, then we can write : 
				\begin{equation}
					W=u_1u_2\hdots u_{2q+1}
				\end{equation}
				Moreover, by construction, for all $k \in 2\{1, \hdots,q\}$, $u_k$ is either equal to $w$ or equal to $w^{-1}$. Thus point \ref{change-letter} of Theorem \ref{magic-len-S_{0,4}} is proved, setting $\cI = 2\{1,\hdots,q\}$. \\ 
				
				\item Let us show that $\max(\# \cJ,\# \inv{\cJ}) \geq 1$ : \\
				
				By contradiction, if $\# \cJ= \# \inv{\cJ}=0$, then, since a subword of the form $W^*_R(t_i,u,v,h_i)$ is necessarily followed by a word of the form $W^*_S(t_{i+1},u,v,h_{i+1})$, we must have $\# \cK \leq 1$ and then $p \leq 1$. Therefore $|W| \leq |\s{W_0}| + |W_1| + |\p{W_{p+1}}|$.
				But we also know, by Lemma \ref{length-W_S-W_R}, that for all $0 \leq j \leq p+1$, $|W_j|\leq l'_i(\gamma)+1$. Then we deduce : 
				\begin{equation}\label{|W|-first-case} 
					|W| \leq 3(l'_i(\gamma)+1) <3(l'_i(\gamma)+l_i(\gamma)+1)
				\end{equation}
				and this contradicts our hypothesis on the length of $W$. In particular, we proved :
				\begin{equation} \label{card-J}
					\# \cJ + \# \inv{\cJ} \geq 1
				\end{equation}
				
				\item Proof of point \ref{proportion} of Theorem \ref{magic-len-S_{0,4}} : \\ 
				
				On one hand we have : 
				\begin{equation} \label{sum-uk}
					\underset{k \in \cI} \sum |u_k| = \underset{k \in \cI} \sum |w| = \# \cI |w|=q |w|= (\# \cJ + \# \inv{\cJ})(l_i(\gamma)-5) \geq (\# \cJ + \# \inv{\cJ})\frac{1}{2}l_i(\gamma)
				\end{equation}
				and the last inequality holds because $l_i(\gamma) \geq 10$. \\
				
				On the other hand, using \eqref{decomposition1}, we have :
			$|W|=|\p{W_0}|+|\s{W_{p+1}}|+\underset{j \in \cJ} \sum |W_j| +\underset{j \in \inv{\cJ}}\sum |W_j|+\underset{j \in \mathcal{K}} \sum |W_j|$.
				In addition, we know, by Lemma \ref{length-W_S-W_R}, that for all $0 \leq j \leq p+1$ : $|W_j|\leq l'_i(\gamma)+1$. Then we deduce~:
				\begin{align*}
					|W| & \leq 2(l'_i(\gamma)+1) + (\# \cJ+ \# \inv{\cJ} + \# \mathcal{K})(l'_i(\gamma)+1) \\
					& \leq 4l_i(\gamma)+2(\# \cJ + \# \inv{\cJ} + \# \cK)l_i(\gamma) \quad \text{ using } l'_i(\gamma)+1 \leq 2l_i(\gamma)
				\end{align*}
				Moreover, recall that a subword of the form $W^*_R(t_i,u,v,h_i)$ is necessarily followed by a subword of the form $W^*_S(t_{i+1},u,v,h_{i+1})$, then if $j \in \cK \cap \{1,\hdots,p-1 \}$, then $j+1 \notin \cK$ and so we deduce that $\# \cK \leq \# \cJ+\# \inv{\cJ} +1$. Therefore :
				\begin{equation} \label{prop-first-case}
					\begin{aligned}
						|W| & \leq 4l_i(\gamma)+2(2(\# \cJ+\# \inv{\cJ})+1)l_i(\gamma)=4(\# \cJ+\# \inv{\cJ}) l_i(\gamma) + 6l_i(\gamma) \\
						& \leq 10(\# \cJ+\# \inv{\cJ}) l_i(\gamma) \qquad \text{ by \eqref{card-J}}\\
						& \leq 20 \underset{k \in \cI} \sum |u_k| \qquad \text{ by \eqref{sum-uk}}
					\end{aligned} 
				\end{equation}
				This last inequality finishes the proof of point \ref{proportion} of Theorem \ref{magic-len-S_{0,4}} in the case where $I \subset S_*(x,u,v)$.  
			\end{enumerate}
			
			\item Suppose now that $I \subset T_*(x,u,v)$. In particular, since $T_*(x,u,v) \subset S'_*(x,u,v)$ (see \eqref{T*-in-S'*}), we have $I \subset S'_*(x,u,v)$. Let us denote $
				(u',v') = \left\{ \begin{array}{cc}
					(u,u+v)  & \text{ if } \slope(u+v)\geq 0  \\
					(u+v,v)  &  \text{ if } \slope(u+v) \leq 0
				\end{array} \right. 
			$. \\ 
			We have $S'_*(x,u,v)=S_*(x,u',v')$ and $S'^\pm(x,u,v)=S^\pm(x,u',v')$. \\
			Recall that inequality \eqref{slope(u+v)} ensures that : $\slope(u+v)\neq 0$. Then we have $
				\slope(u') < 0 < \slope(v')$.\\
			
			We now distinguish according to the signs of $\slope(u'+v')$ and $\slope(v'-u')$.
			\begin{itemize}
				\item Suppose that $\slope(u'+v')$ and $\slope(v'-u')$ are of opposite signs. \\ 
				We will then proceed in exactly the same way as in the first case, changing $(u,v)$ to $(u',v')$. The proof is then completely unchanged, except that we no longer have the inequality $|W_j| \leq l'_i(\gamma)+1$, but instead $|W_j| \leq l'_i(\gamma)+l_i(\gamma)+1$ (see point \ref{len-W'_S} and \ref{len-W'_R} in Lemma \ref{length-W_S-W_R}). The inequalities change as follow : 
				\begin{enumerate}
					\setcounter{enumii}{2}
					\item Let us show that $\max(\# \cJ, \# \inv{\cJ}) \geq 1$ : \\  
					Suppose that $\# \cJ = \# \inv{\cJ} = 0$, then, similarly to \eqref{|W|-first-case}, we can bound $|W|$ : 
					\begin{align*}
						|W|\leq 3(l'_i(\gamma)+l_i(\gamma)+1)
					\end{align*}
					and this is a contradiction with our hypothesis on the length of $W$.\\ 
					
					\item Proof of point \ref{proportion} of Theorem \ref{magic-len-S_{0,4}} : \\ 
					Equation \eqref{sum-uk} doesn't change : 
					\begin{equation}\label{sum-uk-2}
						\underset{k \in \cI} \sum |u_k| \geq (\# \cJ + \# \inv{\cJ})\frac{1}{2}l_i(\gamma).
					\end{equation}
					Then, similarly to computation \eqref{prop-first-case}, we can bound 
					$|W|$ :
					\begin{align*}
						|W|& \leq 2(l'_i(\gamma)+l_i(\gamma)+1)+(\# \cJ + \# \inv{\cJ} + \# \cK)(l'_i(\gamma)+l_i(\gamma)+1) \\
						& \leq 6l_i(\gamma)+3(\# \cJ + \# \inv{\cJ} + \# \cK)l_i(\gamma) \qquad \text{ using } l'_i(\gamma)+l_i(\gamma)+1 \leq 3l_i(\gamma) \\ 
						& \leq 6( \# \cJ+\# \inv{\cJ}) l_i(\gamma)+9l_i(\gamma) \qquad \text{since } \# \cK \leq \# \cJ+ \# \inv{\cJ} +1 \\
						& \leq 15( \# \cJ+\# \inv{\cJ}) l_i(\gamma) \qquad \text{ because } \# \cJ+\# \inv{\cJ} \geq 1  \\
						& \leq 30 \underset{k \in \cI} \sum |u_k| \qquad \text{ by \eqref{sum-uk-2}}
					\end{align*}
				\end{enumerate}
				
				\item Suppose now that $\slope(u'+v')$ and $\slope(v'-u')$ are of the same sign. Notice that our choice of $u,v$ and $u',v'$ imposes that $\slope(u',v')\neq 0$ and $\slope(v'-u')\neq \infty$, we deduce in particular that $S'(x,u,v) \setminus T(x,u,v)$ has non-empty interior. \\ We are going to do almost the same procedure as in the previous cases, but we need to be a little more careful since the rectangles $S(y,u',v')$ for different $y$ intersect, hence the occurrences of $w$ might not be disjoint. Denote by $t$ the type of $x$ and $\inv{t}$ its $(u',v')$-opposite type. \\
				Let $I_W$ be an horizontal segment such that we read $W$ by following $I_W$. 
				By Lemma \ref{int-same-sign}, there exist $x_0,\hdots,x_{p+1} \in \Lambda$ and $h_0, \hdots,h_{p+1}$ some heights such that we can write $I_W$ as a union of segments : 
				\begin{equation} \label{int-cas2}
					I_W=I_0\cup I_1 \cup \hdots \cup I_p \cup I_{p+1}
				\end{equation}
				satisfying : 
				\begin{itemize}
					\item $I_0 \subset I_S(x_0,u',v',h_0), I_{p+1} \subset I_S(x_{p+1},u',v',h_{p+1})$
					\item For all $1 \leq j \leq p$, $I_j=I_S(x_j,u',v',h_j)$
					\item For all $0 \leq j \leq p+1$, $x_j$ is of type $t$ or $\inv{t}$ 
					\item $\inf I_j < \inf I_k$ when $j<k$ (this is a consequence of the third point of Lemma \ref{int-same-sign})
				\end{itemize}
				For all $j \in \{0, \hdots, p+1 \}$, let us denote $W_j$ the subword read by following $I_j$. Therefore for all $1 \leq j \leq p$, $W_j=W_S(x_j,u',v',h_j)$ and $W_0$ and $W_{p+1}$ are subwords respectively of $W_S(x_0,u',v',h_0)$ and $W_S(x_{p+1},u',v',h_{p+1})$.
				Notice that the intervals $I_j$ for $0 \leq j \leq p+1$ might have non-empty intersection, hence the subwords $W_j$ are not necessarily disjoint in $W$.
				Set $\cJ=\{j \in \{1, \hdots,p \} \mid W_j=W_S(t,u',v',h_j) \}$, $\inv{\cJ}=\{j \in \{1, \hdots,p \} \mid W_j=W_S(\inv{t},u',v',h_j) \}$. Of course we have~$\{1,\hdots,p\} = \cJ \sqcup \inv{\cJ}$. Also denote $W^*_j$ the subword read by following $I^*_S(x_j,u',v',h_j)$. 
				
				\begin{enumerate}
					
					\item Let us show that $w$ is a subword of $W_j$ for all $j \in \cJ$ and that $w^{-1}$ is a subword of $W_j$ for all $j \in \inv{\cJ}$~: \\
					
					The proof is the same as in the first case \ref{w-subword} changing the basis $(u,v)$ to $(u',v')$. We recall here the main steps : we have $I_{w'} \subset I \subset S'_*(x,u,v)=S_*(x,u',v')$, so $w'$ is a subword of $W_S^*(t,u',v',h)$, for some well chosen $0<h<1$. In addition, Lemma \ref{W_S} ensures that for all $j \in \cJ$, $W_j$ and $W_S(t,u',v',h)$ are equal up to some change of letters which affect at most the first two letters of $w'$ and the last two letters of $w'$. Since $w'=pws$, with $|p|=|s|=2$, we deduce that $w$ is a subword of $W_j$, for all $j \in \cJ$. By noticing, using Lemma \ref{inverses}, that $w'^{-1}$ is a subword of $W_S(\inv{t},u',v',h')$, for some well chosen $h'$, we deduce in the same way that $w^{-1}$ is a subword of $W_j$, for all $j \in \inv{\cJ}$. \\
					
					\item Decomposition of $W$ and proof of point \ref{change-letter} of Theorem \ref{magic-len-S_{0,4}} : \\
					
					Let $q=\# \cJ + \# \inv{\cJ}$ and $u_2,u_4,\hdots,u_{2q}$ be the occurrences of $w$ and $w^{-1}$ in $W_j$, for~$j \in \cJ \cup \inv{\cJ}=\{1,\hdots,p\}$, guaranteed by the previous point. The key point is to justify that the subwords $u_2,u_4,\hdots,u_{2q}$ are disjoint in $W$. For $j \in \{1,\hdots,p\}$, let $I_{u_{2j}} \subset I_j$ be a horizontal segment such that the subword read by following $I_{u_{2j}}$ is $u_{2j}$. Then, since $I_{w'} \subset I \subset T_*(x,u,v)$ by hypothesis, we deduce that $I_{u_{2j}} \subset T_*(x_j,u,v)$. But recall that by Lemma \ref{T-disjoint}, all the rectangles $T_*(x_j,u,v)$ for $j \in \cJ \cup \inv{\cJ}$ are disjoint. Therefore we deduce that the segments $I_{u_{2j}}$ for $j \in \{1,\hdots,p\}$ are disjoint, then so are the subwords $u_{2j}$ in $W$. 
					Now the rest of the proof follows as in \ref{dec-point1} : let $u_1,u_3,\hdots,u_{2q+1}$ be the remainders in between in the word $W$, so we can write : 
					\begin{equation}
						W=u_1u_2\hdots u_{2q+1}
					\end{equation}
					and by construction, for all $k \in 2\{1,\hdots,q\}$, $u_k$ is either equal to $w$ or to $w^{-1}$. Thus point \ref{change-letter} of Theorem \ref{magic-len-S_{0,4}} is proved, setting $\cI=2\{1,\hdots,q\}$. 
					
					\item Let us show that $p \geq 1$ : \\
					
					By contradiction, if $p=0$, then 
					\begin{align*}
						|W| & \leq |W^*_0|+|W^*_1| \leq 2(l'_i(\gamma)+l_i(\gamma)+1) \text{ using Lemma \ref{length-W_S-W_R} } \\
						& < 3(l'_i(\gamma)+l_i(\gamma)+1)
					\end{align*}
					which is a contradiction with our hypothesis on the length of $W$. \\
					
					\item Proof of point \ref{proportion} of Theorem \ref{magic-len-S_{0,4}} : \\
					
					On the one hand we have : 
					\begin{equation}\label{sum-uk-3}
						\sum_{k \in \cI} |u_k|=\sum_{k \in \cI} |w| =p (l_i{\gamma}-5) \geq \frac{p}{2}l_i(\gamma)
					\end{equation}
					and the last inequality holds because $l_i(\gamma) \geq 10$. \\ 
					
					On the other hand, using the equality \eqref{int-cas2}, we have : 
					\begin{align*}
						|W| & \leq |W^*_0|+|W^*_{p+1}|+ \sum_{j=1}^{p} |W^*_j| \leq (p+2)(l'_i(\gamma)+l_i(\gamma)+1) \text{ by point \ref{len-W'_S} in Lemma \ref{length-W_S-W_R}} \\
						& \leq 3(p+2)l_i(\gamma) \qquad \text{ since } l'_i(\gamma)+1 \leq l_i(\gamma) \\
						& \leq 9pl_i(\gamma) \qquad \text{ using } p\geq 1 \\
						& \leq 18 \sum_{k \in \cI} |u_k|  \qquad \text{ by \eqref{sum-uk-3}}.
					\end{align*}
					This last inequality finishes the proof of point \ref{proportion} of Theorem \ref{magic-len-S_{0,4}} in the case where $I \subset T_*(x,u,v)$. 
				\end{enumerate}
			\end{itemize}
		\end{enumerate}
	\end{proof}

	\section{Properties on Bowditch representations and simple-stable representations}
	
	\label{sph-UQG}
	
	Let $(X,d)$ be a $\delta$-hyperbolic, geodesic and visibility space, and $o$ a basepoint in $X$.  Recall that we have already introduce the \emph{displacement length} of an isometry $A$ of $X$ in section \ref{subsec:Bowditch rep}.
	Another quantity that we will use extensively is the \emph{stable length}, defined by $l_S(A)=\underset{n \to \infty}{\lim}\frac{1}{n}d(A^no,o)$. This definition is independent of the choice of the basepoint $o$ in $X$. The stable length satisfies the following usefull properties : $l_S(A^n)=nl_S(A)$ (which is not true in general for the displacement length) and $l_S(A)\leq \frac{1}{n}d(A^no,o)$. Also recall that when $A$ is a hyperbolic isometry of $X$, then the map from $\Z$ to $X$ that sends $n$ to $A^no$ is a quasi-isometry, by definition. It implies that $A$ has two fixed points in the boundary of $X$, one attracting and the other repelling, denoted by $A^+$ and $A^-$ respectively.  \\
		
	\subsection{A uniform quasi-geodesicity setting}
	\label{sph-first-ex-UQG}
	
	Fix $G,D_1$ and $D_2$ three isometries of $X$. We can consider the set $\mathcal{W}(G,D_1,D_2)$ of (finite) words on the alphabet $\mathcal{A}=\{G,G^{-1},D_1,D_1^{-1},D_2,D_2^{-1}\}$.  For $W \in \mathcal{W}(G,D_1,D_2)$, we denote by $|W|$ its word length ($W$ is seen as a word on the alphabet $\mathcal{A})$. We also consider $\mathcal{H}(G,D_1,D_2)$ the set of bi-infinite (reduced) words on the alphabet $\mathcal{A}$, that is, $H=(H_n)_{n\in \Z} \in \mathcal{H}(G,D_1,D_2)$ if and only if for all $n\in \Z$, $H_n \in \mathcal{A}$ and $H_n\neq H_{n+1}^{-1}$. When we have a bi-infinite word $H=(H_n) \in \mathcal{H}(G,D_1,D_2)$, we associate to it a bi-infinite sequence of finite words $W=(W_n)_{n\in \Z}$ in the following way :
	\begin{equation*}
		W_n=\left\{ \begin{array}{ll}
			H_0H_1\hdots H_{n-1} & \text{for } n>0  \\
			I_d  & \text{for } n=0 \\
			H_{-1}^{-1}\hdots H_n^{-1} & \text{for } n<0
		\end{array}\right.
	\end{equation*}
	Hence $W_n \in \mathcal{W}(G,D_1,D_2)$ for all $n\in \Z$. Moreover, the word length of $W_n$ is $|W_n|=|n|$ for all $n\in \Z$ and the following recursive formula holds for all $n \in \Z$ : $W_{n+1}=W_nH_n$. We denote by $\mathcal{G}(G,D_1,D_2)$ the set of bi-infinite sequences of finite words associate to bi-infinite words in $\mathcal{H}(G,D_1,D_2)$. \\ 
	
	In this section, we will study a particular class of bi-infinite words $H=(H_n)_{n\in \Z}$ and their associate bi-infinite sequences of words $W=(W_n)_{n\in \Z}$. Let us fix an integer $N \geq 1$ and define $\mathcal{H}_N(G,D_1,D_2)$ to be the subset of $\mathcal{H}(G,D_1,D_2)$ consisting of the bi-infinite words $H=(H_n)_{n\in \Z}$ which satisfy the following condition : \\ 
	If $n_1<n_2$ are two integers in $\Z$ such that $H_{n_1},H_{n_2} \in \{ D_1,D_1^{- 1},D_2,D_2^{-1}\}$ and for all $n_1<n<n_2$, $H_n \in \{G^{\pm 1} \}$, then :
	\begin{itemize}
		\item If $H_{n_1} \in \{D_1^{\pm 1}\}$, then $H_{n_2} \in  \{D_2^{\pm 1}\}$, for all $n_1<n<n_2$, $H_n=G^{-1}$ and $n_2-n_1-1\geq N$.
		\item If $H_{n_1} \in \{D_2^{\pm 1}\}$, then $H_{n_2} \in  \{D_1^{\pm 1}\}$, for all $n_1<n<n_2$, $H_n=G$ and $n_2-n_1-1\geq N$.
	\end{itemize}
	Thus, the bi-infinite words in $\mathcal{H}_N(G,D_1,D_2)$ are those of the form : 
	\begin{equation*}
		\hdots D_1^{\pm 1}\underbrace{G^{-1}\hdots G^{-1}}_{\geq N}D_2^{\pm 1}\underbrace{G \hdots G}_{\geq N} D_1^{\pm 1}\underbrace{G^{-1}\hdots G^{-1}}_{\geq N}D_2^{\pm 1}\underbrace{G \hdots G}_{\geq N} D_1^{\pm 1} \hdots
	\end{equation*}
	We denote by $\mathcal{G}_N(G,D_1,D_2)$ the set of bi-infinite sequences $W=(W_n)_{n\in \Z} \in \mathcal{G}(G,D_1,D_2)$ associate to bi-infinite words in $\mathcal{H}_N(G,D_1,D_2)$. \\ 
	
	Fix $o$ a basepoint in $X$. Starting from a bi-infinite sequence $W=(W_n)_{n \in \mathbb{Z}} \in \mathcal{G}(\mathcal{A}(G,D_1,D_2))$, we define the sequence of points in $X$~: $x_n=W_no, \forall n \in \Z$. We will prove in this section that the sequences of points defined by the elements of $\mathcal{G}_N(G,D_1,D_2)$ are uniform quasi-geodesics, that is the existence of two reals $\lambda >0$ and $k \geq 0$ such that for all $n,m \in \mathbb{Z}$, we have : $\frac{1}{\lambda}|n-m|-k \leq d(x_n,x_m)\leq \lambda |n-m| + k$. We also say that $(x_n)_{n \in \Z}$ is a \emph{$(\lambda,k,L)$-local-quasi-geodesic} if we have $\frac{1}{\lambda}|n-m|-k \leq d(x_n,x_m)\leq \lambda |n-m| + k$ whenever $|n-m|\leq L$.
	
	\begin{Proposition} \label{Sphere-unif-quasi-geod}
		Let $X$ be a $\delta$-hyperbolic, geodesic and visibility space, and $o \in X$ any basepoint. Pick $G$ a hyperbolic isometry and $D_1,D_2$ two isometries of $X$. Suppose that $D_1(G^-)\neq G^-$ and $D_2(G^+)\neq G^+$. Then, there exists $\lambda > 0, k\geq0$ and $N \in \mathbb{N}^*$, such that for all bi-infinite sequence $ W=(W_n)_{n \in \mathbb{Z}} \in \mathcal{G}_N(G,D_1,D_2)$, the sequence of points $x_n=W_no$ is a $(\lambda,k)$-quasi-geodesic. 
	\end{Proposition}
	
	\begin{proof}
		$\bullet$ \textbf{Step 1 : Quasi-isometry on a half-period} \\
		We first show that there exist two constants $\lambda > 0$ and $k \geq 0$, only depending on $\delta, G, \cF_1,\cF_2$  and $o$, such that the two following inequalities are satisfied : 
		\begin{equation}\label{step1-1}
			\frac{1}{\lambda}|G^nD_1G^{-m}|-k \leq d(G^nD_1G^{-m}o,o) \qquad \text{ for all } n,m \geq 0
		\end{equation}
		\begin{equation}\label{step1-2}
			\frac{1}{\lambda}|G^{-n}D_2G^{m}|-k \leq d(G^{-n}D_2G^{m}o,o) \qquad \text{ for all } n,m \geq 0
		\end{equation}
		
		By assumption, the two points at infinity $D_1(G^-)$ and $G^-$ are distinct, so by visibility of the space $X$ we can consider a geodesic, called $\Lambda$, with endpoints $D_1(G^-)$ and $G^-$. Now consider $p$ a projection map on $\Lambda$, that is $p : X \to \Lambda$ satisfying $\forall x \in X, d(x,p(x))=d(x,\Lambda)=\underset{y \in \Lambda} \inf \, d(x,y)$ (such a map exists but is not necessarily unique). Since the (half) quasi-geodesic $(G^{-n}o)_{n \in \N}$ and the geodesic  $\Lambda$ have $G^-$ as a common endpoint, we have, by stability of quasi-geodesics in $\delta$ hyperbolic spaces, the existence of a constant $K_1 >0$ (only depending on $\delta,G,D_1$ and $o$) such that $\{G^{-n}o\}_{n \in \N}$ and the half geodesic $[p(o),G^{-})$ remain in the $K_1$-neighborhood of each other. We deduce the following :
		\begin{equation} \label{K1-S}
			d(G^{-n}o,p(G^{-n}o)) \leq K_1, \text{ for all } n \in \N
		\end{equation}
		With the same argument, namely that the (half) geodesic $(D_1G^{-m}o)_{m \in \N}$ and $\Lambda$ share the same endpoint $D_1(G^-)$, we deduce the existence of a constant $K_2 >0$ (only depending on $\delta, G, D_1$ and $o$) such that 
		\begin{equation} \label{K2-S}
			d(D_1G^{-m}o,p(D_1G^{-m}o)) \leq K_2,  \text{ for all } m \in \N.
		\end{equation}
		
		Then we can write the following inequalities :
		\begin{align*}
			d(G^nD_1G^{-m}o,o)& =d(D_1G^{-m}o,G^{-n}o) \text{ because $G^{-n}$ is an isometry } \\
			& \geq d(p(D_1G^{-m}o),p(G^{-n}o)) -d(p(D_1G^{-m}o),D_1G^{-m}o)-d(p(G^{-n}o),G^{-n}o) \\
			&  \geq d(p(D_1G^{-m}o),p(G^{-n}o)) -K_1 -K_2 \text{ by inequalities \ref{K1-S} and \ref{K2-S}}. \\
		\end{align*}
		
		But since $G^{-n}o \underset{n\to \infty} \longrightarrow G^-$, we also deduce $p(G^{-n}o) \underset{n\to \infty} \longrightarrow G^-$ and, again, since $D_1G^{-m}o \underset{m\to \infty} \longrightarrow D_1(G^-)$, we have $p(D_1G^{-m}o) \underset{m\to \infty} \longrightarrow D_1(G^-)$. Then, for $n$ and $m$ sufficiently large, $p(G^{-n}o)$ belongs to $[p(o),G^-) \cap [p(D_1o),G^-)$ and $p(D_1G^{-m}o)$ belongs to $[p(o),D_1(G^-)) \cap [p(D_1o),D_1(G^-))$. This implies that for $n$ and $m$ sufficiently large, the four points $p(G^{-n}o), p(D_1G^{-m}o),p(D_1o)$ and $p(o)$ are aligned in one of the two following orders on the geodesic $\Lambda$ : $p(G^{-n}o),p(o),p(D_1o), p(D_1G^{-m}o)$ or $p(G^{-n}o),p(D_1o),p(o), p(D_1G^{-m}o)$. In both cases, for $n$ and $m$ sufficiently large :
		\begin{equation} \label{min-G-n-DG-m}
			d(p(G^{-n}o),p(D_1G^{-m}o)) \geq d(p(G^{-n}o),p(o))-d(p(o),p(D_1o))+d(p(D_1o),p(D_1G^{-m}o)). 
		\end{equation}
		We also have, 
		\begin{align} \label{half1}
			d(p(G^{-n}o),p(o)) & \geq d(G^{-n}o,o) - d(G^{-n}o,p(G^{-n}o))-d(p(o),o) \nonumber \\ 
			& \geq d(G^{-n}o,o)-2K_1 \text{ by inequality \eqref{K1-S}}
		\end{align}
		and similarly :
		\begin{align} \label{half2}
			d(p(D_1G^{-m}o),p(D_1o)) & \geq d(D_1G^{-m}o,D_1o) - 2K_2 \text{ by inequality \eqref{K2-S}}.
		\end{align}
		We can now prove inequality \eqref{step1-1} :
		\begin{align*}
			d(G^nD_1G^{-m}o,o)& \geq d(p(D_1G^{-m}o),p(G^{-n}o)) -K_1 -K_2 \\
			& \geq d(p(G^{-n}o),p(o))-d(p(o),p(D_1o))+d(p(D_1o),p(D_1G^{-m}o)) -K_1-K_2 \text{ by \eqref{min-G-n-DG-m}}\\
			& \geq d(G^{-n}o,o)+d(D_1G^{-m}o,D_1o)-d(p(o),p(D_1o))-3K_1-3K_2 \text{ by \eqref{half1} and \eqref{half2} } \\
			& = d(G^no,o)+d(G^mo,o)-d(p(o),p(D_1o))-3K_1-3K_2 \\
			& \geq (n+m)l_S(G)-d(p(o),p(D_1o))-3K_1-3K_2
		\end{align*}
		In the last inequality, we used that $d(G^no,o) \geq nl_S(G)$, as recalled in the beginning of this section. 
		Since $(n+m)l_S(G)=(n+m+1)l_S(G)-l_S(G)=|G^nD_1G^{-m}|l_S(G)-l_S(G)$, we have proved the inequality \eqref{step1-1} for $n$ and $m$ sufficiently large, prescribing $\lambda=\frac{1}{l_S(G)}$ (recall $l_S(G)>0$ when $G$ is hyperbolic), and $k=l_S(G)+d(p(o),p(D_1o))+3K_1+3K_2$. Since there is only a finite number of values of $G^nD_1G^{-m}$, for $n$ and $m$ smaller than a fixed constant, the inequality \eqref{step1-1} is still true for all $n,m \in \N$, after possibly changing the values of $\lambda$ and $k$.  \\
		
		In order to prove the inequality \eqref{step1-2}, we change $D_1$ to $D_2$ and $G$ to $G^{-1}$ and use the hypothesis $D_2(G^+)\neq G^+$. \\ 
		
		$\bullet$ \textbf{Step 2 : From local to global quasi-isometry}\\
		
		In the step, in order to conclude the proof of Proposition \ref{Sphere-unif-quasi-geod}, we use the local-global lemma, which enables us to pass from local-quasi-geodesicity to global quasi-geodesicity, under the assumption of hyperbolicity. This classical lemma from hyperbolic geometry can be found for example in \cite{coornaert_geometrie_1990}, we recall it hereafter :
		
		\begin{Lemma}[Local-Global, \cite{coornaert_geometrie_1990}, Chapter 3, Theorem 1.4] \label{local-global}
			Let $X$ be a geodesic $\delta$-hyperbolic space. For all pairs $(\lambda,k)$, with $\lambda \geq 1$ and $k \geq 0$, there exists a real number $L$ and a pair $(\lambda',k')$ such that every $(\lambda,k,L)$-local-quasi-geodesic is a $(\lambda',k')$-quasi-geodesic (global). Moreover, $\lambda',k'$ and $L$ only depend on $\delta, \lambda$ and $k$.
		\end{Lemma}

		The sequence $(G^no)_{n\in \Z}$ is a quasi-isometry (since $G$ is hyperbolic), so there exist two constants $\lambda_1 >0,k_1>0$ such that $\displaystyle \frac{1}{\lambda_1}|G^n|-k_1 \leq d(G^no,o)=d(G^{-n}o,o)$ for all $n \geq 0$. Moreover, we have just proven in step 1 that there exists two constants $\lambda_2 >0, k_2>0$, such that $\displaystyle \frac{1}{\lambda_2}|G^nD_1G^{-m}|-k_2 \leq d(G^{-n}D_1G^mo,o)$ and $\displaystyle \frac{1}{\lambda_2}|G^nD_2G^{-m}|-k_2 \leq d(G^{-n}D_2G^mo,o)$, for all $n \geq0, m \geq 0$. Now apply Lemma \ref{local-global} to the pair  $(\lambda,k)=(\max(\lambda_1,\lambda_2),\max(k_1,k_2))$ and introduce $L>0, (\lambda',k')$ as defined in the Lemma. Fix $N=\lfloor L \rfloor +1$ and choose $W$ any sequence in $\mathcal{G}_N(G,\cF_1,\cF_2)$ which is associate to a bi-infinite words $H=(H_n)_{n\in \Z} \in \mathcal{H}_N(G,D_1,D_2)$. Thus, the subwords of $H$ of length smaller than $L$ are all either of the form $G^n$, $G^{-n}$, $G^nD_1G^{-m}$, $G^nD_1^{-1}G^{-m}$, $G^nD_2G^{-m}$, or $G^nD_2^{-1}G^{-m}$, with $n \geq 0,m \geq 0$. Therefore, the sequence of points $(x_n=W_no)_{n\in \Z}$ is a  $(\lambda,k,L)$-local-quasi-geodesic. So, by the local-global lemma \ref{local-global}, there exists $\lambda' \geq 1,k'\geq 0$ (only depending on $\lambda$ and $k$, that is on $\delta, G, \cF_1, \cF_2$ and $o$), such that $(x_n)_{n\in \Z}$ is a  $(\lambda',k')$-quasi-geodesic (global).  
	\end{proof}
	
	\subsection{Properties on Bowditch representations of $\groupe$}
	\label{sph-prop-Bowditch}
	
	We establish the useful fact that all the images of simple closed curves by a Bowditch representation are hyperbolic isometries. We prove a slightly stronger result, namely that given a Bowditch representation, we can forget the additive constant in the definition and replace the displacement length by the stable length.  
	
	\begin{Lemma} \label{simple->hyp}
		Let $\rho : \groupe \to \mathrm{Isom}(X)$ be a Bowditch representation with constant $C,D$ and $\gamma \in \simple$. Then $\displaystyle \frac{1}{C}\Vert\gamma\Vert \leq l_S(\rho(\gamma))$ and $\rho(\gamma)$ is hyperbolic. 
	\end{Lemma}
	
	\begin{proof} 
		We assume (without loss of generality) that $\gamma$ is cyclically reduced. We introduce $\delta_1, \delta_2 \in \groupe$ given by Corollary \ref{gamma^n-simple}. Therefore, by this corollary, for all $n \in \N$, the word $\gamma^n\delta_1 \gamma^{-n}\delta_2$ is simple. In order to apply the Bowditch hypothesis to $\gamma^n\delta_1 \gamma^{-n}\delta_2$, we want to study the cyclically reduced word length of $\gamma^n\delta_1 \gamma^{-n}\delta_2$. The word $\gamma^n\delta_1 \gamma^{-n}\delta_2$ is not necessarily reduced, but the key point is that $\delta_1$ and $\delta_2$ only depend on $\gamma$ and not on $n$, so the possible simplifications in the word $\gamma^n\delta_1 \gamma^{-n}\delta_2$ are independent of $n$, for $n$ large enough. For completeness, let us write the details : \\
		Consider two integers $n_1$ and $n_2$ satisfying $n_1|\gamma|>|\delta_1|$ and $n_2|\gamma|>|\delta_2|$. Denote $\delta'_1=\gamma^{n_1}\delta_1\gamma^{-n_1}$ and $\delta'_2=\gamma^{-n_2}\delta_2\gamma^{n_2}$. The elements $\delta'_1=\gamma^{n_1}\delta_1\gamma^{-n_1}$ and $\delta'_2=\gamma^{-n_2}\delta_2\gamma^{n_2}$ may have some simplifications, but because of the assumptions on $n_1$ and $n_2$ (and the fact that $\gamma$ is cyclically reduced), the simplifications are "bounded" in the sense that $\delta'_1$ and $\delta'_2$ are non empty (in fact they have at least two letters) and the following hold :
		\begin{itemize}
			\item the first letter of $\delta'_1$ is equal to  the first letter of $\gamma$,
			\item the last letter of $\delta'_1$ is equal to the inverse of the first letter of $\gamma$,
			\item the first letter of $\delta'_2$ is equal to the inverse of the last letter of $\gamma$,
			\item the last letter of $\delta'_2$ is equal to the last letter of $\gamma$.
		\end{itemize}
		Moreover, for all $n \in \N$ we have :
		\begin{align*}
			\gamma^n\delta_1\gamma^{-n}\delta_2 & =\gamma^{n_2}\gamma^{n-n_1-n_2}\gamma^{n_1}\delta_1\gamma^{-n_1}\gamma^{-n+n_1+n_2}\gamma^{-n_2}\delta_2
		\end{align*}
		so the element $\gamma^{n-n_1-n_2}\delta'_1\gamma^{-n+n_1+n_2}\delta'_2$ is a cyclic permutation of $\gamma^n\delta_1\gamma^{-n}\delta_2$. Now notice that the observations made above on the first and last letters of $\delta'_1$ and $\delta'_2$ imply that for all $n\geq n_1+n_2+1$ the word $\gamma^{n-n_1-n_2}\delta'_1\gamma^{-n+n_1+n_2}\delta'_2$ is cyclically reduced (here we write $\delta'_1$ and $\delta'_2$ as reduced word). In other words, for all $n\in \N^*$, $\gamma^n\delta'_1\gamma^{-n}\delta'_2$ is simple and cyclically reduced. \\ 
		
		Then, we use the Bowditch hypothesis on $\gamma^n\delta'_1\gamma^{-n}\delta'_2$ to write the following inequalities : 
		\begin{align*}
			\frac{1}{C}\Vert\gamma^n\delta'_1 \gamma^{-n}\delta'_2\Vert-D & \leq d(\rho(\gamma^n\delta'_1 \gamma^{-n}\delta'_2)o,o) \\
			& \leq d(\rho(\gamma^n)o,o)+d(\rho(\delta'_1)o,o)+d(\rho(\gamma^{-n})o,o)+d(\rho(\delta'_2)o,o)
		\end{align*}
		
		But since $\gamma^n\delta'_1\gamma^{-n}\delta'_2$ is cyclically reduced, we have $\Vert\gamma^n\delta'_1 \gamma^{-n}\delta'_2\Vert=2n|\gamma|+|\delta'_1|+|\delta'_2|=2n\Vert\gamma\Vert+|\delta'_1|+|\delta'_2|$ (the last equality holds since $\gamma$ is supposed cyclically reduced), so after dividing by $n$ :
		\begin{align*}
			\frac{2}{C}\Vert\gamma\Vert+\frac{|\delta'_1|+|\delta'_2|}{nC}-\frac{D}{n} & \leq \frac{2}{n}d(\rho(\gamma^n)o,o)+\frac{1}{n}(d(\rho(\delta'_1)o,o)+d(\rho(\delta'_2)o,o))
		\end{align*}
		And then by taking the limit when $n \to \infty$ (and dividing by 2):
		\begin{align*}
			\frac{1}{C}\Vert\gamma\Vert \leq l_S(\rho(\gamma)) \qquad \text{ by definition of the stable length},  
		\end{align*}
		which is the desired inequality. Hence, for every simple word $\gamma$, $l_S(\rho(\gamma))>0$, thus $\rho(\gamma)$ is hyperbolic. 
	\end{proof}
	
	Now we establish the fact that the hypothesis required by Lemma \ref{Sphere-unif-quasi-geod} is satisfied when the isometries $D_1$, $D_2$ and $G$ comes from a Bowditch representation.
	
	\begin{Lemma} \label{D(G)neqG}
		Let $\rho : \groupe \to \mathrm{Isom}(X)$ be a Bowditch representation with constants $C,D$ and $\gamma \in \simple$. Let $\delta_1, \delta_2 \in \groupe$ such that $\gamma^n\delta_1 \gamma^{-n}\delta_2$ is simple for an infinite number of $n \in \N$. Denote $D_1=\rho(\delta_1), D_2=\rho(\delta_2)$ and $G=\rho(\gamma)$. Then $D_1(G^-)\neq G^{-}$ and $D_2(G^+)\neq G^+$.
	\end{Lemma}
	
	Before starting the proof, recall that we have shown that $G=\rho(\gamma)$ is an hyperbolic isometry (since $\gamma$ is simple, see Lemma \ref{simple->hyp}), therefore $G^-$ and $G^+$ are well-defined. 
	
	\begin{proof}
		First notice that $\gamma^n\delta_1\gamma^{-n}\delta_2$ is not necessarily supposed cyclically reduced. However, in the same way as in the beginning of the proof of the previous Lemma (\ref{simple->hyp}), we can assume that $\gamma^n\delta_1\gamma^{-n}\delta_2$ is cyclically reduced after possibly conjugating $\delta_1$ and $\delta_2$ by a power of $\gamma$ (independent of $n$). This change does not affect the conclusion of the lemma : indeed, if $D_1(G^-)\neq G^-$, the same non-equality holds when changing $D_1$ to a conjugate of $D_1$ by any power of $G$ (because $G^-$ is a fixed point of $G$), and the same is true for the non-equality $D_2(G^+)\neq G^+$. Thus in the following, we will assume that $\gamma^n\delta_1\gamma^{-n}\delta_2$ is cyclically reduced. \\ 
		
		Consider the two sequences of points in $X$ define by $x_n={G}^{-n}o$ and $y_n=D_1G^{-n}o$ for all $n \in \N$. By contradiction, assume that $D_1(G^-) = G^-$. Under this assumption, we want to show that the distance $d(x_n,y_n)$ is bounded. Let $l_{G^-}$ be any geodesic with $G^-$ as an endpoint. Because $G$ is a hyperbolic isometry, the sequence $(G^{-n}o)_{n\in \Z}=(x_n)_{n \in \Z}$ is a quasi-isometry with attracting fixpoint $G^{-}$. Furthermore, $(D_1G^{-n}o)_{n\in \Z}=(y_n)_{n\in \Z}$ is also a quasi-isometry with attracting fixpoint $D_1(G^-)$. Thus, under this assumption $D_1(G^-)=G^{-}$, the sequences $(x_n)_{n\in \Z}$ and $(y_n)_{n \in \Z}$ are both quasi-geodesics with the same attracting fixpoint, $G^-$. Hence, the stability of quasi-geodesic in $\delta$-hyperbolic spaces gives the existence of a constant $K>0$ such that $(x_n)_{n\in \N}$ and $(y_n)_{n \in \N}$ both stay at a distance at most $K$ of $l_{G^-}$. Now consider a projection $p : X \to l_{G^-}$ on the geodesic $l_{G^-}$. By definition of the projection we have that for all $n \in \N$, $d(x_n,p(x_n)) \leq K, d(y_n,p(y_n))\leq K$. Then we deduce :
		\begin{equation} \label{maj-x_n,y_n}
			d(x_n,y_n)\leq d(p(x_n),p(y_n))+2K
		\end{equation}
		Moreover, because $p(x_n),p(y_n)$ and $p(o)$ all belong to the same geodesic $l_{G^-}$, we can write :
		\begin{equation} \label{egal-p(x_n),p(y_n)}
			d(p(x_n),p(y_n))=|d(p(x_n),p(o))-d(p(y_n),p(o))|.
		\end{equation}
		Using again that $(x_n)_{n\in \N}$ and $(y_n)_{n \in \N}$ both stay at a distance at most $K$ of $l_{G^-}$, we obtain the two following inequalities :
		\begin{align*}
			d(x_n,o)-2K \leq d(p(x_n),p(o)) \leq d(x_n,o)+2K \\
			d(y_n,o)-2K \leq d(p(y_n),p(o)) \leq d(y_n,o)+2K 
		\end{align*}
		from which we deduce : 
		\begin{equation} \label{maj-|d-d|}
			|d(p(x_n),p(o))-d(p(y_n),p(o))| \leq |d(x_n,o)-d(y_n,o)|+4K.
		\end{equation}
		Finally, we bound $|d(x_n,o)-d(y_n,o)|$ using the triangle inequality :
		\begin{align*}
			|d(x_n,o)-d(y_n,o)|& =|d(G^{-n}o,o)-d(D_1G^{-n}o,o)| \\ 
			& = |d(D_1G^{-n}o,D_1o)-d(D_1G^{-n}o,o)| \\ 
			& \leq d(D_1o,o)
		\end{align*}
		and then, together with \eqref{maj-x_n,y_n}, \eqref{egal-p(x_n),p(y_n)} and \eqref{maj-|d-d|}, we conclude that $d(x_n,y_n)$ is bounded. 
		But by hypothesis, $\gamma^n\delta_1 \gamma^{-n}\delta_2$ is simple for an infinite number of $n \in \N$, so we can conclude using the Bowditch inequality that, for an infinite number of $n$ :
		\begin{align*}
			\frac{1}{C}\Vert\gamma^n\delta_1 \gamma^{-n}\delta_2\Vert-D & \leq d(\rho(\gamma^n\delta_1 \gamma^{-n}\delta_2)o,o) = d(G^{-n}o,D_1G^{-n}D_2o) \\
			& \leq d(G^{-n}o,D_1G^{-n}o) + d(D_1G^{-n}o,D_1G^{-n}D_2o)\\ 
			& =d(x_n,y_n)+d(D_2o,o)
		\end{align*}
		Hence the right hand side of this inequality is bounded in $n$ whereas the left hand side is not (recall that $\gamma^n\delta_1\gamma^{-n}\delta_2$ is supposed cyclically reduced), this is a contradiction. From this contradiction we deduce that $D_1(G^-)\neq G^{-}$. \\
		Finally notice that if $\gamma^n\delta_1 \gamma^{-n}\delta_2$ is simple, so is $\gamma^{-n}\delta_2\gamma^n\delta_1$. Thus using what has been previously done we also have that $D_2((G^{-1})^{-})\neq (G^{-1})^{-}$. Since $(G^{-1})^{-}=G^+$, the lemma is proved.
	\end{proof}
	
	\section[Uniform tubular neighborhoods for Bowditch   representations of $\groupe$]{Uniform tubular neighborhoods for Bowditch representations of $\groupe$} 
	\sectionmark{Uniform tubular neighborhoods}
	\label{sph-first-step}
	
	This section is dedicated to proving Proposition \ref{BQ=>CTU-Sph}, which says that the orbit map restricted to simple leaves stays in a uniform tubular neighborhood of the axes of simple elements in $X$. This is the core of the proof of the simple-stability of Bowditch representations of the four-punctured sphere group, and the remaining part will be done in section \ref{second-step-proof-sph}. The proof will follow the same strategy as the proof of the corresponding proposition in the case of the free group of rank two $\F_2$, already done by the author in \cite{schlich_equivalence_2024} (Proposition 5.2). Some of the lemmas have a proof that needs to be adapted, because the combinatorics of the words are different in this case. On the other hand, for some lemmas the proof will be identical, and in this case we will refer in the proofs to the corresponding lemmas in \cite{schlich_equivalence_2024}. \\
	
	Let $X$ be a $\delta$-hyperbolic, geodesic and visibility space and $o \in X$ a basepoint. When $A$ is a hyperbolic isometry of $X$, it defines two fixed point $A^+$ and $A^-$ in the boundary of $X$, respectively the attracting and repelling one. Thus we can define $\mathrm{Axis}(A)$ to be the union of all the geodesics in $X$ joining the two distinct points $A^+$ and $A^-$ (such a geodesic always exists because $X$ is a visibility space but is not necessarily unique). We easily see that $\mathrm{Axis}(A)$ is $A$-invariant and so follows the $A$-invariance of the map $d(\cdot, \mathrm{Axis}(A))$. For a subset $Y \subset X$ and $K>0$, let us denote $N_K(Y)$ the $K$-neighborhood of $Y$, namely $N_K(Y):=\{x \in X \, : \, d(x,Y) \leq K \}$. \\
	Now recall that $L_\gamma$ denotes the geodesic in the Cayley graph of $\groupe=\F_3$ which is the axis of $\gamma$, and that for any element $\gamma \in \simple$, we proved in Lemma \ref{simple->hyp} that $\rho(\gamma)$ is hyperbolic so $\axis(\rho(\gamma))$ is well-defined. 
	
	\begin{Proposition}[Uniform tubular neighborhoods] \label{BQ=>CTU-Sph} Let $X$ be a $\delta$-hyperbolic, geodesic and visibility space, and $\rho : \groupe \to \mathrm{Isom}(X)$ a Bowditch representation. Then:
		\begin{equation*}
			\exists K>0, \qquad \forall \gamma \in \simple, \qquad \tau_\rho(L_\gamma) \subset N_K(\axis(\rho(\gamma)))
		\end{equation*}
	\end{Proposition}
	
	\begin{proof}
		Pick $\rho : \groupe \to \mathrm{Isom}(X)$ a Bowditch representation of $\groupe$, $o \in X$ a basepoint and let $C,C'>0$ be two constants such that :
		\begin{equation} \label{BQ-constant-sph}
			\forall \gamma \in \simple, \qquad \frac{1}{C}\Vert\gamma\Vert \leq l(\rho(\gamma)) \quad \text{ and } \quad \forall u \in \groupe, \quad d(\rho(u)o,o) \leq C'\vert u\vert.
		\end{equation}
		Recall that the existence of the constant $C$ comes from the hypothesis that the representation is Bowditch and Lemma \ref{simple->hyp} and the existence of the constant $C'$ is true for any representation of a finitely generated group. Note that such constants automatically satisfy $CC'\geq 1$. \\
		
		By contradiction, let us suppose that there exists a sequence $(\gamma_n)_{n\in \N}$ of cyclically reduced simple elements in $\simple$ satisfying the following hypothesis : 
		
		\begin{equation*}   \tag{$H_2$} \label{hyp-absurde-sph} 
			\sup \{d(x,\axis(\rho(\gamma_n))) \mid x \in \tau_\rho(L_{\gamma_n}) \} \underset{n \to \infty} \longrightarrow +\infty
		\end{equation*}
		
		Such a sequence in now fixed for all the section \ref{sph-first-step}. Up to extracting, we can assume that the elements $\gamma_n$ are pairwise distinct and that $\Vert \gamma_n \Vert \to \infty$, otherwise it would contradict the hypothesis \eqref{hyp-absurde-sph}.
		
		\subsection{Boundedness of the integers of the continued fraction expansion of $\gamma_n$} ~\\ 
		For all integer $n$, the element $\gamma_n$ is simple, thus by Proposition \ref{essential->Q} corresponds to a rational, and then we can consider the continued fraction expansion of its slope : 
		\begin{equation*}   
			\slope(\gamma_n)=[N_1^n,\hdots,N_{r(n)}^n].
		\end{equation*}
		Now we will prove that we can restrict our study to the case where the integers $N_i^n$ are bounded in $n$ : 
		\begin{Lemma} \label{forme-frac-continue-sph}
			Up to subsequence, $r(n) \to \infty$ and for all $i \in \N$, $(N_i^n)_{n\in \N \hspace{0,1cm} \vert \hspace{0,1cm} r(n) \geq i}$ is bounded. 
		\end{Lemma}
		
		\begin{proof}
			Suppose there exists $k \in \N$ such that $(N_k^n)_{n}$ is defined for infinitely many values of $n$ and is not bounded. Then consider the integer $0 \leq i $ such that $(N_{i+1}^n)_n \underset{n \to \infty} \longrightarrow \infty$ (after passing to subsequence) and for all $1\leq j \leq i$, the sequence $(N_j^n)_n$ is bounded. Therefore, again after passing to subsequence we assume that for all $1 \leq j \leq i$ there exists an integer $N_j$ such that for all $n \in \N$ such that $r(n) \geq j$, $N_j^n=N_j$. Thus $\slope(\gamma_n)=[N_1,\hdots,N_i,N_{i+1}^n,\hdots,N_{r(n)}^n]$. Now, use Lemma \ref{general-form-gamma} to deduce the existence of a simple word $\gamma_i \in \simple$ of slope $[N_1,\hdots,N_i]$ and two words $\delta_1,\delta_2 \in \groupe$, such that for all $n \in N$, $\gamma_n$ can be written as a (cyclic-permutation of a) concatenation of subwords of the form : $$(\gamma_i)^{m_1(n)}\t{\delta_1}(\gamma_i)^{-m_2(n)}\t{\delta_2},$$ with $m_1(n),m_2(n) \geq \frac{N_{i+1}^n-1}{2}$, 
			$\t{\delta_1}\in \{\delta_1,\delta_1^{-1}\}$ and $\t{\delta_2}\in \{\delta_2,\delta_2^{-1}\}$. \\
			Now set $G=\rho(\gamma_i),D_1=\rho(\delta_1)$ and $D_2=\rho(\delta_2)$. Then, by considering the bi-infinite word obtained by concatenating infinitely many copies of $\gamma_n$, or equivalently the bi-infinite word obtained by following the geodesic $L_{\gamma_n}$ in the Cayley graph, we can see $\rho_{| L_{\gamma_n}}$ as an element of $\mathcal{G}(G,D_1,D_2)$ (the definition is given at the beginning of section \ref{sph-UQG}). Let us introduce the constants $\lambda>0, k\geq 0$ and $N \in \N^*$ as defined in Lemma \ref{Sphere-unif-quasi-geod}. For $n$ sufficiently large, $m_1(n) \geq N$ and $m_2(n) \geq N$ because $m_1(n),m_2(n) \geq \frac{N_{i+1}^n-1}{2}$, and by hypothesis $N_{i+1}^n \underset{n \to \infty} \longrightarrow \infty$. Hence, $\rho_{L_{\gamma_n}}$ is an element of $\mathcal{G}_N(G,D_1,D_2)$ for $n$ sufficiently large. Let us now use the second part of Lemma \ref{general-form-gamma} to justify that for all $n \in \N$, the word $(\gamma_i)^n\delta_1(\gamma_i)^{-n}\delta_2$ is simple. Then, by Lemma \ref{D(G)neqG}, we obtain that $D_1(G^-)\neq G^-$ and $D_2(G^+)\neq G^+$. Now we can apply Proposition \ref{Sphere-unif-quasi-geod} to justify that $\tau_\rho(\gamma_n)$ are uniform (in $n$) quasi-geodesic in $X$. \\
			This contradicts our hypothesis \eqref{hyp-absurde-sph} on $\rho$ for $n$ sufficiently large. Hence, for all $i \in \mathbb{N}$, $(N_i^n)_n$ is bounded. \\
			Let us now justify that $r(n) \to + \infty$. If $r(n)$ stays bounded, $r(n)\leq R$, then for all $1 \leq i \leq R$, $(N_i^n)_n$ is bounded by what has been previously done and so the word length of $\gamma_n$ is also bounded, which is false. Thus $r(n) \to +\infty$. \\
			In particular, we deduce that under the assumption \eqref{hyp-absurde-sph}, the sequence  $(N_i^n)_n$ is always well-defined for $n$ sufficiently large ($n$ such that $r(n) \geq i$).
		\end{proof}
		
		\subsection{Uniform bound on the lengths $l_i(\gamma_n)$}~\\ 
		\label{bornes-li-sph}
		From Lemma \ref{forme-frac-continue-sph}, we deduce that for all integer $i$, there exists a constant $N_i$ such that $N_i^n \leq N_i$ for all integer $n$.  Then using inequality \eqref{rec-li} and \eqref{l'i<2li} on the lengths $l_i(\gamma_n)$ and $l_{i-1}(\gamma_n)$ of section \ref{subsec:simple closed curves}, we obtain the following uniform bound on the lengths $l_i(\gamma_n)$ :
		
		\begin{align}
			\forall n \in \mathbb{N}, \forall 1 \leq i\leq r(n), \qquad  \frac{l_i(\gamma_n)}{l_{i-1}(\gamma_n)} \leq N_i+1. \\
			\intertext{We deduce for every integer $i$ the existence of constants $L_i>0$ such that :}
			\forall n\in \mathbb{N}, \forall 0 \leq i \leq r(n), \qquad i \leq l_i(\gamma_n) \leq L_i. \label{L_i}
		\end{align}
		
		\subsection{Excursions of the orbit map} 	\label{excursion-orbit-map-sph} ~\\ 
		Let $\gamma$ be a cyclically reduced simple element in $\simple$. Recall that $L_\gamma$ is the axis of $\gamma$ in the Cayley graph of $\groupe=\F_3$.
		We define the map : $E_\gamma : L_\gamma \longrightarrow \R_+ $ by $E_\gamma(u)=d(\tau_\rho(u),\axis(\rho(\gamma)))$. The map $E_\gamma$ is $\gamma$-invariant because of the $\rho$-equivariance of the orbit map $\tau_\rho$  and  the $\rho(\gamma)$-invariance of $\mathrm{Axis}(\rho(\gamma))$. Moreover, the map $E_\gamma$ is Lipschitz-continuous (hence continuous) because the orbit map $\tau_\rho$ is Lipschitz-continuous and the distance map to any subspace in a metric space is $1$-Lipschitz-continuous. 
		
		\begin{Definition} \label{excursion-orbit-map-sph-def} Let $\gamma$ be a simple element in $\simple$. Let $[u,v] \subset L_\gamma$ be a segment in the geodesic~$L_\gamma$. We say that $[u,v]$ is an \emph{excursion} if the map $E_{\gamma}$ satisfies $E_{\gamma}(u)=E_{\gamma}(v)$ and for all $t \in [u,v], E_{\gamma}(t) \geq E_{\gamma}(u)$. In this case, we say that the non-negative real $d(u,v)$, which is the length of the segment $[u,v]$ in $L_\gamma$, is the \emph{length of the excursion $[u,v]$}. Let $K \geq 0$. We say that $[u,v]$ is a \emph{$K$-excursion} if $[u,v]$ is an excursion such that $E_\gamma(u)=K$. \\ 
		At last, we say that $\gamma$ has an excursion (respectively a $K$-excursion) if there exists $[u,v] \in L_\gamma$ such that $[u,v]$ is an excursion (respectively a $K$-excursion). 
		\end{Definition}
		
		The following lemma gives the existence of excursions as large and as long as we want. 
		
		\begin{Lemma} \label{excursions-grandes-sph}
			There exist two sequences of positive reals $(K_n)_{n\in \mathbb{N}}$ and $(l_n)_{n \in \mathbb{N}}$, such that $K_n \to \infty$, $l_n \to \infty$ and, up to subsequence,  for all $n \in \mathbb{N}$, $\gamma_n$ has a $K_n$-excursion of length~$l_n$. 
		\end{Lemma}
		
		\begin{proof}
			The proof is exactly the same as when studying Bowditch representations of $\F_2$, see Lemma 5.11 in \cite{schlich_equivalence_2024}.
		\end{proof}
		
		\subsection{Quasi-loops}~\\
		\label{quasi-loop-sph}
		We now introduce quasi-loops, which are elements in the group  that do not displace the basepoint much, compared to their word-length. 
		\begin{Definition}\label{sph-QL}
			Let $\varepsilon > 0$ and $w \in \groupe=\F_3$ (not necessarily a simple element). We say that $w$ is an $\emph{$\varepsilon$-quasi-loop}$ if it satisfies the following inequality :
			$ d(\rho(w)o,o) \leq \varepsilon ~ |w|$.
		\end{Definition}
		
		Let $\gamma$ be a simple element in $\simple$ and $u \in L_\gamma$. The notation  $\lfloor u \rfloor$ stands for the integer point in $L_\gamma$ just before $u$ (if $u$ is an integer point in $L_\gamma$, $\lfloor u \rfloor=u$) and $\lceil u \rceil$ for the integer point of $L_\gamma$ strictly just after $u$ (thus $\lfloor u \rfloor$ and $\lceil u \rceil$ are the endpoints of an edge of length $1$ in the Cayley graph and $u$ belongs to this edge). \\ 
		
		The following lemma says that by considering very large and long excursions we obtain quasi-loops. 
		\begin{Lemma} \label{excursion->QB-sph}
			Let $\eps > 0$. There exist $l_\eps > 0$ and $K_\eps >0$ such that for all simple elements $\gamma \in \simple$, for all $K \geq K_\eps, l \geq l_\eps$, if $[u,v]$ is a $K$-excursion of length $l$, then the element  $w=\lf u \rf^{-1}\lf v \rf$ (which is a subword of $\gamma$) is an $\eps$-quasi-loop. 
		\end{Lemma}
		\begin{proof}
			The proof is exactly the same as when studying Bowditch representations of $\F_2$, see Lemma 5.12 in \cite{schlich_equivalence_2024}. It relies mainly on the geometry of large uniform neighborhoods of geodesics in $X$, studied in the appendix A in \cite{schlich_equivalence_2024}. 
		\end{proof}
		
		\subsection{Finding many quasi-loops}~\\
	 	Let $\gamma$ be a simple element in $\simple$. In the following lemma, we show that we can find, in any sufficiently long subword $u$ of $\gamma$, many disjoint quasi-loops (conditions \eqref{concat} and \eqref{are-ql}) which cover a uniformly bounded proportion of the subword $u$ (condition \eqref{unif-prop}). We also require that the "remainders" in the word $u$ are not too small (condition \eqref{remainders}). 
		
		\begin{Lemma} \label{decoupe-restes-c-sph} Let $\beta=\frac{1}{480}(=\frac{\alpha}{16}$, where $\alpha$ is the constant introduced in Theorem \ref{magic-len-S_{0,4}}). \\
			Let $\displaystyle 0<\varepsilon < \frac{1}{C}$ and $r >0$. There exists a constant $R>0$ and an integer $n_0 \in \N$, such that, given any  integer $n \geq n_0$ and subword $u$ of $\gamma_n$ such that $|u| \geq R$, there exists a positive integer $q \in \N^*$, a subset $QL \subset \{1,\hdots,q\}$ and $q$ words $u_1, \hdots,u_q \in \F_2$ such that :
			\begin{enumerate}
				\item \label{concat} $u=u_1\hdots u_q$
				\item \label{are-ql} For all $k \in QL$, $u_k$ is an $\varepsilon$-quasi-loop
				\item \label{unif-prop} $\displaystyle \sum_{k \in QL} |u_k| \geq \beta|u|$
				\item \label{remainders} For all $k \notin QL$, $|u_k| \geq  r$
			\end{enumerate}
		\end{Lemma}
		In order to prove Lemma \ref{decoupe-restes-c-sph}, we will use the following result on excursions, which allows us to find many other smaller excursions in a given one. This lemma and its proof can be found in Lemma 6.10 in \cite{schlich_equivalence_2024}. 
		
		\begin{Lemma} \label{temps-excursion}
			Let $l>0,K>0$ and $[u,v]$ be a $K$-excursion of length $l$. Then there exists $[u',v']\subset [u,v]$ such that $[u',v']$ is a $K'$-excursion of length $l'$, with $\frac{l}{2} \leq l' < l$ and $K' \geq K$. 
		\end{Lemma}
		
		\begin{proof}[Proof of Lemma \ref{decoupe-restes-c-sph}]
			Let us first introduce the constants $K_\eps$ and $l_\eps$ defined by Lemma \ref{excursion->QB-sph}. Now we choose and fix for the rest of the proof an integer $i$ satisfying $i \geq \max(10,2r+7,2l_\eps+5)$. Recall that the constant $L_i$ has been introduced in \eqref{L_i}. Let us also introduce the sequences $(K_n)_{n\in \N}$ and $(l_n)_{n\in \N}$ of Lemma \ref{excursions-grandes-sph}. Since $K_n \to + \infty$ and $l_n \to +\infty$, we can find an integer $n_0$ such that for all $n\geq n_0$, we have $K_n \geq K_\eps, l_n\geq l_\eps$ and $L_i-5< l_n$. We also set $\alpha$ to be the constant defined in Theorem \ref{magic-len-S_{0,4}} ($\alpha=\frac{1}{30}$, but its precise value does not matter, as long as it is a universal constant). Finally, we fix $R$ such that $R \geq \frac{3L_i}{\alpha}$. Note that due to the value of the constant $\alpha$, we also automatically have $R \geq 9L_i$. Now let us show that with this choices for $R$ and $n_0$, the lemma is true. \\
			
			Let $n \geq n_0$ and $u$ be a subword of (a cyclic-permutation of) $\gamma_n$ (or its inverse) such that $|u|\geq R$. \\ 
			By Lemma \ref{excursions-grandes-sph}, $\gamma_n$ has a $K_n$-excursion of length $l_n$. Moreover, we have $l_i(\gamma_n) \leq L_i$ by \eqref{L_i} so in particular $l_i(\gamma_n)-5 <l_n$. Therefore, we can use Lemma \ref{temps-excursion} to ensure the existence of a $K'_n$-excursion of length $l'_n$, with $\frac{l_i(\gamma_n)-5}{2}\leq l'_n < l_i(\gamma_n)-5$ and $K'_n \geq K_n$. Denote it by $[x,y] \subset L_{\gamma_n}$. But then we have $K'_n \geq K_n \geq K_\eps$ and 
			\begin{align*}
				l'_n & \geq \frac{l_i(\gamma_n)-5}{2} \geq \frac{i-5}{2} \qquad \text{ by \eqref{L_i}} \\
				& \geq l_\eps \qquad \text{ by the definition of $i$ } 
			\end{align*}
			so we can apply Lemma \ref{excursion->QB-sph} to ensure that the element $v=\lf x \rf^{-1}\lf y \rf $ (which is a subword of $\gamma$) is an $\eps$-quasi-loop. Let us compute the length of $v$ :
			\begin{align}
				\text{ We have } & |v| =d(\lf x \rf, \lf y \rf ) = d(\lf x \rf, x)+d(x,y)-d(y,\lf y \rf) \nonumber \\
				\text{so } \qquad & d(x,y)-1 < |v| < d(x,y)+1 \nonumber \\
				\text{and } \qquad & l'_n-1 < |v| < l'_n+1 \qquad \text{ since } d(x,y)=l'_n \nonumber \\
				\text{ then } \qquad & \frac{l_i(\gamma_n)-5}{2}-1 < |v| < l_i(\gamma_n)-5+1 \nonumber \\
				\text{ and finally } \qquad & \frac{l_i(\gamma_n)-7}{2} \leq |v| \leq l_i(\gamma_n)-5 \qquad \text{ because $|v|$ and $l_i(\gamma_n)$ are integers.} \label{encadr-|v|}
			\end{align}
			Now consider $w$ the subword of $\gamma_n$ of length $l_i(\gamma_n)-5$ such that $v$ is a prefix of $w$ and let us write $w=vv'$, with $|w|=l_i(\gamma_n)-5$. \\
			
			We are now going to use Theorem \ref{magic-len-S_{0,4}} for the subword $u$ of $\gamma_n$ and $w$. This is possible because we have :
			\begin{itemize}
				\item chosen $i$ such that $l_i(\gamma_n)$ is sufficiently large : $l_i(\gamma)\geq i \geq 10$,
				\item chosen $w$ of the right length : $|w|=l_i(\gamma_n)-5$,
				\item fix $u$ sufficiently large : 
				\begin{align*}
					|u| & \geq R \geq 9L_i \qquad \text{by definition of $R$} \\
					& \geq 9l_i(\gamma_n) = 3(2l_i(\gamma)+l_i(\gamma)) \geq 3(l'_i(\gamma_n)+1+l_i(\gamma)) \text{ by the inequality \eqref{l'i+1<2li}} 
				\end{align*}
			\end{itemize}
			Then we can write $u=u_1\hdots u_q$ such that there exists a subset $\cI \subset \{1,\hdots,q\}$ satisfying :
			\begin{enumerate}
				\item For all $k\in \cI$, $u_k \in \{w,w^{-1}\}$.
				\item $\displaystyle \sum_{k\in \cI} |u_k| \geq \alpha|u|$.
			\end{enumerate}
			But recall that $w=vv'$, so the inverse is $w^{-1}={v'}^{-1}v^{-1}$. Then for all $k \in \cI$, we can write $u_k=vv'$ or $u_k={v'}^{-1}v^{-1}$. Therefore, the word $u$ can be written as a concatenation of the subwords $u_k$ for $k \notin \cI$ and $v,v',v^{-1},{v'}^{-1}$, with at least $\# \cI$ terms of the form $v$ or $v^{-1}$. Thus, by denoting $v_k$ these $\# \cI$ appearances of $v$ and $v^{-1}$ and by combining together all successive terms that are not $v$ or $v^{-1}$ (that is all successive terms of the form $u_k$ for $k \notin \cI$, $v',{v'}^{-1}$) into factors called $v'_k$, we can write the following concatenation of $u$ : 
			\begin{equation*}
				u=v'_1v_1v'_2v_2\hdots v'_pv_pv'_{p+1}
			\end{equation*}
			satisfying $p=\# \cI$ and for all $k \in \{1,\hdots,p\}$, $v_k \in \{v,v^{-1} \}$. \\
			Now let us show the four points of Lemma \ref{decoupe-restes-c-sph} :
			\begin{enumerate}
				\item  If $p$ is odd, we combine the previous terms together as follows :
				\begin{align}
					u & =\underbrace{v'_1v_1v'_2}v_2 \underbrace{v'_3v_3v'_4}  \hdots v_{p-1}\underbrace{v'_pv_pv'_{p+1}} \nonumber \\ 
					u & = \quad v''_1 \quad v_2 \quad v''_3 \quad   \hdots v_{p-1} \quad \,\, v''_p \label{decomposition-odd}
				\end{align}
				If $p$ is even, we combine the previous terms together as follows : 
				\begin{align}
					u & =\underbrace{v'_1v_1v'_2}v_2 \underbrace{v'_3v_3v'_4}  \hdots v_{p-2} \underbrace{v'_{p-1}v_{p-1}v'_pv_pv'_{p+1}} \nonumber \\ 
					u & = \quad v''_1 \quad v_2 \quad v''_3 \quad   \hdots v_{p-2} \qquad \quad v''_{p-1} \label{decomposition-even}
				\end{align}
				and this will be our decomposition of $u$, with $q=p$ if $p$ is odd and $q=p-1$ if $p$ is even, $u_k=v''_k$ if $k$ is odd and $u_k=v_k$ if $k$ is even, and $QL = \{ 1 \leq k \leq p-1 \mid k \text{  is even} \}$ . \\
				
				\item Let $k \in \{1,\hdots,p\}$, then $v_k \in \{v,v^{-1}\}$. But recall that by construction, $v$ is an $\eps$-quasi-loop, then so is $v^{-1}$, hence $v_k$ is always an $\eps$-quasi-loop. \\
				
				\item Let us start by showing that $p \geq 3$.\begin{align*}
					pL_i & \geq pl_i(\gamma_n) \qquad \text{ by definition of $L_i$} \\
					& \geq \# \cI (l_i(\gamma_n)-5) = \sum_{k \in \cI} |u_k| \qquad \text{ since } p=\# \cI \text{ and } |u_k|=|w|=l_i(\gamma_n)-5 \\ 
					& \geq \alpha|u| \qquad \text{ using point \ref{proportion} of Theorem \ref{magic-len-S_{0,4}}} \\ 
					& \geq \alpha R \qquad \text{by hypothesis on $u$} \\
					& \geq \alpha \frac{3L_i}{\alpha}=3L_i \qquad \text{ since $R$ is fixed such that } R \geq \frac{3L_i}{\alpha
					} \\
					\text{and thus we deduce } \qquad p & \geq 3.
				\end{align*} This ensures that in our previous decomposition \eqref{decomposition-odd} and \eqref{decomposition-even}, there is at least one term $v_k$. 
				Indeed, recall that $QL = \{ 1 \leq j \leq p-1 \mid j \text{  is even} \} $. In particular we have $ p = \left\{ \begin{array}{cc}
					2 \# QL +1  & \text{ if $p$ is odd, } \\
					2 \# QL +2  & \text{ if $p$ is even, }
				\end{array} \right. $
				\begin{align}
					\text{ so in any cases } \qquad & p\leq 2\# QL +2. \label{p<QL} \\
					\text{ and therefore } \qquad & 1 \leq \# QL. \label{1leqQL}
				\end{align}
				Let us show that $\displaystyle \sum_{k \in QL} |v_k| \geq \beta|u|$ : 
				\begin{align*}
					|u| & \leq \frac{1}{\alpha} \sum_{k \in \cI} |u_k| \qquad \text{ by point \ref{proportion} of Theorem \ref{magic-len-S_{0,4}}} \\
					& \leq \frac{1}{\alpha} \# \cI |w| \qquad \text{ by point \ref{change-letter} of Theorem \ref{magic-len-S_{0,4}}} \\
					& \leq \frac{1}{\alpha} p (l_i(\gamma_n)-5) \qquad \text{since } \# \cI =p \text{ and } |w|=l_i(\gamma_n)-5 \\ 
					& \leq \frac{1}{\alpha} (2 \# QL +2) (l_i(\gamma_n)-5) \qquad \text{ by the inequality \eqref{p<QL}} \\
					& \leq \frac{4}{\alpha} \# QL(l_i(\gamma_n)-5) \qquad \text{ by the inequality \eqref{1leqQL}} \\
					& \leq \frac{4}{\alpha} \# QL(2|v|+2) \qquad \text{using \eqref{encadr-|v|}}
					\end{align*}
				In addition :
					$ \displaystyle |v| \geq \frac{l_i(\gamma_n)-7}{20} \geq \frac{i-7}{2} \geq \frac{10-7}{2} \geq 1$. \\
					Then
					$ \displaystyle |u|  \leq \frac{4}{\alpha}\# QL\times 4 |v| \leq \frac{16}{\alpha} \sum_{k \in QL} |v_k| $   since $ v_k \in \{v,v^{-1} \}$, and thus we have the desired inequality with $\displaystyle \beta=\frac{\alpha}{16}$. \\
				
				\item Finally, let $k \notin QL$, and let us compute the length of $v''_k$ : 
				\begin{align*}
					|v''_k| & \geq |v'_kv_kv'_{k+1}|  \qquad\text{ (and there is in fact equality unless $p$ is even and $k=p-1$)} \\
					& \geq |v_k|=|v| \qquad \text{ by definition of $v_k$} \\
					& \geq \frac{l_i(\gamma_n)-7}{2} \qquad \text{ by \eqref{encadr-|v|}} \\
					& \geq \frac{i-7}{2} \qquad \text{ by inequality \eqref{L_i}} \\
					& \geq r \qquad \text{ since $i$ has been chosen such that $i \geq 2r+7$}. 
				\end{align*}
				And this finishes the proof. 
			\end{enumerate}       
		\end{proof}

		\subsection{Finding arbitrarily many quasi-loops}~\\
		Once we have Lemma \ref{decoupe-restes-c-sph}, the end of the proof of Proposition \ref{BQ=>CTU-Sph} is the same as in the case of Bowditch representations of $\F_2$. For completeness, we recall the last steps. \\ 
		
		Next lemma finds a simple word $\gamma$ (from the sequence $(\gamma_n)_{n\in \N}$) which contains disjoint quasi-loops which cover an arbitrarily large proportion of the word $\gamma$. 
		
		\begin{Lemma} \label{trouve-gamma-sph}
			Let $0 < \varepsilon < \frac{1}{C}$ and $1 - \frac{1}{C'}(\frac{1}{C}-\varepsilon) <  \lambda < 1$. There exists a simple word $\gamma$ such that $\gamma$ contains disjoint $\eps$-quasi-loops that occupy at least a proportion $\lambda$ of $\gamma$.
		\end{Lemma}
		
		\begin{proof}
			We refer the reader to the proof of Lemma 5.16  in \cite{schlich_equivalence_2024}, using here Lemma \ref{decoupe-restes-c-sph} instead of Lemma 5.15 in \cite{schlich_equivalence_2024}.\\
			The idea of the proof is to use Lemma \ref{decoupe-restes-c-sph} recursively : first apply Lemma \ref{decoupe-restes-c-sph} to $u=\gamma$ to find some quasi-loops, then apply again Lemma \ref{decoupe-restes-c-sph} to the subwords of $\gamma$ which are not already known to be quasi-loop, so that we cover these remainders by a uniform proportion of new quasi-loops, and then repeat the procedure until we found as many quasi-loops as required. 
		\end{proof}
		
		Finally we show that if such an element exists, we obtain an inequality on the displacement of the basepoint $o$. 
		\begin{Lemma} \label{contrad-sph}
			Let $0 < \varepsilon < \frac{1}{C}$ and $1 - \frac{1}{C'}(\frac{1}{C}-\varepsilon) <  \lambda < 1$. Let $\gamma$ be a simple word in $\mathbb{F}_2$ which contains $\varepsilon$-quasi-loops which occupy at least a proportion $\lambda$ of $\gamma$. Then $$d(\rho(\gamma)o,o) < \frac{1}{C} |\gamma|. $$  
		\end{Lemma}
		
		\begin{proof}
		For details, we refer the reader to the proof of Lemma 5.17 in \cite{schlich_equivalence_2024}.
		\end{proof}
		
		And thus this inequality contradicts the Bowditch hypothesis \eqref{BQ-constant-sph}, so we found a contradiction. Proposition \ref{BQ=>CTU-Sph} is proved. 
	\end{proof}
	
	\section[Simple-stability of Bowditch representations of $\groupe$]{Simple-stability of Bowditch representations of $\groupe$}

	\sectionmark{Towards simple-stability}
	\label{second-step-proof-sph}
	
	Before starting the last step of the proof of Theorem \ref{sph-BQ=PS}, we need to state a few lemmas on continuous map. The goal is to prove Lemma \ref{eps-slope}, which is in the same flavour as Lemma \ref{temps-excursion}, and which will be used in section \ref{end-proof}.
	
	\subsection{A few lemmas on continuous maps}
	
	\begin{Lemma} \label{slope-ferme}
		Let $x<y$ be two reals and $f:[x,y] \to \R$ be a continuous map. Let $\eps >0$. Denote $l=|y-x|$. Let us define :
		\begin{equation*}
			L_{f,\eps}=\{l' \in [0,l] \mid \exists (x',y') \in [x,y]^2, \, \, l'= |y'-x'| \text{ and } |f(y')-f(x')| \leq \eps |y'-x'|\}
		\end{equation*}
		Then the set $L_{f,\eps}$ is a closed subset of $[0,l]$.    
	\end{Lemma}
	
	\begin{proof}
		Let $(l_n)_{n \in \N}$ be a sequence of $L_{f,\eps}$ such that $l_n \to l_\infty \in [0,l]$. Let $(x_n)_{n\in \N}$ and  $(y_n)_{n\in \N}$ be two sequences of $[x,y]$ such that $|f(y_n)-f(x_n)|\leq \eps(y_n-x_n)$ and $y_n-x_n=l_n$. Up to subsequence, since $[x,y]$ is compact, we can assume that $x_n \to x_\infty \in [x,y]$ and $y_n \to y_\infty \in [x,y]$. Then by continuity, we obtain $|f(y_\infty)-f(x_\infty)| \leq \eps(y_\infty-x_\infty)$ and thus $l_\infty \in L_{f,\eps}$.
	\end{proof}
	
	\begin{Lemma} \label{exc-plus-petite}
		Let $g : [x,y] \to \R$ a continuous map such that $g(x)=g(y)$ and $g(z)\geq g(x)$ for all $z \in [x,y]$. Denote $l=\vert x-y \vert$. Then there exists two reals $x',y' \in [x,y]$ such that $g(x')=g(y')$, $g(z)\geq g(x')$ for all $z \in [x',y']$ and $ \displaystyle \frac{l}{2} \leq \vert x'-y' \vert < l$.
	\end{Lemma}
	
	\begin{proof}
	This lemma and its proof can be found in Lemma 5.8 in \cite{schlich_equivalence_2024}. 
	\end{proof}
	
	\begin{Lemma} \label{slope-plus-petite}
		Let $x<y$ be two reals and $f:[x,y] \to \R$ be a continuous map. Let $\eps >0$ such that $|f(y)-f(x)| \leq \eps |y-x|$. Denote $l=|y-x|$.
		There exists two reals $x',y'\in [x,y]$ such that $\displaystyle \frac{l}{2}\leq |y'-x'| <l$ and $|f(y')-f(x')|\leq \eps |y'-x'|$.
	\end{Lemma}
	\begin{proof}
		Let us define $\displaystyle s=\frac{f(y)-f(x)}{y-x}$. Then by hypothesis on $f$ we have $|s| \leq \eps$. 
		\begin{itemize}
			\item Suppose that there exists $x<z<y$ such that $f(z)=f(x)+s(z-x)$. Then both $f(z)-f(x)=s(z-x)$ and $f(y)-f(z)=s(y-z)$. Moreover either $ \displaystyle z-x \geq \frac{y-x}{2}=\frac{l}{2}$ or $\displaystyle y-z \geq \frac{y-x}{2}=\frac{l}{2}$, and thus the claim is proved by setting $l'=z-x$ or $l'=y-z$.
			\item Now suppose that for all $x<z<y$, $f(z) \neq f(x)+s(z-x)$. Then by continuity, we have that for all $x \leq z \leq y$, $f(z) \geq f(x)+s(z-x)$ or for all $x \leq z \leq y$, $f(z) \leq f(x)+s(z-x)$. We set $g(z)=f(z)-s(z-x)$. Thus either the map $g$ or the map $-g$ satisfies the hypothesis of Lemma \ref{exc-plus-petite}. From this lemma we deduce the existence of $x'<y'$ such that $\displaystyle \frac{l}{2}\leq |x'-y'|<l$ and $g(x')=g(y')$, that is $f(y')-f(x')=s(y'-x')$, hence $\vert f(y')-f(x')\vert  \leq \eps \vert y'-x'\vert $.
		\end{itemize}
	\end{proof}
	
	\begin{Lemma} \label{eps-slope}
		Let $x<y$ be two reals and $f:[x,y] \to \R$ a continuous map. Let $\eps >0$ and assume that $|f(x)-f(y)|\leq \eps |x-y|$. \\
		Denote $l=|x-y|$. Then, for all $0<a<l$, there exists two reals $x'<y'$ such that $x \leq x'<y'\leq y$ satisfying :
		\begin{itemize}
			\item $\displaystyle \frac{a}{2} \leq |x'-y'| < a$
			\item $|f(x')-f(y')| \leq \eps |x'-y'|$
		\end{itemize}
	\end{Lemma}
	
	\begin{proof}
		Define $L_{f,\eps}$ as in Lemma \ref{slope-ferme} : 
		\begin{equation*}
			L_{f,\eps}=\{l' \in [0,l] \mid \exists (x',y') \in [x,y]^2, \, \, l'= |y'-x'| \text{ and } |f(y')-f(x')| \leq \eps |y'-x'|\}.
		\end{equation*}
		Then $L_{f,\eps}\cap [a,l]$ is closed (by Lemma \ref{slope-ferme}) and non-empty (because $l \in L_{f,\eps}\cap [a,l]$). Denote $l'=\min \, L_{f,\eps}\cap [a,l]$. By Lemma \ref{slope-plus-petite}, there exists $l'' \in L_{f,\eps}$ such that $\frac{l'}{2}\leq l'' <l'$. Then $l''<a$ because $l''<l'=\min \, L_{f,\eps}\cap [a,l]$ and $l'' \geq \frac{l'}{2}\geq \frac{a}{2}$. Hence the Lemma.
	\end{proof}
	
	\subsection{End of the proof of Theorem \ref{sph-BQ=PS}} \label{end-proof}
	Using the result of Proposition \ref{BQ=>CTU-Sph} in section \ref{sph-first-step}, we will now finish the proof of Theorem \ref{sph-BQ=PS}, which states that a Bowditch representation is simple-stable. 
	
	\begin{proof}[Proof of Theorem \ref{sph-BQ=PS}]
	Let $X$ be a $\delta$-hyperbolic, geodesic, visibility space, and $o \in X$ a basepoint. \\ 
	Pick once and for all a Bowditch representation $\rho$, with constants $(C,D)$. In Proposition \ref{BQ=>CTU-Sph} of section \ref{sph-first-step}, we prove the existence of a constant $K >0$ such that for all simple elements $\gamma \in \simple$, we have the inclusion $\tau_\rho(L_\gamma) \subset N_K(\mathrm{Axis}(\rho(\gamma)))$. (Recall that $L_\gamma$ is the axis of the element $\gamma$ in the Cayley graph of $\F_3=\groupe$ and $\mathrm{Axis}(\rho(\gamma))$ the axis of the hyperbolic isometry $\rho(\gamma)$ in $X$.) For every $\gamma$ in $\simple$, pick~$\ell_\gamma \subset X$ some geodesic joining the two attracting and repelling points of $\rho(\gamma)$, $\rho(\gamma)^+$ and $\rho(\gamma)^-$. Then~$\ell_\gamma \subset \mathrm{Axis}(\rho(\gamma))$ be definition of the axis (see beginning of section \ref{sph-first-step}). Moreover, using the stability of (quasi)-geodesic in $\delta$-hyperbolic spaces (see for example Theorem 3.1 in \cite{coornaert_geometrie_1990}), we deduce the existence of a constant $C(\delta)$ only depending on $\delta$ such that  $N_K(\mathrm{Axis}(\rho(\gamma))) \subset N_{K+C(\delta)}(\ell_\gamma)$. Then, noting $K_\delta=K+C(\delta)$, we obtain that for all simple word $\gamma \in \groupe$, we have $\tau_\rho(L_\gamma) \subset N_{K_\delta}(\ell_\gamma)$.   \\ 
	Let us choose a projection map $p_\gamma$ on $\ell_\gamma$, that is a map $p_\gamma : X \to \ell_\gamma$ such that for all $x \in X$, $p_\gamma(x) \in \ell_\gamma$ and for all $y \in \ell_\gamma$, we have $d(x,p_\gamma(x)) \leq d(x,y)$. For a point $p$ on the geodesic $\ell_\gamma$, we define the real  $H_\gamma(p)=\pm d(p,p_\gamma(o))$. The sign plus or minus is determined according to which side of $p(o)$ the point $p$ is located on. Thus $H_\gamma$ is an isometry between $\ell_\gamma$ and $\R$ sending $p(o)$ to 0. We also define $E_\gamma : L_\gamma \to \R$ such that for all $x \in L_\gamma$, we have $E_\gamma(x)=H_\gamma(p_\gamma(\tau_\rho(x)))$. 
	
	\begin{Lemma} \label{lien-E-d}
		Let $\gamma \in \simple$ and $x,y \in L_\gamma$. Then we have :
		\begin{equation*}
			d(\tau_\rho(x),\tau_\rho(y))-2K_\delta \leq |E_\gamma(x)-E_\gamma(y)|\leq d(\tau_\rho(x),\tau_\rho(y))+2K_\delta
		\end{equation*}
	\end{Lemma}
	
	\begin{proof}
		We have $|E_\gamma(x)-E_\gamma(y)|=d(p_\gamma(\tau_\rho(x)),p_\gamma(\tau_\rho(y)))$, then, by the triangle inequality :
		
		\begin{align*}
			|E_\gamma(x)-E_\gamma(y)|& \leq d(p_\gamma(\tau_\rho(x)),\tau_\rho(x))+d(\tau_\rho(x),\tau_\rho(y))+d(\tau_\rho(y),p_\gamma(\tau_\rho(y))) \\
			& \leq 2K_\delta + d(\tau_\rho(x),\tau_\rho(y)) \qquad \text{ because } \tau_\rho(L_\gamma) \subset N_{K_\delta}(\ell_\gamma)\\
			\intertext{and on the other hand :}
			d(\tau_\rho(x),\tau_\rho(y)) & \leq d(\tau_\rho(x),p_\gamma(\tau_\rho(x)))+d(p_\gamma(\tau_\rho(x)),p_\gamma(\tau_\rho(y)))+d(\tau_\rho(y),p_\gamma(\tau_\rho(y))) \\ 
			& \leq 2K_\delta+|E_\gamma(x)-E_\gamma(y)| \qquad \text{ because } \tau_\rho(L_\gamma) \subset N_{K_\delta}(\ell_\gamma)
		\end{align*}
	\end{proof}
	
	Note that $E_\gamma$ is not necessarily continuous. (This comes from the fact that the projection map $p_\gamma$ is itself not necessarily continuous. Indeed, such a projection may not be unique). Since we will need to use the results of the previous section on continuous map (and in particular Lemma \ref{eps-slope}), we are going to consider a continuous approximation of $E_\gamma$. Thus we define $\t{E}_\gamma$ to be the map from the Cayley geodesic $L_\gamma$ to $\R$ such that, on every integer point $x \in L_\gamma$, $\t{E}_\gamma(x)=E_\gamma(x)$ and between two integer points, we do a linear interpolation, that is : $\t{E}_\gamma(x)=tE_\gamma(\lf x \rf)+(1-t)E_\gamma(\lc x \rc)$, with $t \in [0,1]$ such that $d(\lf x \rf,x)=1-t$ and $d(\lc x \rc,x)=t$. The map $\t{E}$ is continuous. The next Lemma aims to compare $E_\gamma$ and $\t{E}_\gamma$. 
	
	\begin{Lemma} \label{lien-E-tE}
		For all $x,y \in L_\gamma$, we have : 
		\begin{equation*}
			|E_\gamma(x)-E_\gamma(y)|-(4C'+8K_\delta) \leq |\t{E}_\gamma(x)-\t{E}_\gamma(y)|\leq |E_\gamma(x)-E_\gamma(y)|+4C'+8K_\delta
		\end{equation*}
	\end{Lemma}
	
	\begin{proof}
		\begin{itemize}
			\item We first prove that for all $x \in L_\gamma$ :
			\begin{equation} \label{E-x-floorx}
				|E_\gamma(x)-E_\gamma(\lf x \rf)| \leq C'+2K_\delta
			\end{equation}
			Indeed, by Lemma \ref{lien-E-d}, we have $|E_\gamma(x)-E_\gamma(\lf x \rf )|\leq d(\tau_\rho(x),\tau_\rho(\lf x \rf))+2K_\delta$. But $d(x,\lf x \rf) \leq 1$ and $\tau_\rho$ is $C'$-Lipschitz, hence the inequality is true. \\ 
			
			\item We secondly prove that for all $x \in L_\gamma$ :
			
			\begin{equation}\label{tE-x-floorx}
				|\t{E}_\gamma(x)-\t{E}_\gamma(\lf x \rf)| \leq C'+2K_\delta
			\end{equation}
			Indeed :
			\begin{align*}
				|\t{E}_\gamma(x)-\t{E}_\gamma(\lf x \rf)| & =|tE_\gamma(\lf x \rf)+(1-t)E_\gamma(\lc x \rc)-E_\gamma(\lf x \rf)| \text{ by definition of $\t{E}_\gamma$}  \\
				& = |1-t||E_\gamma(\lf x \rf)-E_\gamma(\lc x \rc) | \\ 
				& \leq d(\tau_\rho(\lf x\rf),\tau_\rho (\lc x \rc))+2K_\delta \qquad \text{by Lemma \ref{lien-E-d}}\\
				& \leq C'+2K_\delta \qquad \text{because } d(\lf x \rf, \lc x \rc)\leq 1 
			\end{align*}
			
			\item We deduce that for all $x,y \in L_\gamma$, we have :
			\begin{equation} \label{Exy,floorxy}
				|E_\gamma(\lf x \rf )-E_\gamma(\lf y \rf)| - (2C'+4K_\delta) \leq  |E_\gamma(x)-E_\gamma(y)| \leq 2C'+4K_\delta+|E_\gamma(\lf x \rf )-E_\gamma(\lf y \rf)|
			\end{equation}
			Indeed, by the triangle inequality we have :
			\begin{align*}
				|E_\gamma(x)-E_\gamma(y)| & \leq |E_\gamma(x)-E_\gamma(\lf x \rf )|+|E_\gamma(\lf x \rf )-E_\gamma(\lf y \rf)|+|E_\gamma(\lf y \rf )-E_\gamma(y)| \\
				& \leq 2C'+4K_\delta+|E_\gamma(\lf x \rf )-E_\gamma(\lf y \rf)| \text{ by \eqref{E-x-floorx}}
				\intertext{and on the other hand}
				|E_\gamma(\lf x \rf )-E_\gamma(\lf y \rf)|&\leq    |E_\gamma(x)-E_\gamma(\lf x \rf )|+|E_\gamma(x)-E_\gamma(y)|+|E_\gamma(\lf y \rf )-E_\gamma(y)| \\ 
				& \leq 2C'+4K_\delta+|E_\gamma(x)-E_\gamma(y)| \qquad \text{ again by \eqref{E-x-floorx}}
			\end{align*}
			
			\item Similarly, we deduce using \eqref{tE-x-floorx} and the fact that $\t{E}_\gamma(\lf x \rf)=E_\gamma(\lf x \rf), \t{E}_\gamma(\lf y \rf)=E_\gamma((\lf y \rf)$, that for all $x,y \in L_\gamma$, we have : 
			\begin{equation} \label{tExy,floorxy}
				|E_\gamma(\lf x \rf )-E_\gamma(\lf y \rf)| - (2C'+4K_\delta) \leq  |\t{E}_\gamma(x)-\t{E}_\gamma(y)| \leq 2C'+4K_\delta+|E_\gamma(\lf x \rf )-E_\gamma(\lf y \rf)|
			\end{equation}
			
			\item Finally, the desired inequality follows from \eqref{Exy,floorxy} and \eqref{tExy,floorxy}.
		\end{itemize}
	\end{proof}

	Let us now prove that $\rho$ is simple-stable. By contradiction, suppose that it is not. Then for all $n \in \N$, we can find a cyclically reduced simple element $\gamma_n \in \simple$ together with two integer points $x_n$ and $y_n$ on $L_{\gamma_n}$ such that \begin{equation} \label{eq-QB-sph}
		d(\tau_\rho(x_n),\tau_\rho(y_n)) \leq \frac{1}{n}d(x_n,y_n)-1 
	\end{equation}
	In particular, we have that $d(x_n,y_n) \geq n$. We can make the assumption that the elements $\gamma_n$ are pairwise distinct. Indeed, if the sequence ${(\gamma_n)}_n$ only takes finitely many values, then, up to subsequence, we can suppose that $\gamma_n=\gamma$ for some simple word $\gamma$. But $\rho(\gamma)$ is a hyperbolic isometry so there exist two constants $C_\gamma$ and $D_\gamma$ (depending on $\gamma$) such that $\tau_\rho(L_\gamma)$ is a $(C_\gamma,D_\gamma)$-quasi-geodesic. Then, since $x_n$ and $y_n$ belong to $L_\gamma$, we have : 
	\begin{align*}
		\frac{1}{C_\gamma}d(x_n,y_n)-D_\gamma & \leq d(\rho(x_n)o,\rho(y_n)o) \leq \frac{1}{n}d(x_n,y_n)-1 \\
		\text{so} \qquad \frac{1}{C_\gamma}-\frac{D_\gamma}{d(x_n,y_n)} & \leq \frac{1}{n}-\frac{1}{d(x_n,y_n)}, \nonumber \\
		\text{then, taking the limit when $n \to \infty$,} \quad \frac{1}{C_\gamma} & \leq 0, \quad \text{which is absurd}. \nonumber
	\end{align*}
	Thus we can suppose that the elements $\gamma_n$ are pairwise distinct and therefore $\Vert \gamma_n \Vert  \to \infty$.\\
	 Denote by $[N_1(\gamma_n),\hdots,N_{r(\gamma_n)}(\gamma_n)]$ the continued fraction expansion of the slope of $\gamma_n$. As in the proof of the previous section (section \ref{sph-first-step}), we can prove the following lemma. 
	
	\begin{Lemma} \label{sph-Ni-borné}
		For all $i \in \N^*$, there exists a constant $N_i >0$ such that for all $n \in \N^*$, whenever $N_i(\gamma_n)$ is well defined (that is $r(\gamma_n) \geq i$), we have $N_i(\gamma_n) \leq N_i$. Moreover, up to subsequence, $r(\gamma_n) \to \infty$. 
	\end{Lemma}
	
	\begin{proof}
		The proof is the same as the proof of Lemma \ref{forme-frac-continue-sph}, and the contradiction is this time on inequality \eqref{eq-QB-sph} (in both Lemmas, the sequence $(\gamma_n)_{n\in \N}$ is chosen in order to contradict simple-stability).
	\end{proof}
	
	As in section \ref{bornes-li-sph}, we deduce for all integer $i$ the existence of constants $L_i>0$ such that :
	\begin{equation} \label{L_i-final}
		\forall n\in \mathbb{N}, \forall 0 \leq i \leq r(n), \qquad i \leq l_i(\gamma_n) \leq L_i
	\end{equation}
	
	Then, as in the proof of the existence of uniform tubular neighborhoods (Proposition \ref{BQ=>CTU-Sph}), we are able to find a lot of quasi-loops inside a sufficiently large subword of $\gamma$ :
	
	\begin{Lemma} \label{decoupe-restes-final-sph} Let $\beta=\frac{1}{168}(=\frac{\alpha}{8}$, where $\alpha$ is the constant introduced in Theorem \ref{magic-len-S_{0,4}}). \\
		Let $\displaystyle 0<\varepsilon < \frac{1}{C}$ and $r >0$. There exists a constant $R>0$ and an integer $n_0 \in \N$, such that, given any integer $n \geq n_0$ and  any subword $u$ of $\gamma_n$ such that $|u| \geq R$, there exists a positive integer $q \in \N^*$, a subset $QL \subset \{1,\hdots,q\}$ and $q$ words $u_1, \hdots,u_q \in \F_2$ such that :
		\begin{enumerate}
			\item $u=u_1\hdots u_q$
			\item For all $k \in QL$, $u_k$ is an $\varepsilon$-quasi-loop
			\item $\displaystyle \sum_{k \in QL} |u_k| \geq \beta|u|$
			\item For all $k \notin QL$, $|u_k| \geq  r$
		\end{enumerate}
	\end{Lemma}
	
	The proof of this lemma is similar to the proof of Lemma \ref{decoupe-restes-c-sph}, except that is this context we find quasi-loops directly using our hypothesis on the sequence $(\gamma_n)_{n \in \N}$ (equation \eqref{eq-QB-sph}), and that this time we need to consider the map $E_{\gamma_n}$ and its continuous approximation $\t{E}_{\gamma_n}$ to control the length of the quasi-loops we consider. 
	\begin{proof}
		We start by fixing an integer $i$ such that $i\geq \max(10,2r+7)$. Then, we choose an integer $n_0$ satisfying $\displaystyle n_0 \geq \max \left( \frac{4}{\eps},\frac{8K_\delta}{\eps},\frac{2}{\eps}(6C'+10K_\delta)+2,L_i \right)$ (recall that the constant $L_i$ are defined in \ref{L_i-final}). Finally, let $\displaystyle R \geq \max\left(\frac{3L_i}{\alpha}, 9L_i\right)$. Now we fix $n\geq n_0$ and $u$ a subword of $\gamma_n$ such that $|u| \geq R$. \\ 
		
		The first step is to find an $\eps$-quasi-loop of length comprise between $\displaystyle \frac{l_i(\gamma_n)-7}{2}$ and $l_i(\gamma_n)-5$.\\
		First observe that, by \eqref{eq-QB-sph}, and since $n \geq n_0 \geq \frac{4}{\eps}$, we have :
		\begin{equation} \label{qy-QB}
			d(\rho(x_n)o,\rho(y_n)o) \leq \frac{\eps}{4} d(x_n,y_n)
		\end{equation}
		so we deduce that :
		\begin{align*}
			|E_{\gamma_n}(x_n)-E_{\gamma_n}(y_n)|& \leq 2K_\delta + d(\rho(x_n)o,\rho(y_n)o) \qquad \text{ by inequality \eqref{lien-E-d}} \\
			& \leq 2K_\delta + \frac{\eps}{4} d(x_n,y_n) \qquad \text{ by inequality \eqref{qy-QB}} \\
			\intertext{ but on another hand we have }  
			d(x_n,y_n) & \geq n \geq n_0 \geq \frac{8K_\delta}{\eps} \qquad \text{ by our initial choice on $n_0$}, \\
			\text{hence } \qquad |E_{\gamma_n}(x_n)-E_{\gamma_n}(y_n)| & \leq \frac{\eps}{2} d(x_n,y_n).
		\end{align*}
		Recall that $x_n$ and $y_n$ were taken to be integer points in $L_{\gamma_n}$. So $\t{E}_{\gamma_n}(x_n)=E_{\gamma_n}(x_n)$ and $\t{E}_{\gamma_n}(y_n)=E_{\gamma_n}(y_n)$, therefore :
		\begin{equation*}
			|\t{E}_{\gamma_n}(x_n)-\t{E}_{\gamma_n}(y_n)| \leq \frac{\eps}{2} d(x_n,y_n).
		\end{equation*}
		Now notice that we have $0<l_i(\gamma_n)-5 \leq L_i-5<n_0 \leq n \leq d(x_n,y_n)$, 
		so, by continuity of $\t{E}_{\gamma_n}$, we can use Lemma \ref{eps-slope} to find $x'_n,y'_n \in L_{\gamma_n}$ such that the two following hold :
		\begin{align}
			\frac{l_i(\gamma_n)-5}{2} \leq d(x'_n,y'_n) < l_i(\gamma_n)-5 \label{encadr-xy},\\
			|\t{E}_{\gamma_n}(x'_n)-\t{E}_{\gamma_n}(y'_n)| \leq \frac{\eps}{2} d(x'_n,y'_n) \label{Ex'}.
		\end{align}
		Now let $v=\lf x'_n \rf ^{-1} \lf y'_n \rf $.     
		
		Let us show that $v$ is an $\eps$-quasi-loop. 
		\begin{align*}
			d(\rho(v)o,o) & = d(\rho(\lf x'_n \rf)o,\rho(\lf y'_n \rf)o) \qquad \text{ by definition of $v$}\\ 
			& \leq  2C'+d(\rho(x'_n)o,\rho(y'_n)o) \qquad \text{ since $d(x'_n,\lf x'_n \rf)<1,d(y'_n,\lf y'_n \rf)<1 $}\\ 
			& \leq 2C'+2K_\delta + |E_{\gamma_n}(x'_n)-E_{\gamma_n}(y'_n)| \qquad \text{ by inequality \eqref{lien-E-d}} \\ 
			& \leq 2C'+2K_\delta + |\t{E}_{\gamma_n}(x'_n)-\t{E}_{\gamma_n}(y'_n)| + 4C'+8K_\delta  \text{ by Lemma \eqref{lien-E-tE}} \\
			& \leq 6C' + 10K_\delta + \frac{\eps}{2} d(x'_n,y'_n) \qquad \text{ by \eqref{Ex'}} \\ 
			& \leq 6C'+10K_\delta+\frac{\eps}{2} (|v|+1) \\ 
			& \leq 6C'+10K_\delta+\frac{\eps}{2}+\frac{\eps}{2}|v| 
		\end{align*}
		
		But on an other hand, since $\displaystyle n_0 \geq \frac{2}{\eps}(6C'+10K_\delta)+2$, we have $
		|v|=d(\lf x'_n \rf,\lf y'_n \rf) \geq d(x'_n,y'_n)-1 \geq n-1 \geq n_0-1 \geq \frac{2}{\eps}(6C'+10K_\delta)+1$, and then 
		$d(\rho(v)o,o) \leq \eps|v|$, which ensures that $v$ is an $\eps$-quasi-loop. \\ 
		
		Note that we have the following inequality on the length of $v$ :
		\begin{align*}
			d(x'_n,y'_n)-1&<|v|<d(x'_n,y'_n)+1 \\ 
			\text{ so } \qquad \frac{l_i(\gamma_n)-7}{2}&< |v| <l_i(\gamma_n)-4 \qquad \text{ by \eqref{encadr-xy}} \\ 
			\text{then } \qquad \frac{l_i(\gamma_n)-7}{2} &\leq |v| \leq l_i(\gamma_n)-5 \qquad \text{ since $|v|$ is an integer}
		\end{align*}
		
		The rest of the proof is exactly the same as for Lemma \ref{decoupe-restes-c-sph}. We consider $w$ the subword of $\gamma_n$ of length $l_i(\gamma_n)-5$ such that $v$ is a prefix of $w$. Then we can apply Theorem \ref{magic-len-S_{0,4}} to $u$ and $w$ (since we have imposed $R \geq 9L_i$) and we verify the four points of Lemma \ref{decoupe-restes-final-sph}. Indeed :
		
		\begin{enumerate}
			\item comes directly from Theorem \ref{magic-len-S_{0,4}}.
			\item is true since we proved that $v$ is an $\eps$-quasi-loop. 
			\item is true since we fixed $R$ such that $R \geq \frac{3L_i}{\alpha}$.
			\item is true since we fixed $i$ such that $i \geq 2r+7$.
		\end{enumerate}
	\end{proof}
	
	Once we have Lemma \ref{decoupe-restes-final-sph}, the procedure for the end of the proof of simple-stability is the same as for Proposition \ref{BQ=>CTU-Sph}. For completeness, we recall the last steps. \\ 
	
	We are now able to find a simple element $\gamma$ (from the sequence $(\gamma_n)_{n\in \N}$) which contains a very large proportion of quasi-loops.
	
	\begin{Lemma} \label{trouve-gamma-final}
		Let $0 < \varepsilon < \frac{1}{C}$ and $1 - \frac{1}{C'}(\frac{1}{C}-\varepsilon) <  \lambda < 1$. There exists a simple element $\gamma$ such that $\gamma$ contains  $\eps$-quasi-loops that occupy at least a proportion $\lambda$ of $\gamma$.
	\end{Lemma}
	
	\begin{proof}
		The details of the proof can be found in Lemma 5.16 in \cite{schlich_equivalence_2024}, using now Lemma \ref{decoupe-restes-final-sph} instead of Lemma 5.15 in \cite{schlich_equivalence_2024} ; an idea of the proof is given in Lemma \ref{trouve-gamma-sph}.
	\end{proof}
	
	The existence of such an element implies an inequality on the displacement of the basepoint $o$. 
	\begin{Lemma} \label{contrad-final}
		Let $0 < \varepsilon < \frac{1}{C}$ and $1 - \frac{1}{C'}(\frac{1}{C}-\varepsilon) <  \lambda < 1$. Let $\gamma$ be a simple word of $\pi_1(\sph)$ which contains $\varepsilon$-quasi-loops which occupy at least a proportion $\lambda$ of $\gamma$. Then $$d(\rho(\gamma)o,o) < \frac{1}{C} |\gamma|. $$  
	\end{Lemma}
	
	\begin{proof}
		The details are in the proof of Lemma 5.17 in \cite{schlich_equivalence_2024}.
	\end{proof}
	
	This last inequality contradicts the Bowditch hypothesis, and then Theorem \ref{sph-BQ=PS} is proved. 
	
	\end{proof}


	\printbibliography  
	
\end{document}